\documentclass{amsart}
\usepackage[utf8]{inputenc}
\usepackage{amsmath, amssymb, amsfonts, amsthm}
\usepackage{enumerate}
\usepackage{bbold}
\usepackage{color}
\usepackage{graphicx}
\usepackage{fullpage}
\usepackage{tikz}
\usepackage[boxsize = .3em]{ytableau}
\usepackage{blkarray, bigstrut} %
\usepackage[shortlabels]{enumitem}

\usepackage{tikz-cd}
\usepackage{hyperref}
\usepackage{lipsum} 
\usepackage{comment}
\usepackage{bbm}

\allowdisplaybreaks
\DeclareMathOperator{\Rep}{\textrm{Rep}}
\DeclareMathOperator{\C}{\mathbb{C}}
\DeclareMathOperator{\Z}{\mathbb{Z}}

\DeclareMathOperator{\End}{\textrm{End}}
\DeclareMathOperator{\N}{\mathcal{N}}
\DeclareMathOperator{\id}{\textrm{id}}

\DeclareMathOperator{\Hom}{\textrm{Hom}}

\newcommand{\q}[1]{[#1]_q}

\newcommand{\att}[2]{\raisebox{-.5\height}{ \includegraphics[scale = #2]{diagrams/#1.pdf}}}
\usepackage{array}   
\newcolumntype{L}{>{$}l<{$}}

\def\End{{\rm End}}
\def\hbeta{\hat\beta}
\def\al{\alpha}
\def\la{\lambda}
\def\ga{\gamma}
\def\om{\omega}
\def\Gn0{\bar G_n[0]}
\def\Ga{\Gamma}
\def\La{\Lambda}
\def\C{{\mathbb C}}

\def\N{{\mathbb N}}
\def\Z{{\mathbb Z}}
\def\S{{\mathcal S}}
\def\B{{\mathcal B}}
\def\E{{\mathcal E}}
\def\Ca{{\mathcal C}}

\def\q{{\mathfrak q}}
\def\ep{\varepsilon}
\def\Slnk{\operatorname{Paths}(\lambda,n)}
\def\Slk{\S_\la^{(n)}}
\def\Galnk{\tilde\Gamma(N)}
\def\Glnk{\Gamma(N)}
\def\Vlnk{V_\la^{(n)}}
\def\Ulnk{U_\la^{(n)}}
\def\Ul{U_\nu}
\def\llan{\ell(\la)^{(n)}}
\def\lmun{\ell(\mu)^{(n+1)}}
\def\Gno{G_n[0]}
\def\Gnoo{G_{n+1}[0]}
\def\bGn{\bar G_n[0]}

\theoremstyle{plain}
\newtheorem{thm}{Theorem}[section]
\newtheorem{lem}[thm]{Lemma}
\newtheorem{prop}[thm]{Proposition}
\newtheorem{cor}[thm]{Corollary}

\theoremstyle{definition}
\newtheorem{defn}[thm]{Definition}

\newtheorem{remark}[thm]{Remark}

\title{Hecke-Clifford algebras at roots of unity and conformal embeddings}
\author{Cain Edie-Michell and Hans Wenzl}
\address{Cain Edie-Michell\\
University of New Hampshire\\
Durham, 
New Hampshire}
\email{cain.edie-michell@unh.edu}
\address{Hans Wenzl\\
University of California at San Diego\\
La Jolla,
California}
\email{hwenzl@ucsd.edu}
\date{}

\begin{document}

\maketitle
\begin{abstract}

In this paper we give a combinatorial description of the Cauchy completion of the categories $\mathcal{E}_q$ and $\overline{\mathcal{SE}_N}$ recently introduced by the first author and Snyder. This in turns gives a combinatorial description of the categories $\overline{\Rep(U_q(\mathfrak{sl}_N))}_{A}$ where $A$ is the \textit{\`etale} algebra object corresponding to the conformal embedding $\mathfrak{sl}_N$ level $N$ into $\mathfrak{so}_{N^2-1}$ level 1. In particular we give a classification of the simple objects of these categories, a formula for their quantum dimensions, and fusion rules for tensoring with the defining object. Our method of obtaining these results is the Schur-Weyl approach of studying the representation theory of certain endomorphism algebras in $\mathcal{E}_q$ and $\mathcal{SE}_N$, which are known to be subalgebras of Hecke-Clifford algebras. We build on existing literature to study the representation theory of the Hecke-Clifford algebras at roots of unity. 

\end{abstract}

\section{Introduction}
Given a braided fusion category $\Ca$ and an \textit{\' etale} algebra object $A$, it is well-known that we obtain a new tensor category $\Ca_A$ of $A$-module objects internal to $\mathcal{C}$ \cite[Chapter 8.8]{Book}. A large source of \textit{\`etale} algebras for fusion categories related to a Lie algebra $\mathfrak{g}$ comes from conformal inclusions of Lie algebras $\mathfrak{g}\subset\mathfrak{h}$. While a complete list of conformal embeddings is known \cite{Embeddings1,Embeddings2,LagrangeUkraine}, the explicit structure of the corresponding categories $\Ca_A$ is only known for a handful of small examples \cite{Xu}, and for the family $\mathfrak{sl}_N \subset \mathfrak{sl}_{N(N\pm 1)/2}$ \cite{LiuYB,LiuRing}.

In recent work of the first author and Snyder \cite{ConformalA}, a generators and relations presentation of a discrete family of categories $\mathcal{SE}_N$ was given. It was shown that
\[         \operatorname{Ab}(\overline{  \mathcal{SE}_N}) \simeq    \overline{\Rep(U_q(\mathfrak{sl}_N))}_{A}  \]
where $A$ is the \textit{\`etale} algebra
coming from the conformal embedding $\mathfrak{sl}_N \subset  \mathfrak{so}_{N^2-1}$, and $\operatorname{Ab}$ denotes the Cauchy completion. They also define a continuous family of tensor categories $\mathcal{E}_q$ which interpolates the categories $\mathcal{SE}_N$. Their methods for proving the above equivalence were non-constructive, and hence the structure of the categories $\operatorname{Ab}(\overline{  \mathcal{SE}_N})$ and $\operatorname{Ab}(  \mathcal{E}_q)$ was not obtained.

The purpose of this paper is to obtain the structure of the categories $\operatorname{Ab}(\overline{  \mathcal{SE}_N})$ and $\operatorname{Ab}( \mathcal{E}_q)$ (for $q$ in a dense subset of $\mathbb{C}$). In the case of $\operatorname{Ab}(\overline{  \mathcal{SE}_N})$, we will achieve this via a Schur-Weyl type approach by studying endomorphism algebras $\End_{\mathcal{SE}_N}(+^n)$, where $+$ is the distinguished generating object of $\mathcal{SE}_N$. As shown in \cite[Corollary 5.12]{ConformalA}, these algebras are distinguished subalgebras of the \textit{Hecke-Clifford} algebras. These algebras have appeared previously in the literature in the context of the quantum group $U_q(\mathfrak{q}_N)$ where $q_N$ is the isomeric (or queer) Lie super algebra \cite{Ol}. Here the Hecke-Clifford algebras appear as centraliser algebras of the vector representation of $U_q(q_N)$. This surprising connection between $U_q(q_N)$ and the conformal embedding $\mathfrak{sl}_N \subseteq \mathfrak{so}_{N^2-1}$ will be a key tool in this paper.

Our Schur-Weyl type approach will involve studying the representation theory of the Hecke-Clifford algebras, and of the subalgebras occurring as centraliser algebras of the categories $\mathcal{SE}_N$. The representation theory of these algebras has already been studied in \cite{JN} over a finite algebraic extension of $\mathbb{C}(q)$. One of the main contributions of this paper is to extend these results to a specialised $q\in \mathbb{C}$. In the case where $q$ is a primitive $4N$-th root of unity (which corresponds to the centraliser algebras of the categories $\mathcal{SE}_N$) these algebras are non-semisimple. This makes the representation significantly more complicated, and a large portion of these paper is devoted to studying the Hecke-Clifford algebras at these roots of unity. We anticipate these results will be useful in the study of $U_q(\mathfrak{q}_N)$ at $q$ a root of unity.

Our main results determine the simple objects, their dimensions, and the fundamental fusion rules for $\operatorname{Ab}(\overline{\mathcal{SE}_{N}})$ and $\operatorname{Ab}(\mathcal{E}_q)$ for generic $q\in \mathbb{C}$. By \cite[Theorem 1.5]{ConformalA} this gives the same structural results for the categories $\overline{\operatorname{Rep}(U_q(\mathfrak{sl}_N))}_A$ for all $N$. 
While the categories $\overline{\operatorname{Rep}(U_q(\mathfrak{sl}_N))}_A$ were known to exist before this paper, very little was known about their structure. The special case of $N=4$ was worked out in \cite[Example 6]{Xu}. We note that the objects of $\overline{\operatorname{Rep}(U_q(\mathfrak{sl}_N))}_A$ correspond to modules of the vertex operator algebra $\mathcal{V}(\mathfrak{sl}_N,N)$ which have a compatibility with the extended VOA $\mathcal{V}(\mathfrak{so}_{N^2-1},1)$ \cite[Remarks 2.7 and 2.8]{mcrae}, and so our results have implications in conformal field theory. 

The following definitions are required to state the theorem.

\begin{defn}
    Let $N \in \mathbb{N}\cup \infty$. We will write $Y^<_N$ for the set of strictly decreasing Young diagrams $\lambda$ with $\lambda_1 < N$. For $\lambda, \mu \in Y^<_N$ we will say $\lambda \to^{(N)} \mu$ if $\mu$ can be obtained from $\lambda$ by adding a single box, and removing rows with $N-1$ boxes. For a Young diagram $\lambda$ with $\ell(\lambda)$ parts we define the rational function
    \[q_\la\ =\ \prod _{j=1}^{\ell(\lambda)} q_{\la_j}\ \prod_{i<j} \frac{[\la_i-\la_j]}{[\la_i+\la_j]}, \quad \text{where} \; q_m\ :=\ \prod_{j=1}^m i\frac{ q^{j-1} + q^{1-j}}{q^j-q^{-j}}. \] 
    For a pair $\lambda, \mu \in Y_N^<$ we define the rational function
    \[q_{(\la,\mu)}\ :=\ q_\la q_\mu \prod_{(r,s)}\frac{q^{\la_r+\mu_s}+q^{-\la_r-\mu_s}}{q^{\la_r-\mu_s}+q^{\mu_s-\la_r}},\quad {\rm where}\ 1\leq r\leq \ell(\la),\ 1\leq s\leq \ell(\mu).\]
\end{defn}

We then show the following.
\begin{thm}\label{thm:mainHans}
    The isomorphism classes of simple objects in $\operatorname{Ab}(\overline{\mathcal{SE}_{N}})$ are parameterised by the set
   \[    \{  \lambda : \lambda\in Y^<_N :\:\ell(\lambda) \text{ odd}\} \cup  \{  (\lambda, \pm) :  \lambda\in Y^<_N :\:\ell(\lambda) \text{ even}\} .    \]
    We have the following decomposition formulae for the tensor product of the simple $\square$ with any other simple:
    \begin{align*} 
    \square \otimes \lambda &\cong \bigoplus_{  \lambda \to^{(N)} \mu ,\: \ell(\mu) \text{ even}     } (\mu, +) \oplus (\mu, -)\oplus  \bigoplus_{  \lambda \to^{(N)} \mu , \:\ell(\mu) \text{ odd}     }  \mu\oplus \mu \\    
    \square \otimes (\lambda,\pm) &\cong \bigoplus_{  \lambda \to^{(N)} \mu , \:\ell(\mu) \text{ even}     } (\mu, +) \oplus (\mu, -)\oplus  \bigoplus_{  \lambda \to^{(N)} \mu , \:\ell(\mu) \text{ odd}     }   \mu.
    \end{align*}
    Furthermore, the quantum dimensions of the simple objects are given by 
    \[
        d( \lambda ) = q_\lambda \cdot 2^{-\frac{\ell(\lambda)-1}{2}}, \qquad 
        d( \lambda,\pm ) = q_\lambda\cdot 2^{-\frac{\ell(\lambda)}{2}}\]
    where the rational functions $q_\lambda$ are evaluated at $q=e^{2\pi i \frac{1}{4N}}$.
\end{thm}

As an example, we have the following fusion graph for the action of $\square$ on the simples of $\operatorname{Ab}(\overline{\mathcal{SE}_5})$, i.e. the category coming from the conformal embedding $\mathcal{V}(\mathfrak{sl}_{5},5) \subset \mathcal{V}(\mathfrak{so}_{24},1)$. The orientations of the edges in this graph can be deduced from the vertex labeling. The graphs for larger $N$ can be easily deduced from the combinatorial formulae given in Theorem~\ref{thm:mainHans}.
\[\att{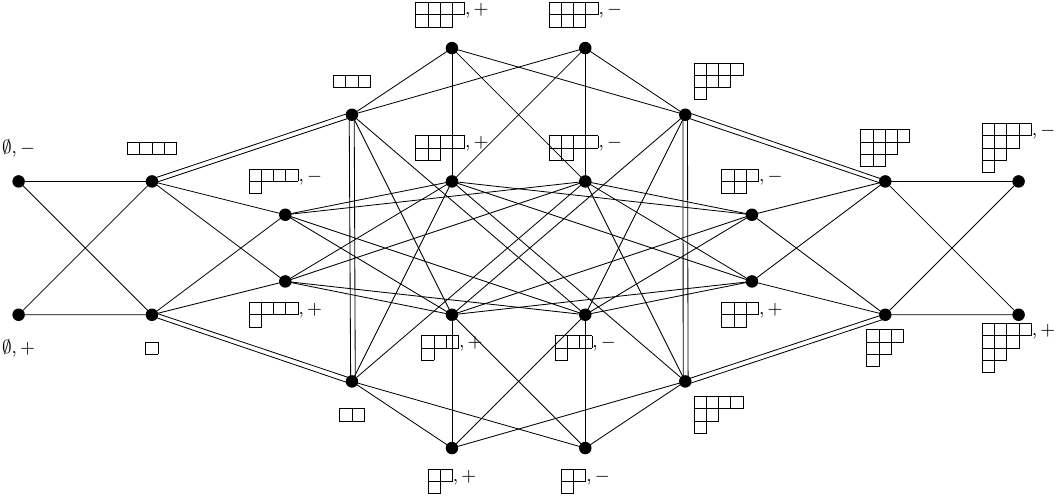}{.8}\]

We are able to find a labeling of the simple objects of $\operatorname{Ab}(\overline{\mathcal{SE}_N})$ which is stable in a certain sense as $N\to \infty$. Using this new labeling set, we are able to obtain a description of the categories $\operatorname{Ab}(\mathcal{E}_q)$ for nearly all values of $q\in \mathbb{C}$.

\begin{thm}
    There exists a dense subset $Q\subseteq \mathbb{C}$ such that $\mathcal{E}_q$ is semisimple for all $q\in Q$. For $q\in Q$ the simple objects of $\operatorname{Ab}(\mathcal{E}_q)$ are parameterised by the set
    \[  \{  (\lambda_1, \lambda_2) : \lambda_1,\lambda_2 \in Y_\infty^<,\quad \ell(\lambda_1) +  \ell(\lambda_2) \text{ odd}  \}     \cup \{  (\lambda_1, \lambda_2,\pm) : \lambda_1,\lambda_2 \in Y_\infty^<,\quad \ell(\lambda_1) +  \ell(\lambda_2) \text{ even}  \}. \]
    We have the following decomposition formulae for the tensor product of the simple object $(\square, \emptyset)$ with any other simple $(\lambda_1, \lambda_2,\varepsilon)$, where $\varepsilon \in \{+,- ,\emptyset\}$:
    \[\resizebox{\hsize}{!}{$(\square, \emptyset) \otimes (\lambda_1, \lambda_2,\varepsilon) \cong \bigoplus
    \limits_{\substack{  \lambda_1 \to^{(N)} \mu \\ \ell(\mu)+\ell(\lambda_2) \text{ even}  \\ \nu \in \{+,-\} }   } (\mu,\lambda_2, \nu)\oplus\bigoplus\limits_{\substack{  \mu \to^{(N)} \lambda_2 \\ \ell(\mu)+\ell(\lambda_1) \text{ even}  \\ \nu \in \{+,-\} }   } (\lambda_1,\mu, \nu)\oplus  \bigoplus\limits_{  \substack{\lambda_1 \to^{(N)} \mu \\ \ell(\mu) +\ell(\lambda_1)\text{ odd}   }  } m_{\lambda_1, \lambda_2}^{\mu}(\mu, \lambda_2)\oplus\bigoplus\limits_{\substack{  \mu \to^{(N)} \lambda_2 \\ \ell(\mu) +\ell(\lambda_1)\text{ odd} }   } m_{\lambda_2, \lambda_1}^{\mu}(\lambda_1,\mu)$}\]
    Here $m_{\lambda_1, \lambda_2}^\mu = 2$ if both $\ell(\lambda_1) + \ell(\lambda_2)$ and $\ell(\mu) + \ell(\lambda_2)$ are odd, and $m_{\lambda_1, \lambda_2}^\mu = 1$ otherwise. 
    
    The quantum dimensions of the simple objects are given by 
    \[
        d( \lambda_1, \lambda_2 ) =  q_{(\la_1,\lambda_2)}2^{- \frac{\ell(\la)+\ell(\mu)-1}{2}}, \qquad 
       d( \lambda_1, \lambda_2,\pm ) =  q_{(\la_1,\lambda_2)}2^{- \frac{\ell(\la)+\ell(\mu)}{2}}\]
    where the rational functions $q_{(\la_1,\lambda_2)}$ are evaluated at $q$.

\end{thm}
We also obtain similar fusion rules for tensoring by the object $(\emptyset, \square)$ in Theorem~\ref{Eqobjects}.

\begin{remark}
    We conjecture
    that the $Q \subseteq \mathbb{C}$ in the above theorem is the set of all complex numbers which are not roots of unity. 
\end{remark}

As an example, we include a subgraph of the fusion graph of $\mathcal{E}_q$ for the object $(\square, \emptyset)$. This subgraph consists of all simple objects with $|\lambda_1| + |\lambda_2| \leq 3$. As before, this must be read as an oriented graph. The orientations can be deduced from the vertex labels. 
\[\att{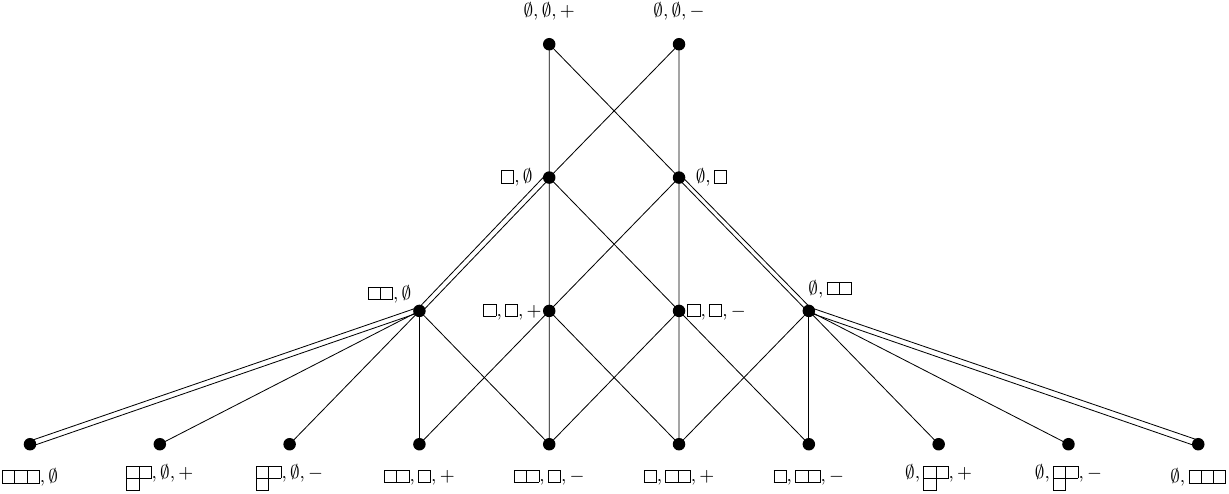}{.7}\]

The paper is outlined as follows.

In Section~\ref{sec:prelim} we review the required background material for this paper, and prove some basic results. We first review several categorical constructions: Cauchy completion and semisimplification. We then describe the categories $\mathcal{E}_q$ and $\mathcal{SE}_N$, as introduced in \cite{ConformalA}. Certain endomorphism algebras in these categories are subalgebras of the Hecke-Clifford algebras $G_n$. We describe these algebras, and their subalgebras $G_n[0]$. We prove several basic facts on these algebras. Next we turn to Sergeev duality, which is a Schur-Weyl type duality between the isomeric Lie super algebra $\q(N)$ and the Seergev algebras (which are the limit as $q\to 1$ of the algebras $G_n$). We also give properties of character formulas for representations of $\q(N)$.
Finally we describe the representations of the algebras $G_n$. These results were obtained by A. Jones and Nazarov in \cite{JN} over algebraic extensions of the field of rational functions $\C(q)$. As we will eventually be interested in the case of $q$ being a complex number, some additional details about the matrix coefficients are added. 



In Section \ref{sec:Gnreprootofunity} we define representations of $G_n$ and $\Gno$ 
for $q$ a primitive $4N$-th root of unity. We can not directly use the representations in \cite{JN} whose matrix entries would have poles for such $q$'s. We obtain well-defined representations at roots of unity by replacing tableaux in \cite{JN} by paths in certain graphs. We finish this section by defining a semisimple quotient $\bar G_n[0]$ of $G_n[0]$ using these new representations. Furthermore, we show for $q$ not a root of unity that the representations of \cite{JN} are well-defined, and that the algebras $G_n$ and $G_n[0]$ are semisimple. While our results on the representation theory of $G_n$ are not directly applicable to the structure of either $\operatorname{Ab}(\overline{\mathcal{SE}_N})$ nor $\operatorname{Ab}(\mathcal{E}_q)$, we expect that they will 
play a similar role for the computation of the structure of the super-categories (in the sense of \cite{MR4045471,Savage}) $\operatorname{Ab}((\overline{\mathcal{SE}_N})_{\operatorname{sVec}})$ and $\operatorname{Ab}((\mathcal{E}_q)_{\operatorname{sVec}})$. Here the copy of $\operatorname{sVec}$ in these categories is generated by the simple objects $(\emptyset, -)$ and $(\emptyset, \emptyset, -)$ respectively. 

In Section \ref{sec:description} we show that the algebra $\bar G_n[0]$ from Section \ref{sec:Gnreprootofunity} is isomorphic to $\End_{\overline{\mathcal{SE}_N}}(+^n)$. Our approach is basically the same as that used in \cite{HansThesis} where quotients of Hecke algebras were determined in connection to $SU(N)_k$ fusion categories. Our main technique is to define a trace on $\bar G_n[0]$ which agrees with the pull-back of the categorical trace.  This allows us to prove Theorem~\ref{thm:mainHans} using standard techniques. 
We are also able to obtain the structure of $\operatorname{Ab}(\mathcal{E}_q)$ at all values $q$ where this category is semisimple. 

We expect that our general approach will generalise to the categories coming from the conformal embeddings 
\[  \mathfrak{so}_N \subseteq    \mathfrak{so}_{N(N-1)/2} \quad \text{ and } \quad   \mathfrak{sp}_{2N} \subseteq    \mathfrak{so}_{2N^2 + N}.   \]
Here one would now have to define and study  \textit{BMW-Clifford} algebras, i.e. suitable $q$-deformations of semidirect products of Clifford algebras with Brauer algebras, to obtain similar results.

\subsection*{Acknowledgments}
CE was supported by NSF DMS grant 2245935 and 2400089. HW would like to thank CE and the University of New Hampshire for support and hospitality during his visit. CE would like to thank Daniel Copeland for helpful conversations. HW thanks Lilit Martirosyan and Brendon Rhoades for useful references. Both authors would like to thank Noah Snyder for helpful conversations. This material is based upon work supported by the National Science Foundation under Grant No. DMS-1928930, while both authors were in residence at the Mathematical Sciences Research Institute in Berkeley, California, during the Summer of 2024. Both authors would also like to thank BIRS for hosting them while part of this project was completed.

\section{Preliminaries}\label{sec:prelim}
We direct the reader to \cite{Book} for the basics on tensor categories.


\subsection{Cauchy Completion, Negligibles, and Semisimplification} \label{sec:Cauchy} 

The Schur-Weyl style approach to tensor categories is to study a category via a distinguished planar subcategory (often formalised as a planar algebra \cite{PA1}, or monoidal algebra \cite{SovietHans}). A fairly routine set of constructions allows one to reconstruct the entire category from the planar subcategory. The first of these constructions is the idempotent completion.
\begin{defn}
    Let $\mathcal{C}$ be a pivotal tensor category. The objects of $\operatorname{Idem}(\mathcal{C})$ are pairs $(X,p_X)$ where $X\in \mathcal{C}$, and $p_X\in \operatorname{End}_\mathcal{C}(X)$ is an idempotent. The morphisms are defined by 
    \[ \operatorname{Hom}_{\operatorname{Idem}(\mathcal{C})}( 
  (X,p_X) \to (Y,p_Y)  ):= \{  f\in \Hom_\mathcal{C}(X\to Y) : p_X \circ f = f = f\circ p_Y        \}.  \]
  The tensor product, composition, and pivotal structure are inherited from the base category $\mathcal{C}$.
\end{defn}

The second is the additive envelope.
\begin{defn}
    Let $\mathcal{C}$ be a pivotal $\mathbb{C}$-linear category. We define $\operatorname{Add}(\mathcal{C})$ as the category with objects formal finite direct sums
    \[  \bigoplus_i X_i   \]
    where each $X_i\in \mathcal{C}$, and morphisms matrices 
    \[  \begin{blockarray}{ccccc}
&Y_1  & Y_2 &\cdots & Y_m    \\
\begin{block}{c[cccc]}
X_1 & f_{1,1} & f_{1,2} & \cdots & f_{1,m} \\
X_2 & f_{2,1} & f_{2,2} & \cdots & f_{2,m} \\
\vdots & \vdots & \vdots& \ddots & \vdots \\
X_n &f_{n,1} & f_{n,2} & \cdots & f_{n,m} \\
\end{block}
\end{blockarray}\in \Hom_{\operatorname{Add}(\mathcal{C})}\left(\bigoplus_{i=1}^n X_i \to \bigoplus_{j=1}^m Y_j\right) \]
where $f_{i,j}\in \Hom_\mathcal{C}(X_i\to Y_j)$. The composition of morphisms is given by matrix composition, and the tensor product of morphisms is given by the Kronecker product. The pivotal structure is inherited from the pivotal structure on $\mathcal{C}$.
\end{defn}

The Cauchy completion of a category $\mathcal{C}$ is defined as the abelian envelope of the idempotent completion of $\mathcal{C}$. We will write $\operatorname{Ab}(\mathcal{C}):= \operatorname{Add}(\operatorname{Idem}(\mathcal{C}))$. It is shown in \cite[Theorem 3.3]{Tuba} that if the endomorphism algebras of $\mathcal{C}$ are semisimple, then $\operatorname{Ab}(\mathcal{C})$ is a semisimple category.


\begin{remark} \label{rem:RepstoObjects}
If $Y$ is an object in a category $\mathcal{C}$, any idempotent $\pi$ in $\End(Y)$ corresponds to a representation $\End(Y)\pi$ of $\End(Y)$. If $\End(Y)$ is semisimple, then any irreducible representation comes from a projection in this way. If $\pi_1$ and $\pi_2$ are projections in the same algebra $\End(Y)$ then the Homs between them in  $\operatorname{Ab}(\mathcal{C})$ is $\mathrm{Hom}_{\End(Y)}(\End(Y)\pi_1, \End(Y)\pi_2$. The Hom space between pairs $(Y,\pi)$ and $(Y', \pi')$ is much more complicated, and instead one typically embeds into a larger object containing both $Y$ and $Y'$ to work with representations of a single algebra.
\end{remark} 

The following Lemma is well-known and follows from the definition of induction.

\begin{lem}
    Suppose that $\pi \in \End(Y)$ is an idempotent, then $\End(Y\otimes X) (\pi \otimes \id)$ is isomorphic as an $\End(Y\otimes X)$ representation to $\mathrm{Ind}_{\End(Y)}^{\End(Y \otimes X)} \End(Y) \pi$. 
\end{lem}

By Frobenius reciprocity, we see that the fusion rules for tensoring with $X$ can be read off from the restriction rules for representations.

\begin{cor} \label{cor:RestrictionFusion}
    Suppose that $\pi \in \End(Y)$ and $\psi \in \End(Y \otimes X)$ are projections, then $\mathrm{Hom}_{\operatorname{Ab}(\mathcal{C})}(\pi \otimes \mathrm{id}_X, \psi) = \mathrm{Hom}_{\End(Y)}(\End(Y)\pi, \operatorname{Res} \End(Y \otimes X) \psi)$.
\end{cor}


The categories we work with in this paper will not always be semisimple on the nose. To obtain semisimple categories we will have to quotient out by the negligible ideal.

\begin{defn}
    Let $\mathcal{C}$ be a spherical category. We define the negligible ideal of $\mathcal{C}$ as
    \[ \operatorname{Neg}(\mathcal{C}) := \left\{ f\in \Hom_\mathcal{C}(X\to Y) : \operatorname{tr}(f\circ g) = 0 \text{ for all } g\in\Hom_\mathcal{C}(Y\to X) \right\}  \]
    where $\operatorname{tr}$ is the categorical trace.
\end{defn}
We have that $\operatorname{Neg}(\mathcal{C})$ is a tensor ideal of $\mathcal{C}$ \cite[Lemma 2.3]{Simp}. Hence we can form the quotient category. This category inherits the spherical structure of $\mathcal{C}$.
\begin{defn}
    Let $\mathcal{C}$ be a spherical category. We will write
    \[   \overline{\mathcal{C}}:= \mathcal{C} /   \operatorname{Neg}(\mathcal{C}).    \]
\end{defn}

\subsection{The Categories $\mathcal{E}_q$ and $\mathcal{SE}_N$.} 

In \cite{ConformalA} the first author and Snyder introduced a one parameter family of tensor categories. These categories are defined by generators and relations as follows.

\begin{defn}\cite[Definition 1.1]{ConformalA}\label{def:Eq}
    Let $q \in \mathbb{C}- \{-1,0,1\}$ and define $\mathcal{E}_q$ as the rigid $\mathbb{C}$-linear monoidal category with objects strings in $\{+,-\}$ and morphisms generated by the two morphisms
    \[ \att{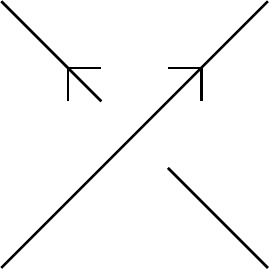}{.25}\quad,\quad\att{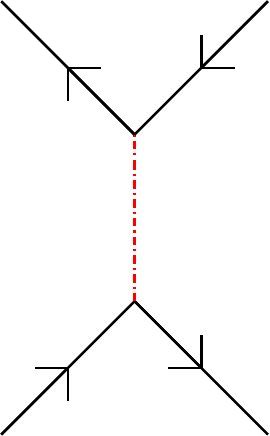}{.2} \in \End_{\mathcal{E}}(++) \]
    satisfying the relations:
    \begin{align*}
    &(\text{Loop}) \att{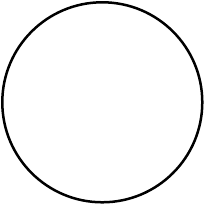}{.15} = \frac{2\mathbf{i}}{q-q^{-1}} \quad
    (\text{R1}) \att{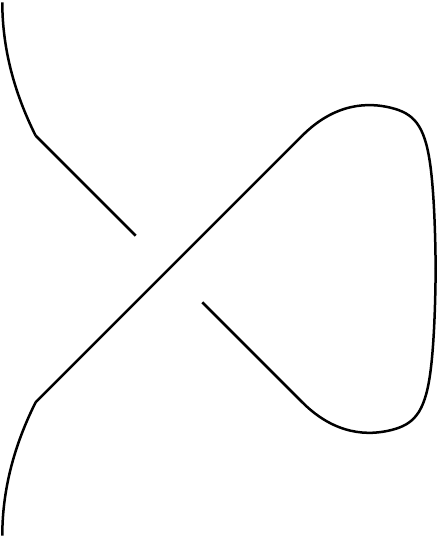}{.15} = \mathbf{i}\att{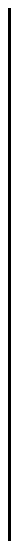}{.15}  \quad (\text{R2}) \att{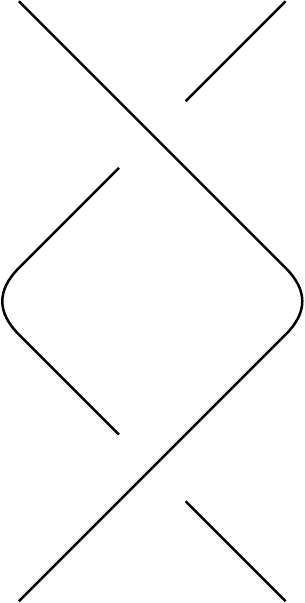}{.15} = \att{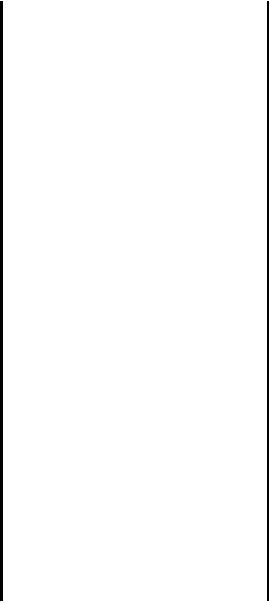}{.15}   \quad (\text{R3}) \att{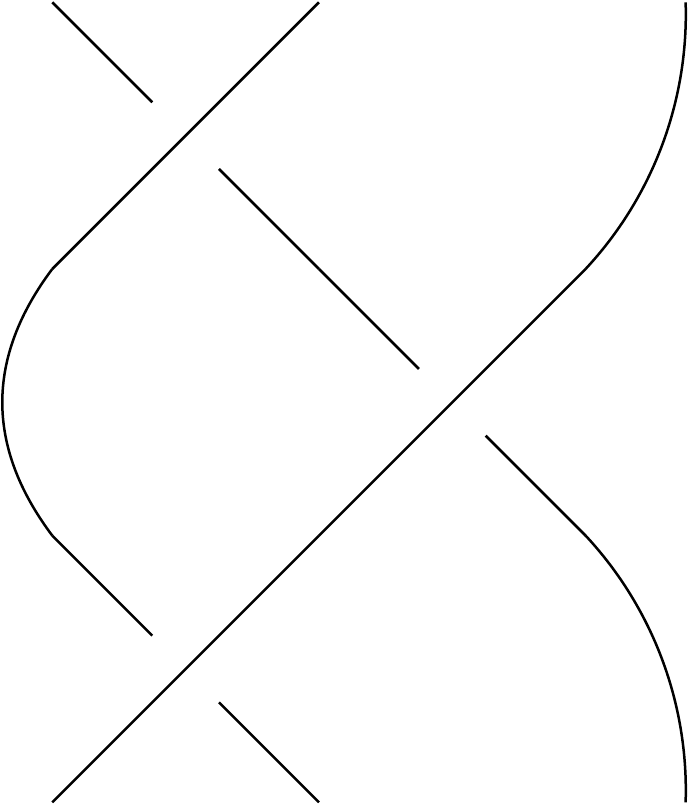}{.10} = \att{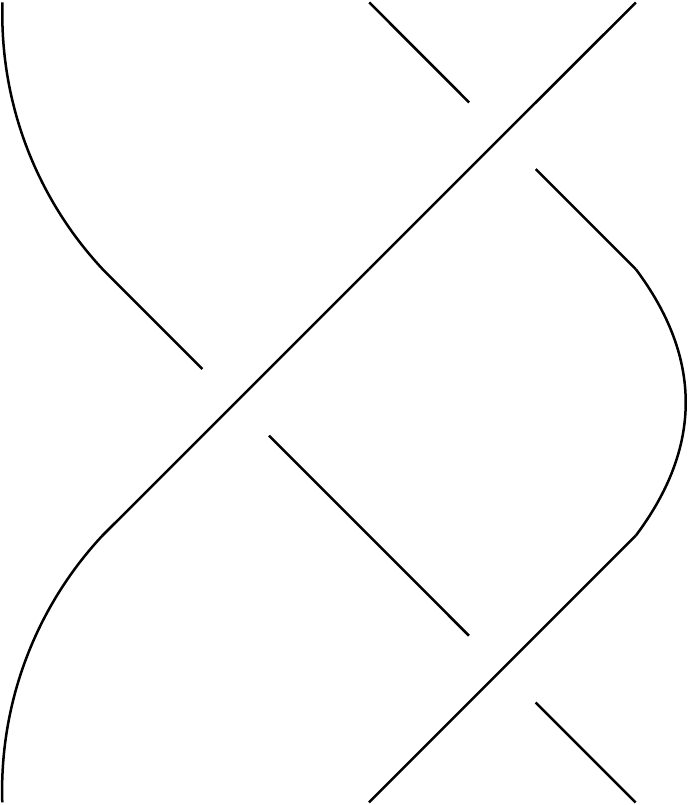}{.10}\\
    &(\text{Hecke})   \att{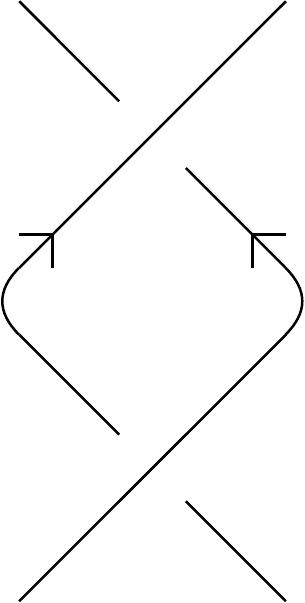}{.15} = \att{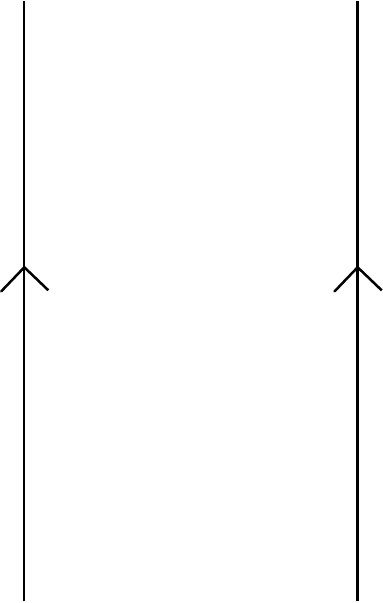}{.15} + (q-q^{-1})\att{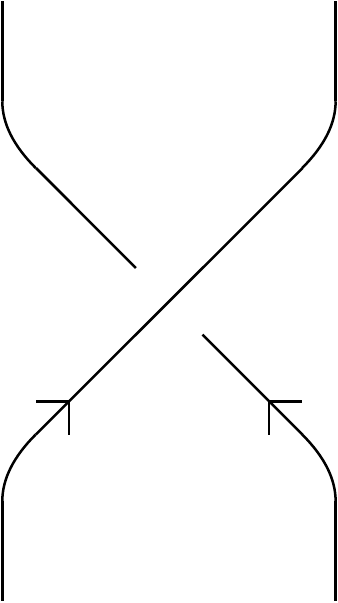}{.15}
   \quad  (\text{Trace}) \quad \att{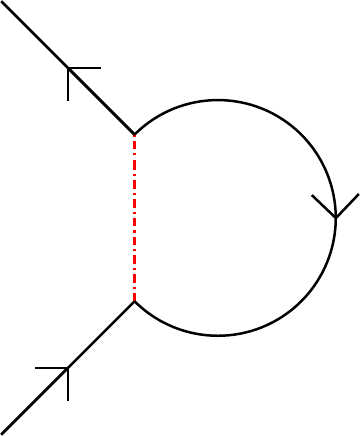}{.15} = \frac{q-q^{-1}}{2\mathbf{i}}\att{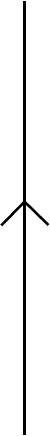}{.15}\quad (\text{Dual})  \left(\att{splittingEq}{.15}\right)^* =\att{splittingEq}{.15}\\
   &(\text{Half-Braid}) \quad \att{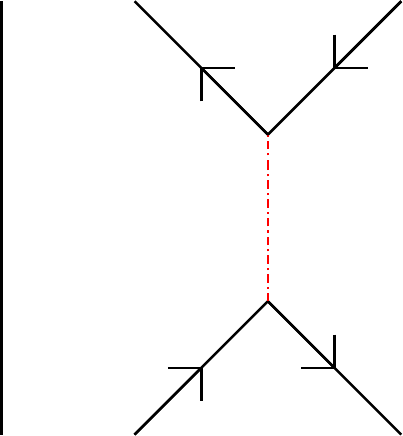}{.15} = \att{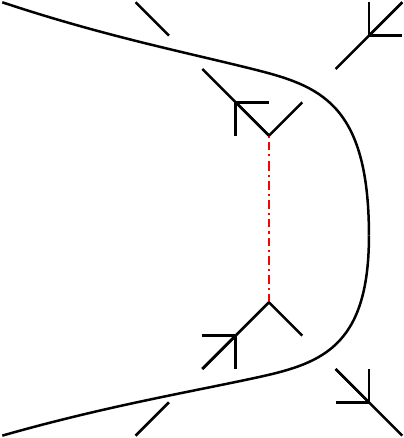}{.15} \quad (\text{Tadpole}) \att{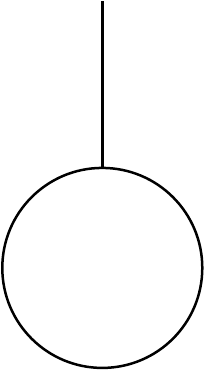}{.15} = 0\quad 
  (\mathbb{Z}_2)  \att{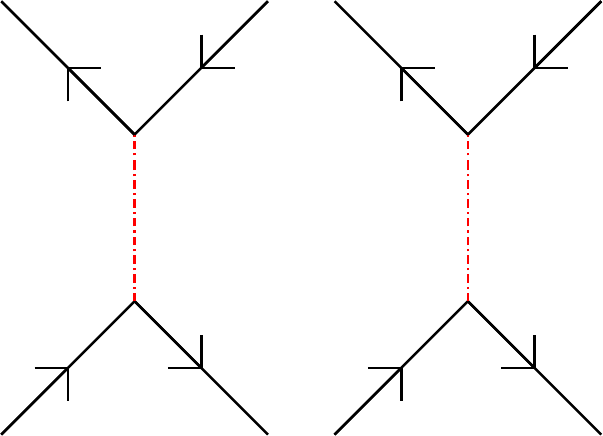}{.15} =  \att{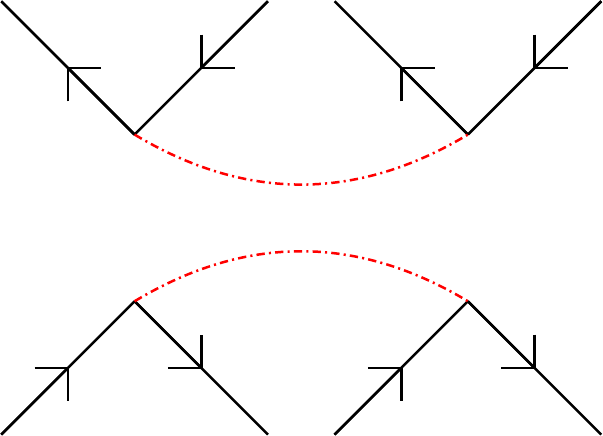}{.15}
    \end{align*}
\end{defn}

The categories $\mathcal{E}_q$ interpolate certain quantum subgroups of $\mathfrak{sl}_N$ in the same sense that Deligne's categories $\Rep(S_t)$ interpolate the representation categories $\Rep(S_n)$ \cite{Deligne}. More precisely at discrete values of the parameter $q$ an extension is defined as follows.

\begin{defn}\cite[Definition 1.3]{ConformalA}\label{def:ext}
    Let $N\in \mathbb{N}_{\geq 2}$. We define $\mathcal{SE}_N$ as the extension of $\mathcal{E}_{e^{2\pi i \frac{1}{4N}}}$ by the additional generator
    \[  \att{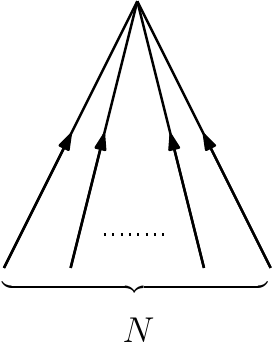}{.25}  \]
    satisfying the relations
    \[  (\text{$q$-Braid}) \att{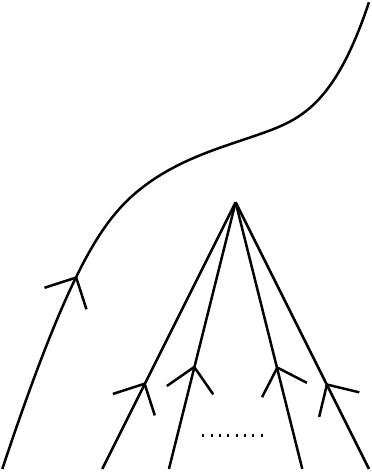}{.25} = q \att{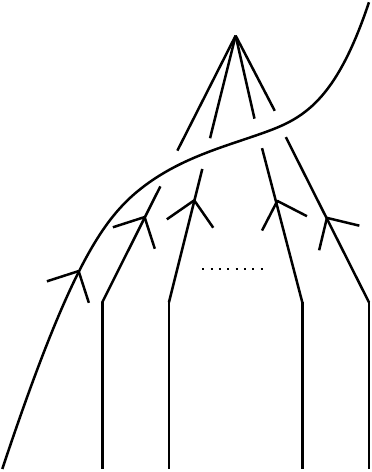}{.25}\qquad \text{ and }\qquad   (\text{Pair}) \att{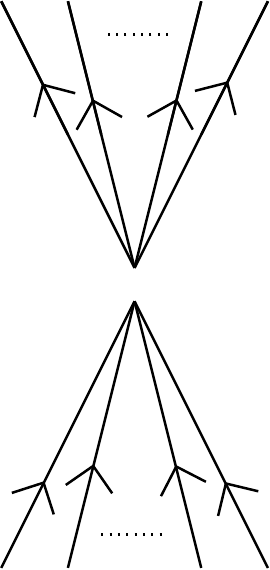}{.25} =  \att{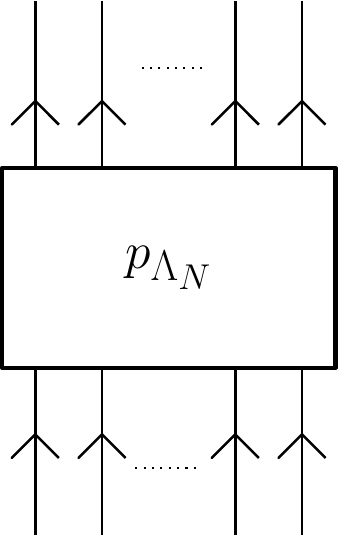}{.25} . \]
\end{defn}

The following equivalence is then shown in  \cite[Theorem 1.5]{ConformalA}.

\begin{thm}\cite[Theorem 1.5]{ConformalA}
Let $N\in \mathbb{N}_{\geq 2}$. Then there is a monoidal equivalence
\[ 
 \operatorname{Ab}(\overline{\mathcal{SE}_{N}})\simeq \overline{\operatorname{Rep}(U_q(\mathfrak{sl}_N))}_A\]
 where $A$ is the \'etale algebra object corresponding (in these sense of \cite[Theorem 5.2]{Kril}) to the conformal embedding
    \[     \mathcal{V}(\mathfrak{sl}_N,N) \subset  \mathcal{V}(\mathfrak{so}_{N^2-1},1).      \]
\end{thm}

Their proof of this theorem is non-constructive. In particular not even a classification of the simple objects is known for the category $\operatorname{Ab}(\overline{\mathcal{SE}_{N}})$. The goal of this paper is to give a complete description of the categories $\operatorname{Ab}(\mathcal{E}_{q})$ and $\operatorname{Ab}(\overline{\mathcal{SE}_{N}})$. We will achieve this by studying the endomorphism algebras $\operatorname{End}_{\operatorname{Ab}(\mathcal{E}_{q})}(+^n)$ and $\operatorname{End}_{\operatorname{Ab}(\overline{\mathcal{SE}_{N}})}(+^n)$ in these categories. It is shown in \cite[Corollary 5.12]{ConformalA} that the endomorphism algebras $\operatorname{End}_{\mathcal{E}_q}(+^n)$ are the \textit{Hecke-Clifford} algebras, which we now introduce. It follows that the algebras $\operatorname{End}_{\operatorname{Ab}(\overline{\mathcal{SE}_{N}})}(+^n)$ are quotients of the Hecke-Clifford algebras at certain roots of unity.

\subsection{The Hecke-Clifford algebras} Hecke-Clifford algebras were defined in \cite{Ol} in the context of centralizer algebras of quantum group versions of isomeric Lie super algebras.
We review results from \cite{JN} about their representations. They were proved over an algebraic extension of the field $\C(q)$ of rational functions over $\C$. In our setting, $q$ will be a complex number which is not a root of unity, for which essentially the same results hold. Adjustments will have to be made if $q$ is a root of unity. This will be done in Section \ref{sec:Gnreprootofunity}, based on the material in this section.

\begin{defn} We define the Hecke-Clifford algebras $G_n$ 
via generators $t_j$, $1\leq j<n$ and $v_j$, $1\leq j\leq n$ as follows:

\begin{itemize}
\item[(H)] The generators $t_j$ satisfy the relations of the Hecke algebras $H_n = \langle t_i : 1\leq i <n \rangle$ of type $A_{n-1}$. This means they satisfy
 the braid relations as well as the quadratic equation $t_j-t_j^{-1}=q-q^{-1}$.

\item[(C)] The elements $v_j$ generate the Clifford algebra $\mathrm{Cliff}(n)$
with relations $v_jv_k+v_kv_j=2\delta_{jk}$. 

\item[(M)] Moreover, we have the additional relations

\begin{equation}\label{cross1}
t_jv_j=v_{j+1}t_j,
\end{equation}

\begin{equation}\label{cross2}
t_jv_{j+1}= v_jt_j-(q-q^{-1})(v_j-v_{j+1}),
\end{equation}

\begin{equation}\label{cross3}
t_jv_l=v_lt_j, \hskip 3em l\neq j, j+1.
\end{equation}
    
\end{itemize}

\end{defn}

The following result was shown in \cite{JN} over the field $\C(q)$ of rational functions in $q$, where our generators $t_j$ and $v_j$ correspond to the elements $T_j$ and $iC_j$. It also holds for the corresponding complex algebra if we view $q$ as a complex number. 


\begin{thm}\label{Gnproperties} The algebra $G_n$ is equal to $H_n\mathrm{Cliff}(n)$
as a vector space and has dimension $2^n n!$. In particular, $G_n$ has a standard basis of the form $h_w c_s$ where $w$ ranges over $S_n$ and $s$ ranges over $\{0,1\}^n$. Similarly, $G_n = \mathrm{Cliff}(n)H_n$ and has a basis of the form $c_s h_w$.
\end{thm}

Observe that the algebra $G_n$ has a $\Z_2$-grading, where the elements $t_j$ have degree 0 and 
the elements $v_j$ have degree 1. It follows from the relations that the map
\begin{equation}\label{autom}
\al:\quad v_j\mapsto -v_j,\hskip 3em t_j\mapsto t_j,
\end{equation}
defines an automorphism of order two with eigenspaces $G_n[0]$ and $G_n[1]$.

\subsection{The algebras $\Gno$}\label{Gn0def} 


We first write generators and relations for $\Gno$, which requires a change of variables since the $v_j$ are odd. So we set $e_j=v_jv_{j+1}$, and note that $e_j \in \Gno$.

\begin{lem} \label{Gn0Presentation}
    The following relations hold in $\Gno$:
    \begin{itemize}
        \item[(E)] The $e$'s satisfy the relations $e_ie_j=e_je_i$ for $|i-j|\neq 1$, $e_j^2=-1$ and $e_je_{j+1}=-e_{j+1}e_j$, $1\leq j<n-1$.
        \item[(M)] We have the mixed relations

\begin{equation}\label{Heck2}
e_jt_j+t_je_j\ =\ (q-q^{-1})(1+e_j) \quad \Leftrightarrow\quad e_jt_j-(e_jt_j)^{-1}=(q-q^{-1})1.
\end{equation}

\begin{equation}\label{com1}
t_je_{j+1}=-e_je_{j+1}t_j, \quad e_jt_{j+1}=-t_{j+1}e_je_{j+1}.
\end{equation}

\begin{equation}\label{coj}
e_jt_{j+1}t_j=t_{j+1}t_je_{j+1}.
\end{equation}

    \end{itemize}

Moreover, $\Gno$ has a basis of the form $h_w e_s$ where $w$ ranges over $S_n$ and $s$ ranges over elements of $\{0,1\}^{n-1}$ where we again multiply the $e_i$ in lexicographical order.
\end{lem}
\begin{proof}
It is straightforward if somewhat tedious to check that the relations are satisfied, we check Relation \ref{Heck2} as an example:
\begin{align}
t_j(v_jv_{j+1})\ &=\ v_{j+1}t_jv_{j+1}\ =\ v_{j+1}( v_jt_j-(q-q^{-1})(v_j-v_{j+1}))\ =\cr
&=\ -(v_jv_{j+1})t_j+(q-q^{-1})(v_jv_{j+1}+1).\cr
\end{align}

Clearly $G_n[0]$ is a product of the Hecke algebra and the even part of the Clifford algebra. A simple induction on the number of strands shows that the even part of the Clifford algebra has a basis given by products of $v_jv_{j+1}$ in lexicographic order.
\end{proof}

By \cite[Corollary 5.12]{ConformalA} we have the isomorphism of algebras
\begin{equation}\label{eq:iso}
    \operatorname{End}_{\mathcal{E}_q}(+^n) \cong G_n[0]
\end{equation}
for all $q\in \mathbb{C} - \{-1,0,1\}$.

\begin{remark}\label{Gnremark} The algebras $G_n$ and $G_n[0]$ are $q$-deformations of the semidirect products $\operatorname{Cliff}(n)\rtimes S_n$ and $\mathrm{Cliff}(n)[0]\rtimes S_n$, with the obvious $S_n$ action on the generators $v_i$. Just like $\operatorname{Cliff}(n)\rtimes S_n$ quotients onto $\End_{\q_N}(V^{\otimes n})$ (see Theorem \ref{Sergeev theorem}), the algebra $G_n$ quotients onto $\End_{U_q(\q_N)}(V^{\otimes n})$, see \cite{Ol}. The algebra $G_n[0]$ quotients onto the ring of even endomorphisms $\End_{U_q(\q_N)}(V^{\otimes n})$. 
Although we will eventually be interested in the structure of $G_n[0]$, it will be more convenient to work with the algebra $G_n$ first.
\end{remark}

\begin{prop} \label{HCSpanning}
    $G_{n+1}$ is spanned by elements of the form $a\chi b$, with $a,b\in G_n$ and $\chi\in \{ 1, t_n, v_{n+1}, v_{n+1}t_n\}$.

    $G_{n+1}[0]$ is spanned by elements of the form $a\chi b$, with $a,b\in G_n[0]$ and $\chi\in \{ 1, t_n, e_{n}, e_{n}t_n\}$
\end{prop}
\begin{proof}
It is well-known that the Hecke algebra $H_{n+1}$ is spanned by elements of the form $a \chi' b$ for $\chi' \in \{1,t_n\}$ and $a, b \in H_n$ (e.g. \cite[Eq (2.2)]{HansThesis}). Moreover, any element of $\mathrm{Cliff}(n+1)$ is of the form $c \chi'' $ for $\chi'' \in \{1, v_{n+1} \}$ and $c \in \mathrm{Cliff}(n)$. Thus by, Thm. \ref{Gnproperties}, every element of $G_n$ is of the form $c\chi''a\chi'b$. But $c\chi''a\chi'b = c a \chi'' \chi' b$, since $1$ and $v_{n+1}$ commute with $H_n$, so the first result follows. The proof for $G_{n+1}[0]$ follows in the same manner. 
\end{proof}

\begin{lem} \label{q1lemma}
    The automorphism $\gamma_n$ of $G_n[0]$, defined by $\gamma_n(c) = v_n^{-1} c v_n = v_n c v_n$, sends $t_{n-1} \mapsto t_{n-1} e_{n-1}$ and $e_{n-1} \mapsto -e_{n-1}$ and fixes all other generators. This outer automorphism is the restriction of conjugation by $e_n$ on $G_{n+1}[0]$. 
\end{lem} 
\begin{proof}
    If $c$ is even then $v_n^{-1} c v_n$ is also even, so this is clearly an automorphism. The formula on the generators follows by direct calculation, for example
    \[ v_n t_{n-1} v_n = t_{n-1} v_{n-1} v_n = t_{n-1} e_{n-1} \]
\end{proof}





The algebras $G_n$ and $G_n[0]$ both contain as subalgebras the Hecke algebras $H_n$. These algebras have been well studied in the literature \cite{HansThesis}. In the case that $q^{2j}\neq 1$ for $2\leq j\leq n$ there exist minimal central idempotents $p_{(n)}$ and $p_{(1^n)}$ in the $H_n$ which satisfy $t_jp_{(n)}=qp_{(n)}$ and $t_jp_{(1^n)}=-q^{-1}p_{(1^n)}$,
$1\leq	j<n$. We can show that these idempotents remain minimal in the extension $G_n[0]$.

\begin{lem}\label{pnlemma}Assume that $q^{2j}\neq 1$ for $2\leq j\leq n$. Then the idempotents $p_{(n)}$ and $p_{(1^n)}$ are minimal idempotents in $\Gno$.
\end{lem}

\begin{proof}  We will give the proof for $p_{(n)}$, as the $p_{(1^n)}$  case is near identical. 

We want to show $p_{(n)}fp_{(n)}$ is a scalar multiple of $p_{(n)}$ for all $f$. Let $F_n$ be the subalgebra generated by $e_j$, $1\leq j<n$. Since $p_{(n)}$ is a minimal (and not merely minimal central) idempotent in $H_n$, we see that $hp_{(n)}$ is a scalar multiple of $p_{(n)}$ for every $h\in H_n$. Since $\Gno=F_n H_n$, it suffices to show that $p_{(n)}fp_{(n)}$ is a scalar multiple of $p_{(n)}$ for $f\in F_n$. We proceed by induction on $n$.

For the base case $n=2$, observe that $[2]p_{(2)}=t_1+q^{-1}$. It then follows from \ref{Heck2} that
$$p_{(2)}e_1p_{(2)}\ =\  \frac{1}{[2]}p_{(2)}(-t_1^{-1}e_1+(q-q^{-1})+q^{-1}e_1)\ =\ \frac{q-q^{-1}}{q+q^{-1}}p_{(2)}.$$

We now do the induction step $n-1$ to $n$. For $f\in F_{n-1}$, since $p_{(n)}f p_{(n)}=p_{(n)}(p_{(n-1)}fp_{(n-1)})p_{(n)}$ the claim follows from the induction hypothesis.

The remaining cases are all of the form $f=f'e_{n-2}e_{n-1}$ or $f = f' e_{n-1}$ with $f'\in F_{n-2}$.

If $f=f'e_{n-2}e_{n-1}$ with $f'\in F_{n-2}$ , we have
\begin{align*}
p_{(n)}f'e_{n-2}e_{n-1}p_{(n)}&=
q^{-1}p_{(n)}t_{n-1}f'e_{n-2}e_{n-1}p_{(n)} = q^{-1}p_{(n)}f't_{n-1}e_{n-2}e_{n-1}p_{(n)}\\
&=-q^{-1}p_{(n)}f'e_{n-2}t_{n-1}p_{(n)}=-p_{(n)}f'e_{n-2}p_{(n)},
\end{align*}
where we used Relation \ref{com1}. Since $f' e_{n-2} \in F_{n-1}$, we have completed the proof in this case.

If $f'e_{n-1}\in F_{n-2}e_{n-1}$, we have 
$$p_{(n)}f'e_{n-1}p_{(n)}=p_{(n)}f'(1_{n-2}\otimes p_{(2)}e_1p_{(2)})p_{(n)}=\frac{q-q^{-1}}{q+q^{-1}}p_{(n)}f'p_{(n)},$$
where we've used our calculation for $n=2$ above. As $f' \in H_{n-1}$,  we've also finished the proof in this case.
\end{proof}

By Equation~\eqref{eq:iso} we have that $G_n[0]$ is an endomorphism algebra of a pivotal category, and hence inherits a diagrammatic trace. By setting $d = \frac{2i}{q-q^{-1}}$ and normalising the diagrammatic trace by $\frac{1}{d^n}$ we obtain a Markov trace in the sense of \cite[Definition 3.1]{HansThesis}. It is immediate from the definition that a Markov trace on $G_n[0]$ restricts to a Markov trace on $H_n$. It is proven in \cite[Equation 3.2]{HansThesis} that a Markov trace on $H_n$ is fully determined by the value
\[\eta := \mathrm{tr}(p_{(2)}) = \mathrm{tr}\left(\frac{q - t_1}{q+q^{-1}}\right). \] 
The value $\eta$ is easily computed to be $\frac{1}{2}$ using the defining relations of $\mathcal{E}_q$ given in Definition~\ref{def:Eq}.

In \cite[Equation 3.6]{HansThesis} the traces of all the simple projections for all Markov traces were calculated via a hook-length formula. We will require only the following special case for the diagram $(m)$. 

\begin{prop} \label{prop:HansDimensions}
    Let $\mathrm{tr}_\eta$ be the Markov trace on $H_n(q)$ with parameter $\eta$, then
    \[\mathrm{tr}_{\eta}(p_{(m)}) = \prod_{j=1}^m \frac{(q+q^{-1})(q^{j-1}-q^{1-j})\eta -(q^{j-2}-q^{2-j})}{q^{j}-q^{-j}}\] 
    
    In particular, when $\eta=1/2$ we have,
    \[\mathrm{tr}_{\frac{1}{2}}(p_{(m)}) = \frac{1}{2^m}\prod_{j=1}^m \frac{ (q^{j-1} + q^{1-j})(q-q^{-1})}{q^j-q^{-j}}.\] 
    \end{prop}

\begin{remark}
    When comparing our formulas to \cite{HansThesis} note that there the conventions for the Hecke algebra are slightly different: his $q$ is our $q^2$, and his generator $g_i$ is our $q t_i$.
\end{remark}

\subsection{Strict Young diagrams and strict tableaux}\label{strict} We call a Young diagram $\la=(\la_1,\la_2,\la_3,...\la_{\ell(\la)})$ $strict$
if $\la_1>\la_2>\ ...\ >\la_{\ell(\la)}>0$. A \emph{shifted Young diagram}
is obtained from the strict Young diagram $\la$ by shifting its $i$-th row
by $(i-1)$ boxes to the right. We will denote it by  $[\la_1,\la_2,\la_3,...\la_{\ell(\la)}]$.
Given a strict Young diagram $\la$, we denote by $\S_\la$ the set of all standard tableaux $\La$
for its corresponding \emph{shifted} Young diagram.
We have
 a 1-1 correspondence between the elements $\La\in\S_\la$ and paths of the form
\begin{equation}\label{paths}
\emptyset\ \to\ \La(1)\ \to\ \La(2)\ \to\ ...\ \ \to\ \La(n)=\la,
\end{equation}
where $\La(i+1)$ is a shifted Young diagram obtained by adding a box to the shifted diagram $\La(i)$.
The correspondence is, of course, given by defining $\La(i)$ to be the diagram consisting of the boxes of $\La$
containing the numbers 1 through $i$. It is easy to see that the inverse of the shift map assigns to each tableau in $\S_\la$ a standard tableau of shape $\la$. So we can identify $\S_\la$ as a subset of the standard tableaux of the ordinary Young diagram corresponding to $\la$. 

As an example, the only tableau in $\S_{[2,1]}$ is
\ytableausetup{boxsize=normal}
\[\begin{ytableau}
1&  2  \\
3
\end{ytableau}\]

since the shifted tableaux is increasing in the second column ($2 < 3$), while 
\[\begin{ytableau}
1& 3  \\
2
\end{ytableau}\]
has shifted tableaux which is not standard since $3 > 2$.

\ytableausetup{boxsize = .3em}

Let $[a,b]_{\Z}=\{ m\in \Z, a\leq m\leq b\}$ and let $R_\la=S_{[1,\la_1]_{\Z}}\times S_{[\la_1+1,\la_1+\la_2]_{\Z}}\times ... \times S_{[n-\la_{\ell(\la)}+1,n]_{\Z}}$
be the row stabilizer of the standard tableau $\La^r$ obtained by filling $\la$ row by row.  We have an obvious action of the symmetric group $S_n$ on a tableau $\La\in\S_\la$ by permuting the numbers in its boxes. For each $1\leq k \leq n$ we define $s_k(\La)$ as the (possibly non-standard) tableau obtained by interchanging the numbers $k$ and $k+1$ in $\La$. A partial order was defined on $\S_\la$ in \cite[Section 6]{JN}.

\begin{defn}\label{def:parOrd}
    Let $\lambda$ be a strict Young diagram, and $\La_0$ the column tableau on $\lambda$ (i.e. the numbers are filled in column by column). We say $w_1(\La_0)<w_2(\La_0)$ if $w_1\prec w_2$ in weak Bruhat order. 
\end{defn}

\subsection{The isomeric Lie super algebra $\q(N)$} The isomeric (or queer) Lie super algebra $\q(N)$ consists of all $2N\times 2N$
matrices of the form
$$\left[\begin{matrix} A&B\cr B&A\end{matrix}\right],$$
acting on the super vector space $V=\C^{N|N}$. 
Here the even part consists of matrices with $B=0$ and the odd part consists of matrices with $A=0$.

For representations of super algebras, $\Hom(X,Y)$ is naturally $\mathbb{Z}/2$-graded. There is a tensor super category (where the exchange relation only holds up to sign) where you allow all morphisms, or a tensor category where you only allow the even morphisms $\Hom(X,Y)_0$. We will denote the latter by $\Rep(\mathfrak{g})$. Note that by isomorphism of representations we will always mean \emph{even} isomorphisms, i.e. isomorphisms in $\Rep(\mathfrak{g})$. 

The operator $P: V \rightarrow V$, defined by $P(v|w)=i(w|-v)$, is an odd homomorphism. The element $P_i\in \End(V^{\otimes n})$ acts via $P$ on the $i$-th factor of $V^{\otimes n}$, and as the identity on all the other factors.
The following result is usually referred to as Schur-Sergeev duality for $\q(N)$.

\begin{thm}\label{Sergeev theorem} \cite{Ser}\begin{enumerate}[(a)]
    \item  The algebra of all endomorphisms $\End_{\q(N)}(V^{\otimes n})$ is generated by the symmetric group $S_n$, acting via permuting the factors in $V^{\otimes n}$ and the element $P_1$.

    \item  The algebra of even endomorphisms $\End_{\q(N)}(V^{\otimes n})_0$ is generated by the symmetric group $S_n$ and the element $P_1P_2$. 

    \item  The algebras $\End_{\q(N)}(V^{\otimes n})$ and $\End_{\q(N)}(V^{\otimes n})_0$ are semisimple \cite[Proposition 2.2]{JN}.
\end{enumerate}

\end{thm}

In particular, by part (c), the category of polynomial representations of $\q(N)$ (i.e. the ones which appear in $V^{\otimes n}$) is a semisimple category. By contrast, the category of all finite-dimensional representations is not semisimple.

\subsection{Characters for $\q(N)$}\label{qNcharactersection}
Recall that in the setting of super algebras there are two different notions of characters: characters (which don't distinguished between odd and even weight vectors) and super characters (which introduce a minus sign for odd weight vectors) \cite[Sec. 7]{MR1773773}. In the isomeric case the super characters all vanish, so we only consider characters. Kac \cite{KacAlg} showed that irreps are determined up to parity shift by their highest weights. As a consequence, any semisimple representation of $\q(N)$ (in particular, a polynomial representation) is determined up to parity shift by its character. 

The irreducible polynomial representations of the isomeric Lie super algebra $\q(N)$ and their characters were determined by Sergeev in \cite{Ser}. In particular, for each strict Young diagram $\la$ with $\ell(\la) \leq N$ 
Sergeev defines a $\q(N)$ representations $M_\la$, and shows $M_\la$ is irreducible if $\ell(\la)$ is odd, and the direct sum of two
irreducible modules $M_{\la, \pm}$ for $\ell(\la)$  even. Note that $M_{\la, \pm}$ differ by parity shift, while $M_{\la}$ is isomorphic to its own parity shift. Moreover, he proves that any polynomial representation of $\q(N)$ is isomorphic to one of these representations.

In order to describe the characters of Sergeev's polynomial representations we will introduce some auxiliary functions. Let $x_i$, $1\leq i\leq N$ and $t$ be variables. 

\begin{defn}
For the strict Young diagram
$\la$, we define the  polynomial
\begin{equation}\label{Qdef}
Q_\la(x,t)\ =\ \frac{(1-t)^N}{\Phi_{N-r}(t)} \sum_{w\in S_N}  w \left(x^\la\prod_{i<j}\frac{x_i-tx_j}{x_i-x_j}\right).
\end{equation}
Here $r=\ell(\la)$ is the number of rows of $\la$, $x^\la=\prod_{j=1}^Nx_j^{\la_j}$, $w\in S_N$ acts via permuting the variables $x_i$ and 
$$\Phi_m(t)=\prod_{j=1}^m(1-t^j).$$    
We further define the specialization $Q_\la(x)=Q_\la(x,-1)$. The polynomials $Q_\la(x,t)$ are Hall-Littlewood polynomials, up to a power of $(1-t)$, see \cite{Ser} after Proposition 1.
\end{defn}
Let $\operatorname{ch}(M_\la)$ be the trace of a diagonal matrix in $\q(N)$ with eigenvalues $x_i$,  $1\leq i\leq N$ in the representation $M_\la$.  Then it was shown in \cite{Ser} (see last formula on p 426) that:

\begin{prop}\label{qNcharacter}
If $\ell(\la)$ is odd, then
\begin{equation}
\operatorname{ch}(M_\la)=2^{-\lfloor \ell(\la)/2\rfloor}\ Q_\la(x,-1) =2^{-\lfloor \ell(\la)/2\rfloor}\  Q_\la(x),
\end{equation}
 where  $\lfloor y \rfloor$ is the integer part of $y$. 
 
If $\ell(\la)$ is even, then 
\begin{equation}
\operatorname{ch}(M_{\la,\pm})=2^{-\lfloor \ell(\la)/2\rfloor}\ Q_\la(x,-1) =2^{-\lfloor \ell(\la)/2\rfloor}\  Q_\la(x),
\end{equation}
\end{prop}




\subsection{Formal characters for $\q(\infty)$}\label{sub:formal} We review the usual construction of defining Hall-Littlewood functions if we let the number of variables go to infinity, see e.g. \cite{Mac}.
\begin{prop}\label{prop:Qrecursion}
Let $Q_\la$ be as in the last subsection, with $N>|\la|$.
We then have the following recursion relations 
\begin{align}\label{recursion1}
Q_{(\la_1,\la_2)}&=Q_{\la_1}Q_{\la_2}-2Q_{\la_1+1}Q_{\la_2-1}+2Q_{\la_1+2}Q_{\la_2-2}\ ...  +(-1)^{\ell(\la)}Q_{\la_1+\la_2},\cr
&=Q_{\la_1}Q_{\la_2}-Q_{\la_1+1}Q_{\la_2-1}-Q_{(\la_1+1,\la_2-1)}.
\end{align}
\begin{equation}\label{recursion2}
Q_{(\la_1,\la_2,\ ...,\la_r)}\ =\ Q_{\la_1}Q_{(\la_2,\ ...\ \la_r)}-Q_{\la_2}Q_{(\la_1,\la_3\ ...\ \la_r)}+\ ...\ +Q_{\la_r}
Q_{(\la_1,\la_2,\ ...,\la_r)}
\end{equation}
 if $r$ is odd. For $r$ even, we have
\begin{equation}\label{recursion3}
Q_{(\la_1,\la_2,\ ...,\la_r)}\ =
\end{equation}
$$=\ \ Q_{(\la_1,\lambda_2)}Q_{(\la_3,\ ...\ \la_r)}-Q_{(\la_1,\la_3)}Q_{(\la_2,\la_4\ ...\ \la_r)}+\ ...\ +Q_{(\la_1,\la_r)}
Q_{(\la_2,\la_3,\ ...,\la_{r-1})}.
$$    

\end{prop}
\begin{proof}
    See \cite{Ser}, below Proposition 3.
\end{proof}

Observe that the equations in Proposition \ref{prop:Qrecursion} provide an algorithm how to express the function $Q_\la$ as a polynomial in the functions $Q_{(m)}$ by induction on
$r=\ell(\la)$ and $\la_r$. In particular, we can make the following definition.


\begin{defn}\label{rem:Hall}
Let $Q_\la$ be the unique polynomial in countably many variables $a_1, a_2, \ldots$ with the property that 
\[Q_\la(Q_{(1)}, \ldots Q_{(n)}, 0, 0, \ldots) = Q_\lambda(x_1, \ldots x_n).\]

For the rest of the paper the symbol $Q_\la$ will mean this polynomial in infinitely many variables. 


\end{defn}

Let us write $K_0(\operatorname{Rep}^\textrm{poly}(q_\infty))$ for the formal fusion ring of polynomial representations of $q_\infty$. That is the subcategory of representations generated by the vector representation. The character map $\operatorname{ch}$ gives a homomorphism from $K_0(\operatorname{Rep}^\textrm{poly}(q_\infty))$ into the ring of polynomials in infinitely many variables. This map is not injective, as for $\ell(\lambda)$ even we have $\operatorname{ch}([M_{\lambda,+}]) = \operatorname{ch}([M_{\lambda,-}])$. However defining 
\[\hat{M}_\lambda := \begin{cases}
    M_\lambda \quad & \ell(\lambda) \text{ odd}\\
    M_{\lambda,+} \oplus  M_{\lambda,-}\quad & \ell(\lambda) \text{ even}
\end{cases}\] gives that $\operatorname{ch}$ is injective on the subring generated by the elements $[\hat{M}_\lambda]$. With this we can prove the following results on dimension functions on $K_0(\operatorname{Rep}^\textrm{poly}(q_\infty))$.
\begin{prop}\label{prop:dimFun}
    Let $d : K_0(\operatorname{Rep}^\textrm{poly}(q_\infty))\to \mathbb{C}$ be a homomorphism such that $d(  [M_{\lambda,+}])  = d( [ M_{\lambda,-}])$. Then $d$ is determined by its values on the elements $[M_{(m)}]$, and furthermore we have that
    \[ d([M_{\lambda}]) = \frac{Q_\lambda}{  
 2^{ \frac{\ell(\lambda) -1}{2}}  }  \qquad  \text{if }\ell(\lambda) \text{ odd} \qquad\text{and}\qquad d([M_{\lambda,\pm}]) = \frac{Q_\lambda}{  
 2^{ \frac{\ell(\lambda)}{2}}  }  \qquad  \text{if }\ell(\lambda) \text{ even}      \]
 where $Q_\lambda$ are specialised as in Definition~\ref{rem:Hall} to $a_m = d([M_{(m)}])$.
\end{prop}
\begin{proof}
    We prove this by induction on $\ell(\lambda)$. The case of $\ell(\lambda)=1$ holds by the choice of specialisation. For the case of $\ell(\lambda) = 2$ we further induct on $\lambda_2$. We have from Equation 6 that
    \[  [\hat{M}_{(\lambda_1, \lambda_2)}]  = [ \hat{M}_{(\lambda_1)}][ \hat{M}_{(\lambda_2)}] - [ \hat{M}_{(\lambda_1+1)}][ \hat{M}_{(\lambda_2-1)}] - [ \hat{M}_{(\lambda_1+1,\lambda_2-1)}].\]
    By the inductive assumption we have
    \[ d([\hat{M}_{(\lambda_1, \lambda_2)}]) = Q_{(\lambda_1)} Q_{(\lambda_2)}  - Q_{(\lambda_1+1)}Q_{(\lambda_2-1)}-Q_{(\lambda_1+1,\lambda_2-1)} = Q_{(\lambda_1, \lambda_2)}.   \]
    As $d([M_{(\lambda_1, \lambda_2), +}])=d([M_{(\lambda_1, \lambda_2), -}])$ it follows that $d([M_{(\lambda_1, \lambda_2), \pm}]) = \frac{Q_{(\lambda_1, \lambda_2)}}{2}$ as desired.

    When $\ell(\lambda)$ is even we have from Equation 8 that
    \[  2[\hat{M}_{(\lambda_1,\cdots, \lambda_r)}]  =  [\hat{M}_{(\lambda_1, \lambda_2)}][\hat{M}_{(\lambda_3, \cdots,\lambda_r)}] -[\hat{M}_{(\lambda_1, \lambda_3)}][\hat{M}_{(\lambda_2, \lambda_4,\cdots,\lambda_r)}] +\cdots +  [\hat{M}_{(\lambda_1, \lambda_r)}][\hat{M}_{(\lambda_1,\cdots,\lambda_{r-1})}].\]
    The same logic as in the $\ell(\lambda) = 2$ case gives that $d([\hat{M}_{(\lambda_1,\cdots, \lambda_r)}]) =\frac{Q_{(\lambda_1, \cdots, \lambda_r)}}{2^{\frac{r-2}{2}}}$, and so $d([M_{(\lambda_1,\cdots, \lambda_r),\pm}]) =\frac{Q_{(\lambda_1, \cdots, \lambda_r)}}{2^{\frac{r}{2}}}$.

    The easier case that $\ell(\lambda)$ is odd is left to the reader. 
\end{proof}

We now make a concrete choice of specialisation of the function $Q_\la$ as explained in Definition~\ref{rem:Hall}. While this choice of specialisation may seem unmotivated as of now, it is exactly the specialisation that will allow us to give formulae for the quantum dimensions of the simple objects in $\operatorname{Ab}(\mathcal{E}_q)$ and $\operatorname{Ab}(\mathcal{SE}_N)$.

We define for each positive integer $m$ the rational function
\begin{equation}\label{Qmvalues}
a_m\ =\ \left(\frac{i}{q-q^{-1}}\right)^m\ \prod_{j=1}^m\frac{ q^{j-1} + q^{1-j}}{[j]}\ =\ \prod_{j=1}^m i\frac{ q^{j-1} + q^{1-j}}{q^j-q^{-j}}.
\end{equation}
This is the (unnormalized) diagram trace $d^m \mathrm{tr}_{\frac{1}{2}}(p_{(m)})$ from Proposition~\ref{prop:HansDimensions}. The rationale for this choice of specialisation will become clear in Section~\ref{sec:Cauchy}.

\begin{thm}\label{qdimensionsGn}
Let $\la$ be a Young diagram with $r$ rows and let $q_\la=Q_\la(\mathbf{a})$ be the rational function as defined in Remark \ref{rem:Hall} specialized to the values $a_m$ as in \ref{Qmvalues}.
Then we have
$$q_\la\ =\ \prod _{j=1}^r q_{(\la_j)}\ \prod_{i<j} \frac{[\la_i-\la_j]}{[\la_i+\la_j]}.$$
\end{thm}

\begin{proof}
To check the formula for $\la=(\la_1,\la_2)$, we deduce from \ref{recursion1} that
$$Q_{(\la_1,\la_2)}\ =\ Q_{\la_1}Q_{\la_2}- Q_{\la_1+1}Q_{\la_2-1} - Q_{(\la_1+1,\la_2-1)}.$$
We can now prove the claim by induction on $\la_2$ using $q_{(m)}=a_m$ and
\begin{align}
q_{(\la_1)}q_{(\la_2)}-q_{(\la_1,\la_2)}\ &=\  q_{(\la_1+1)}q_{(\la_2-1)} + q_{(\la_1+1,\la_2-1)}\cr
&=\ \frac{[\la_1+\la_2]+[\la_1-\la_2+2]}{[\la_1+\la_2]}\ q_{(\la_1+1)}q_{(\la_2-1)}\cr
&=\ \frac{[\la_1+1](q^{\la_2-1}+q^{1-\la_2})}{[\la_1+\la_2]}\ q_{(\la_1+1)}q_{(\la_2-1)}\ =\cr
&=\ \frac{[\la_2](q^{\la_1}+q^{-\la_1})}{[\la_1+\la_2]}\ q_{(\la_1)}q_{(\la_2)}\ =\ \frac{[\la_1+\la_2]-[\la_1-\la_2]}{[\la_1+\la_2]}\ q_{(\la_1)}q_{(\la_2)}\cr
&=\left(1-\frac{[\la_1-\la_2]}{[\la_1+\la_2]}\right)\ q_{(\la_1)}q_{(\la_2)}.
\end{align}
This proves the claim for diagrams with two rows. To prove it for diagrams with more rows, it will be convenient to 
make the substitution $y_i=q^{2\la_i}$. Then we have
$$\frac{[\la_i-\la_j]}{[\la_i+\la_j]}\ =\ \frac{y_i-y_j}{y_iy_j-1}.$$
Let $[m,n]_{\Z}$ denote the set of all integers $k$ satisfying $m\leq k\leq n$, and let, for any subset $S\subset [1,r]_{\Z}$
$$\Delta(S)=\prod_{i,j\in S, i<j} y_i-y_j, \hskip 3em \Delta_+(S)=\prod_{i,j\in S, i<j} 1-y_iy_j.$$
Using the induction assumption for diagrams with less than $r$ rows, we obtain from Eq \ref{recursion2} for $r$ odd that
$$q_{(\la_1,\la_2\ .... \la_r)}\ =\ \prod_{j=1}^r q_{(\la_j)}\ \sum_{j=1}^r  
(-1)^{j-1}\ \frac{\Delta([1,r]_{\Z}\backslash \{j\})}{\Delta_+([1,r]_{\Z}\backslash \{j\})}.$$
Let the symmetric group $S_r$ act on the right hand side via permutation of the variables $y_j$. Then we observe that
the permutation $(i,i+1)$ changes the sign of each summand belonging to $j\not\in \{ i,i+1\}$, and it permutes the summands
(without the factor $(-1)^{j-1}$) with indices $i$ and $i+1$. Hence applying a permutation $w$ to the right hand side results in multiplying it by its sign $\ep(w)$.
Observe that the common denominator is $\Delta_+([1,r]_{\Z})$, which is a symmetric function. 
Hence its numerator must be an antisymmetric function of degree $\binom{r}{2}$, and is therefore divisible by $\Delta([1,r]_{\Z})$.
Equality follows from the fact that its leading term $y_1^{r-1}y_2^{r-2}\ ...\ y_{r-1}$ has coefficient 1.
The claim for $r$ even is proved by a similar argument, using relation \ref{recursion3}.
\end{proof}

For this paper we are particularly interested in the case that $q^N = i$. In this setting we can prove some useful equations on the values $q_\la$.

\begin{cor}\label{Nadcor}
Assume $q$ is a primitive $4N$-th root of unity such that $q^N=i$.

\begin{enumerate}[(a)]
    \item The quantity $q_\la$ is nonzero for any strict Young diagram with $\la_1< N$, and it is equal to 0 for any strict Young diagram with $\la_1=N+1$ and $\la_2\leq N-2$.

    \item $a_{N-m}=a_m$ for $0\leq m\leq N$, where $a_0=1$.

    \item $q_{(N,\la_2,\ ..., \la_r)}=q_{(\la_2,...\la_r)}$ for any strict Young diagram $(N,\la_2,\ ...\ \la_r)$.

\end{enumerate}

\end{cor}

\begin{proof}  Part (a) is an easy consequence of Theorem \ref{qdimensionsGn}. We can assume $m\leq N/2$ for part (b). We then obtain
$$\frac{a_{N-m}}{a_m}\ =\ \prod_{j=m+1}^{N-m} \frac{q^{N+1-j}-q^{j-1-N}}{q^j-q^{-j}}\ =\ 1,$$
as the assignment $j\mapsto N+1-j$ maps the set $\{ m+1, m+2,\ ...\ N-m\}$ to  itself.
Part (c) follows from
\begin{align}
q_{(N,\la_2,\ ...\ r)}\ &=\ q_{(\la_2,\ ...\ r)}\ \prod_{j=2}^r\frac{[N-\la_j]}{[N+\la_j]}\prod_{j=1}^Ni\frac{q^{j-1}+q^{1-j}}{q^j-q^{-j}}\\
&=q_{(\la_2,\ ...\ r)} \prod_{j=1}^N\frac{-q^{-N}q^{j-1}+q^Nq^{1-j}}{q^j-q^{-j}}\\
&=q_{(\la_2,\ ...\ r)} \prod_{j=1}^N\frac{q^{N+1-j}-q^{-(N+1-j)}}{q^j-q^{-j}}= q_{(\la_2,\ ...\ r)}.
\end{align}
where we used $q^N=i$, $[N+j]=[N-j]$, and the bijection $j \mapsto N+1-j$ above in the special case of $m=0$.

\end{proof}


\subsection{Matrix coefficients}\label{matrixcoeff}
Explicit representations of the algebra $G_n$ were obtained in \cite{JN}. In this section we summarize some of their results and clarify when the denominators vanish.

Let $[n]_q = (q^n-q^{-n})/(q-q^{-1})$ denote the usual quantum numbers. If we drop the subscript, we mean $[n] = [n]_q$, but we will also have need to refer to $[n]_{q^2}$.

The representations of Jones-Nazarov are defined over an algebraic extension of the field of rational functions, more explicitly their representations of $G_n$ depend on a choice of square root $\sqrt{[m+1]_{q^2}[m]_{q^2}}$ for all $m < n$. 
We will need in Proposition~\ref{prop:GrothendieckDeform} to make these choices in a continuously varying way in a certain neighborhood of $1$, but otherwise the details of the square root won't matter. We remark that a different choice of square roots \textit{a-priori} could result in a different parameterisation of the representations we define in Section~\ref{sec:Gnreprootofunity} of $G_n$ and $G_n[0]$. It is in fact possible to prove that the isomorphism class of these representations are in fact invariant under the Galois actions $\sqrt{[m+1]_{q^2}[m]_{q^2}}\mapsto -\sqrt{[m+1]_{q^2}[m]_{q^2}}$. We neglect to include this proof, as the result is not necessary for this paper.

To define these representations, we will need the following quantities. For each $m\in \mathbb{N}_{\geq 0}$ we define
\[  x_m:=[m+1]_{q^2}-[m]_{q^2}-(q-q^{-1}) \sqrt{[m+1]_{q^2}[m]_{q^2}}.    \]
Fix a (shifted) standard tableau $\La\in\S_\la$. The following quantities
$q_k^\La$ and $\beta^\La$ depend on $\La$. But we will usually suppress the index $\La$.
The quantity $q_k=q_k^\La$ is defined in \cite{JN} (4.6) by
\begin{equation}\label{def:qk}
q_k\ :=\ x_{m_k}
\end{equation}
with $m_k=j_k-i_k$ where $(i_k,j_k)$ are the coordinates of the box which contains the number $k$ in the given standard tableau
$\La$. In circumstances where $k$ is fixed, we will drop the index $k$ and write $q_k = x_m$. 

When $q_k \neq q_{k+1}^{\pm 1}$,
the quantity $\beta_k = \beta_k^\La$ is defined in \cite{JN} above (6.4) by
$$\beta_k=1\ -\ (q-q^{-1})^2 \left( \frac{q_{k+1}^{-1}q_k}{(q_{k+1}^{-1}q_k-1)^2}+ \frac{q_{k+1}q_k}{(q_{k+1}q_k-1)^2}\right).$$


 It will be convenient to give
somewhat more explicit expressions for some of their matrix coefficients. Let $$\{m\}\ =\ \frac{q^{2m+1}+q^{-2m-1}}{q+q^{-1}}.$$

\begin{lem}\label{xmdef1} 
We have the following.
\begin{enumerate}[(a)]
\item $x_m$ satisfies the quadratic polynomial $x_m^2 -2\{m\}x_m +1 =0$. So, $x_m \neq 0$ and $x_m+x_m^{-1}=2\{ m\}$. Moreover, we can also write $x_m=\{m\} -(q-q^{-1}) \sqrt{[m+1]_{q^2}[m]_{q^2}} =\{m\} -(q-q^{-1})  \sqrt{ \{m\}^2-1  }$. 

 \item  We have $x_m=x_n^{\pm 1}$ only if $\{m\}=\{n\}$, which is equivalent to $q^{2(n-m)}=1$ or $q^{2(n+m+1)}=1$.

 \item   We have $x_m =  1$ if and only if $m=0$, $q^m=1$, or $q^{m+1}=1$, and $x_m =  -1$ if and only if $m=-1$, $q^{4m}=1$, or $q^{4m+4}=1$.
\end{enumerate}
\end{lem}
\begin{proof}
Part (a) is a straightforward calculation. 

For part (b), the first statement follows from part (a), and the second
equivalence follows from 
\begin{equation} \label{eqmn}
    \{m\}-\{n\}\ =\ \frac{(q^{m+n+1}-q^{-m-n-1})(q^{m-n}-q^{n-m})}{q+q^{-1}}.
\end{equation} 

 For (c) we observe from part (a) that $x_m = \pm 1$ only if $\{m\} = \pm 1$. We can factor
 \begin{align*}
     q^{2m+1} + q^{-2m-1} - q - q^{-1} &= q^{-2 m-1} \left(q^m-1\right) \left(q^m+1\right) \left(q^{m+1}-1\right) \left(q^{m+1}+1\right)\\
     q^{2m+1} + q^{-2m-1} + q + q^{-1} &= q^{-2 m-1} \left(q^{2 m}+1\right) \left(q^{2 m+2}+1\right).
 \end{align*}

\end{proof}

\begin{lem}\label{xmdef2}
    Fix a tableau $\Lambda$ and an index $k$, and let $a = m_k = j_k -i_k$, $b =  m_{k+1} = j_{k+1} -i_{k+1}$.
\begin{enumerate}[(a)]
 \item  The quantity $\beta_k$ does not change if we replace $q_k$ by $q_k^{-1}$ or $q_{k+1}$ by $q_{k+1}^{-1}$.
Moreover we can write
$$\beta_k\ =\ \frac{[a+b+2][a+b][a-b+1][a-b-1]}{[a+b+1]^2[a-b]^2}.$$
\item Suppose that $q$ is not a root of unity, then $q_k \neq q_{k+1}^{\pm 1}$ (and thus $\beta_k$ is defined).
\end{enumerate}
\end{lem}

\begin{proof}
For part (a) observe that
$$  \frac{q_{k+1}^{-1}q_k}{(q_{k+1}^{-1}q_k-1)^2}+ \frac{q_{k}q_{k+1}}{(q_kq_{k+1}-1)^2}{}\ =
\ \frac{(q_k+q_k^{-1})(q_{k+1}+q_{k+1}^{-1})-4}{(q_k+q_k^{-1}-q_{k+1}-q_{k+1}^{-1})^2}.$$
This shows that $\beta_k$ is invariant under the changes $q_k\leftrightarrow q_k^{-1}$ and/or 
$q_{k+1}\leftrightarrow q_{k+1}^{-1}$. 
Let now $a=m_k$ and $b=m_{k+1}$. 
Then we can write the quantity $\beta_k$ as
$$\beta_k\ =\ 1-(q-q^{-1})^2\ \frac{\{ a\} \{ b\}-1}{(\{ a\} -\{ b\})^2}.$$
We obtain the claimed expression for $\beta_k$ using the identity \ref{eqmn} for $\{a\}-\{b\}$ and
$$\left(\frac{q+q^{-1}}{q-q^{-1}}\right)^2(\{a\}\{b\}-1\})\ =\ [a+b+1]^2+[a-b-1][a-b+1].$$

 By Lemma \ref{xmdef1}(b), $q_kq_{k+1}^{\pm 1}-1=0$ if an only if $q^{2(m_k-m_{k+1})}=1$ or $q^{2(m_k+m_{k+1}+1)}=1$. If $q$ is not a root of unity, this is equivalent to $m_k-m_{k+1}$ or $m_k+m_{k+1}+1 = 0$. The first case happens iff the $k$ and $k+1$st boxes are added to the same diagonal, which contradicts strictness. The second case is impossible since $m_k$ and $m_{k+1}$ are both non-negative in a shifted tableau.
 
\end{proof}


The following lemma will be useful for constructing representations for $q$ a root of unity. 

\begin{lem}\label{betazero} Let $q$ be a primitive $4N$-th root of unity with $q^N = \mathbf{i}$. Then we have

\begin{enumerate}[(a)]
    \item $x_N=x_{N-1}=- 1$.  


     \item  The quantity $\beta_k=0$ if $m_k+m_{k+1}+2=2N$. 
\end{enumerate} 
\end{lem}
\begin{proof}
Since $\{N\}=\{N-1\} = -1$, Lemma \ref{xmdef1}(a) tells you that $x_N = x_{N-1} = -1$. 
Part (b) follows from Lemma \ref{xmdef2}(a) since $[m_k+m_{k+1}+2] = 0$.

\end{proof}

\subsection{Representations of $G_n$}\label{Gnrep} The representations of the algebras
$G_n$ in \cite{JN} were found in the context of an affine version of $G_n$,
where eventually the affine generators $X_k$ were specialized to Jucys-Murphy elements
$J_k\in G_n$. The elements
$J_k$ are defined  inductively by $J_1=1$ and
\begin{equation}\label{jkdef}
J_k=(t_{k-1}+(q-q^{-1})v_{k-1}v_k)J_{k-1}t_{k-1}\quad {\rm for\ } k>1,
\end{equation}
see \cite{JN} (3.10).
It is shown in \cite{JN} Prop. 3.5 that the elements $J_k$ commute with each other. We define for a fixed
standard tableau $\La\in\S_\la$ the character $\Omega^\La$ on the abelian algebra $A_n=\langle J_k, 1\leq k\leq n\rangle$ by
$$\Omega^\La(J_k)\ =\ q_k.$$
We define for each strict Young diagram $\la$ with $n$ boxes
the vector space $V_\la$ via a basis 
\begin{equation}\label{vladef}
\B_\la\ =\ \{ v\psi_\La,\ \La\in \S_\la,\ v=v_{i_1}v_{i_2} ... v_{i_r}\in \mathrm{Cliff}(n), 1\leq i_1<i_2<\ ...\ <\ i_r\leq n \}.
\end{equation}
The elements $\psi_\La$ were defined in \cite{JN}. For our purposes, it is enough to consider
$\B_\la$ just as a convenient labeling set. The action of the $J_k$ is defined by
\begin{equation}\label{jkrelation}
J_kv\psi_\La\ =\ \Omega^\La(J_k)^{\nu(k)}v\psi_\La,
\end{equation}
where $\nu(k)=\nu_v(k)=\pm 1$ depends on whether $v=v_{i_1}v_{i_2}\ ...\ v_{i_r}$ contains the factor $v_k$ (for $+1)$
or not (for $-1$), see \cite{JN}, above Theorem 6.2. The action of $t_k$ on $\psi_\La$ is defined by
\begin{equation}\label{tkrep}
t_k\psi_\La\ =\ \hbeta \psi_{s_k(\La)} -\frac{q-q^{-1}}{q_k^{-1}q_{k+1}-1}\psi_\La+\frac{q-q^{-1}}{q_kq_{k+1}-1}v_kv_{k+1}\psi_\La,
\end{equation}


where $\hbeta=\beta_k$ if $s_k(\La)<\La$,  $\hbeta=1$  if $s_k(\La)>\La$ and $\hbeta=0$  if $s_k(\La)$ is not a standard tableau. Here we use the partial ordering from Definition~\ref{def:parOrd}.

We then have:

\begin{prop}\label{prop:gkaction}\cite[Theorem 6.2]{JN} Let $q\in \mathbb{C} $ be not a root of unity and let $\la$ be a strict Young diagram with $|\la|=n$. Then there is a unique $G_n$ module $V_\la$ with basis $\B_\la$ with the action of $t_k$ defined as in \ref{tkrep}.
\end{prop}
\begin{proof}
    The proof that the operators $t_k$ and $v_j$ satisfy the defining $G_n$ relations over generic $q$ is exactly \cite[Theorem 6.2]{JN}. We then have from Lemma~\ref{xmdef2}(b) that these operators are well-defined at $q$ not a root of unity.
\end{proof}


Representations $V_\lambda$ are not irreducible in general. It is shown in \cite{JN} for generic $q$ that $\End_{G_n}(V_\la)\cong\mathrm{Cliff}(\ell(\la))$. For the case that $q$ is specialised to a complex number we only know \textit{a-priori} that
\[   \mathrm{Cliff}(\ell(\la)) \subseteq \End_{G_n}(V_\la). \]
In Section~\ref{sec:Gnreprootofunity} we will show that for $q$ not a root of unity, or $q = \pm 1$ that the above is an isomorphism. Furthermore, in the case that $q$ is a primitive $4N$-th root of unity we will construct additional commuting operators generating the entire endomorphism algebra.

\begin{defn}\label{def:weights}
    Let $V$ be a $G_n$ module, and let $A_n$ be the abelian algebra generated by $J_k$, $1\leq k\leq n$. We call a character $\Omega$ of $A_n$ a weight if there exists $0\neq v\in V$ such that $av=\Omega(a)v$ for $a\in A_n$. Weight spaces and multiplicity of a weight are defined as usual. We denote by $P(\la)$ the set of all weights of  the $G_n$ module $V_\la$. As $A_n\subset G_n[0]$, the same definitions apply to $G_n[0]$ modules.
\end{defn}
\begin{lem}\label{lem:weights}
Let $q\in \mathbb{C}$ be not a root of unity, let $\lambda, \mu$ be Young diagrams with $|\lambda| = n = |\mu|$, let $\Lambda \in S_\lambda$, $\Gamma \in S_\mu$, and let $\nu_1, \nu_2\in \{-1,1\}^n$. Then 
\[ ((q_k^\Lambda)^{\nu_1(k)})_{k=1}^n =  ((q_k^\Gamma)^{\nu_2(k)})_{k=1}^n     \]
only if $\Lambda = \Gamma$. In particular $P(\la)\cap P(\mu)=\emptyset$ if $\lambda \neq \mu$. 

\end{lem}


\begin{proof} 

    Let $\Lambda \in S_\lambda$. It follows from Lemma~\ref{xmdef1}(b) that $(q^\Lambda_k)^{\pm 1}$ determines $m^\Lambda_k$ uniquely. It is a straightforward proof by induction on the number $r$ of boxes that there exists at most one shifted strict tableau $\La$ for which the number $k$ sits in the box $(i_k,j_k)$ such that $j_k-i_k=m_k$ for $1\leq k\leq r$. Hence the sequence $(q_j^\Lambda)$ uniquely determines $\Lambda$.
\end{proof}

\subsection{Classical limit $q\to 1$}\label{sec:classlimit} 
The representations in \cite{JN} for $G_n[0]$ are $q$-deformations of representations of $G_n[0]$ at $q=1$, constructed in \cite{Nazarov}. This can be seen by the following limit calculation.

\begin{prop}
    Let $\Lambda$ be a shifted strict Young tableaux, and set $a = m_k$ and $b = m_{k+1}$ as in Lemma~\ref{xmdef2}. We have
    \[\lim_{q\to 1}\ -\frac{q-q^{-1}}{q_k^{-1}q_{k+1}-1}\ =\ \frac{-1}{\sqrt{a(a+1)}-\sqrt{b(b+1)}}.\]
\end{prop}
\begin{proof}
We can write $q_k=x_m=\{m\} -(q-q^{-1}) \sqrt{[m+1]_{q^2}[m]_{q^2}}$, see Lemma \ref{xmdef1}. Here we assume the branch of the square root  such that
$\lim_{q\to 1}\sqrt{[m+1]_{q^2}[m]_{q^2}} \ =\ \sqrt{m(m+1)}$ for $m$ a positive integer.
 Using the notations as in Lemma \ref{xmdef2}, we observe that
$$ \frac{q-q^{-1}}{q_k^{-1}q_{k+1}-1}\ =\ \frac{q-q^{-1}}{(\{a\}+(q-q)^{-1}\sqrt{[a+1]_{q^2}[a]_{q^2}})(\{b\}-(q-q)^{-1}\sqrt{[b+1]_{q^2}[b]_{q^2}})-1}\ =\ $$
$$=\ \frac{1}{\{b\}\sqrt{[a+1]_{q^2}[a]_{q^2}})-\{a\}\sqrt{[b+1]_{q^2}[b]_{q^2}}},$$
where we used  $\lim_{q\to 1}\frac{\{a\}\{b\}-1}{q-q^{-1}}=0$, which follows from
the last display in the proof of Lemma \ref{xmdef2}.
Using this and $\lim_{q\to 1} \{a\}=1=\lim_{q\to 1} \{b\}$, we obtain
  $$\lim_{q\to 1}\ -\frac{q-q^{-1}}{q_k^{-1}q_{k+1}-1}\ =\ \frac{-1}{\sqrt{a(a+1)}-\sqrt{b(b+1)}}.$$
  \end{proof}
We now compare \ref{tkrep} with \cite{Nazarov}, (7.10-13) where 
  the quantities $(j-i)$ and $(j'-i')$ correspond to $a$ and $b$ here. 
  Similarly, one shows that the coefficients of $v_kv_{k+1}\psi_\La$ and $\psi_{s_k(\La)}$ in \ref{tkrep} converge to the corresponding coefficients in \cite{Nazarov}, (7.11).

\section{Representations of $G_n$ at $q\in \mathbb{C}$}\label{sec:Gnreprootofunity}

The goal of this section is to extend the results of \cite{JN} to study the representation theory of $G_n$ and $G_n[0]$ at specialised values of $q\in \mathbb{C}$. There is a bifurcation here depending on if $q$ is a root of unity or not. In the case that $q$ is not a root of unity, we are able to make minor adjustments to the results of \cite{JN} to show that $G_n$ and $G_n[0]$ are semisimple, and obtain a classification of the irreducible representations of these algebras. The case where $q$ is a root of unity is significantly more difficult, as the representations defined in \cite{JN} are not well-defined in general, and we have to make major adjustments to obtain irreducible representations of $G_n$ and $G_n[0]$.

We begin with the case where $q$ is a primitive $4N$-th root of unity. The case of $q$ not a root of unity will fall to much easier arguments. We note that we do not consider the case where $q$ is a root of unity with order is non-zero modulo $4$, in those cases we expect that the corresponding category $\overline{\mathcal{E}_q}$ is not semisimple.


\subsection{Representations of $G_n$ at $q$ a primitive $4N$-th root of unity}

Throughout this subsection we will be working with $q$ a primitive $4N$-th root of root unity. We will overload $G_n$ and $G_n[0]$ throughout to denote the specializations of these algebras at these values, along with the generic algebras. It will be clear from context which algebra is being referred to. 

It is clear from the explicit formulas for $q_k$ and $\beta_k$ in Section~\ref{matrixcoeff} 
that the representations in Proposition \ref{prop:gkaction} are not well-defined in general if $q$ is a root of unity. In this section, we make some adjustments to obtain representations of $G_n$ for $q$ a primitive $4N$-th root of unity. Tableaux can be thought of as walks on a graph whose vertices are diagrams and whose edges add boxes. At a root of unity we will consider walks on a new graph, which can be thought of as an SL(N) version of the graph of used for generic $q$. This is similar to the approach for representations of Hecke algebras at roots of unity in \cite{HansThesis}.

\subsection{Restricted diagrams and tableaux} 

\begin{defn} Let $N\in \mathbb{N}\cup \{\infty\}$ and define $ Y^<_N$ as the set of all strict Young diagrams with $\la_1<N$. We define the stable $N$-graph $\Galnk$ to be the directed graph whose vertices are the Young diagrams in $ Y^<_N$ and where there's an edge  $\la\ \to^{(N)}\ \mu$
if either:

\begin{itemize}
    \item $\mu$ can be obtained by adding a single box to $\la$, or
    \item $\la_1=N-1$ and $\mu = [\la_2,\la_3\ ...\ \la_{\ell(\la)}]$.
\end{itemize}

We call the former kind of edge a generic edge and the latter a restricted edge. Note that in the $N = \infty$ case there are no restricted edges.

\end{defn}

We now define the vector spaces on which $G_n$ at a specialised $q$ will act.
\begin{defn}
Let $\lambda$ a strict Young diagram, and let $n\in \mathbb{N}$. Let $\Slnk$ denote the set of all paths $\La$ of length $n$ from $[0]$ to $\la$ in the graph $\Galnk$. If we let $f(\la)$ denote the number of restricted edges in a path $\Gamma$, then we have that $n=|\la|+f(\la)N$. Note that this formula is independent of the path $\Gamma$. Also note that in the case $N=\infty$ 
, we have that $n = |\lambda|$, and that $\Slnk$ is identified with the set of standard tableaux on $\lambda$.

 Let $\Vlnk$ be the vector space spanned by the set $\{ v\psi_\La,\ v\in \mathrm{Cliff}(n), \La\in\Slnk\}$. Note that for each $\La\in\Slnk$ we have at most one additional path
\[  \emptyset \to^{(N)} \Lambda(1) \to^{(N)} \cdots \Lambda(k-1) \to^{(N)} x  \to^{(N)} \Lambda(k+1) \to^{(N)} \cdots \in  \Slnk\]
with $x\neq \Lambda(k)$. If such a new path exists, we denote it by $s_k(\La)$. If there is no such new path, we define $s_k(\La)=0$. If the $k$-th edge is restricted, and the $(k+1)$-st edge is not, we define $s_k(\La)<\La$. If both edges are generic, then we define $<$ as before, i.e. $x$ and $\Lambda(k)$ are obtained by adding boxes to $\Lambda(k-1)$ in different rows and the path using the smaller row is smaller.
\end{defn}




We can now define the action of the generators of $G_n$ on the space $V_\lambda^{(n)}$.





\begin{lem} \label{lem:betaWellDefined}
   Let $N\in \mathbb{N}_{\geq 2}$, $q$ a primitive $4N$-th root of unity, $\lambda \in Y_N^<$, let $n\geq |\la|$ an integer such that $n \equiv |\lambda| \pmod N$, and let $1 \leq k \leq n$ be an integer, then $q_k \neq q_{k+1}^{\pm 1}$, and thus $\beta_k$ is defined. 
\end{lem}
\begin{proof}
    By Lemma \ref{xmdef1}(b), $q_k \neq q_{k+1}^{\pm 1}$ iff $q^{2(m_k-m_{k+1})}=1$ or $q^{2(m_k+m_{k+1}+1)}=1$, which in turn is equivalent to $m_k \equiv m_{k+1} \pmod {2N}$ or $m_k+m_{k+1} + 1 \equiv 0 \pmod {2N}$.
    Since $\la \in Y_N^<$, $m_k$ and $m_{k+1}$ lie between $0$ and $N$, and since $\la$ is strict, they are distinct. The result follows.
\end{proof}
With this lemma in hand we can now define representations of $G_n$ at $4N$-th roots of unity. This is done in an analogous manner to the generic $q$ case. The only major difference to the generic case is that we now have to consider restricted edges in our actions.

\begin{defn}\label{def:gnreproot}
With the same notation as in the previous lemma, 
let $\Vlnk$ be a vector space with basis
$v\psi_\La=v_{i_1}v_{i_2}\ ...\ v_{ir}\psi_\La$, $\La\in \Slnk$. In the case that the $m$-th edge $\Lambda$ is restricted, we define $q^\Lambda_m=-1$. For the generic edges, we remind the reader that $q_k = x_{m_k}$ where $m_k = j_k-i_k$ and $(i_k,j_k)$ is the position of the box added in the $k$-th edge of the shifted Young diagram $\lambda$. 

Now we define the action of $v_k$ to be free, and the $t_k$ action is defined by 
\[
t_k\psi_\La\ =\ \hbeta \psi_{s_k(\La)} -\frac{q-q^{-1}}{q_k^{-1}q_{k+1}-1}\psi_\La+\frac{q-q^{-1}}{q_kq_{k+1}-1}v_kv_{k+1}\psi_\La,\] 
where $\hbeta=\beta_k$ if $s_k(\La)<\La$, $\hbeta=1$  if $s_k(\La)>\La$ and $\hbeta=0$  if $s_k(\La) = 0$.



In particular, the action of $t_k$ preserves 
the span of $\{  \psi_\La, v_k v_{k+1}\psi_\La, \psi_{s_k(\La)}, v_kv_{k+1}\psi_{s_k(\La)}\}$, with matrix (in the case of $s_k(\Lambda)>\Lambda$) being
\[  \begin{bmatrix}
    -\frac{q-q^{-1}}{q_k^{-1}q_{k+1} -1} & \frac{q-q^{-1}}{q_kq_{k+1} -1}- (q-q^{-1}) &  \widehat{\beta_{k}}&0\\
    \frac{q-q^{-1}}{q_kq_{k+1} -1} & \frac{q-q^{-1}}{q_k^{-1}q_{k+1} -1} + (q-q^{-1})  & 0 & -\widehat{\beta_{k}}\\
    1&0& -\frac{q-q^{-1}}{q_kq_{k+1}^{-1} -1}&\frac{q-q^{-1}}{q_kq_{k+1} -1}- (q-q^{-1})\\
    0&-1& \frac{q-q^{-1}}{q_kq_{k+1} -1} &\frac{q-q^{-1}}{q_kq_{k+1}^{-1} -1}+(q-q^{-1})\\
\end{bmatrix}.           \]
The matrix in the case of $s_k(\Lambda)<\Lambda$ is nearly the same, with the off-diagonal blocks exchanged. More generally the action of $t_k$ preserves the subspace $\{  v\psi_\La, v_k v_{k+1}v\psi_\La, v\psi_{s_k(\La)}, v_kv_{k+1}v\psi_{s_k(\La)}\}$, with matrices given by similar formulas as above.

It is useful to observe that the definition of the action of $t_k$ on $\psi_\Lambda$ is local in the sense that the coefficients of the action only depend on the $k$-th and $k+1$-th edge of $\Lambda$.

\begin{remark}
    Note that the matrix coefficients of $t_k$ are evaluations of a strict subset of the algebraic functions used to define the operator $t_k$ in the generic $q$ setting. This will allow us to apply algebraic geometry style techniques throughout the remainder of the paper to study these representations.
\end{remark}
The first such application of this remark is in showing that $V_\la^{(n)}$ indeed define representations at roots of unity.


\end{defn}

\begin{prop}\label{Vlnkrep} Let $q$ be a primitive $4N$-th root of unity. Then 
$\Vlnk$ is a $G_n$-module, with the action of the generators defined in Definition \ref{def:gnreproot}.
\end{prop}

\begin{proof}


We need to check that the actions of our generators satisfy the defining relations of $G_n$. The commutation relation $t_it_j = t_jt_i$ for $|i-j|\geq 2$ follows immediately from the local definition of the action. The local relations require more work. We will achieve this by locally lifting
restricted edges to generic ones, and then appeal to Proposition \ref{prop:gkaction}
where the relations have been proved in that setting. The only complication comes from the fact that in the generic setting
we may have more tableaux than in the restricted setting, and we will need to check that the terms contributed by these extra tableaux vanish when we specialize $q$ to a root of unity.
We will do this in detail for the most complicated
relations, checking the braid relations for  $t_k$ and $t_{k+1}$. We note that with the other (easier to verify) relations in hand, it suffices to check this relation on the vectors $\psi_\Lambda$. For this we need to consider the $k$-th, $(k+1)$-st and $(k+2)$-nd edges of $\La$. There are several cases to consider, depending on whether these edges are restricted or not.

We begin with the most difficult case, when two edges are restricted. As we can not have two consecutive restricted edges for any $\La\in\Slnk$, we only have to consider the following, where $\tilde\la=[\la_3,\la_4, ...]$:
$$  [N-1,N-2,\tilde\la ]\ \to^{(N)}\  [N-2,  \tilde\la]\ \to^{(N)}\   [ N-1, \tilde\la] \ \to^{(N)}\      \tilde\la.$$
From Definition~\ref{def:gnreproot} we have that $t_k$ and $t_{k+1}$ preserve the subspace 
\[V:= \operatorname{span}_{\mathbb{C}}\{ \psi_\La, v_kv_{k+1}\psi_\La , v_kv_{k+2}\psi_\La, v_{k+1}v_{k+2}\psi_\La    \}\]
with matrices 
\begin{align*} t_k|_V &=  \begin{bmatrix}
    -\frac{q-q^{-1}}{-x_{N-2} -1} & \frac{q-q^{-1}}{-x_{N-2} -1}- (q-q^{-1}) & 0&0\\
    \frac{q-q^{-1}}{-x_{N-2} -1} & \frac{q-q^{-1}}{-x_{N-2} -1} + (q-q^{-1})  & 0 & 0\\
   0&0  &  -\frac{q-q^{-1}}{-x_{N-2} -1} &-\frac{q-q^{-1}}{-x_{N-2} -1} - (q-q^{-1})\\
   0&0 & -\frac{q-q^{-1}}{-x_{N-2} -1}&\frac{q-q^{-1}}{-x_{N-2} -1}+(q-q^{-1})
\end{bmatrix}       \\        
t_{k+1}|_V &=  \begin{bmatrix}
    -\frac{q-q^{-1}}{-x_{N-2}^{-1} -1} & 0& 0&\frac{q-q^{-1}}{-x_{N-2}-1}+(q-q^{-1})\\
   0 &-\frac{q-q^{-1}}{-x_{N-2}-1}  & -\frac{q-q^{-1}}{-x_{N-2}^{-1} -1}-(q-q^{-1})  & 0\\
   0&-\frac{q-q^{-1}}{-x_{N-2}^{-1} -1} & \frac{q-q^{-1}}{-x_{N-2}-1}+(q-q^{-1})&0\\
   \frac{q-q^{-1}}{-x_{N-2}-1}&0 & 0&\frac{q-q^{-1}}{-x_{N-2}^{-1} -1} + (q-q^{-1}) 
\end{bmatrix}    
\end{align*}

We now consider the operators $t_{2N + |\lambda| -2}$ and $t_{2N + |\lambda| -1}$ on the space $V_{ [N+1, N-1, \tilde \la]}$ in the generic setting. By Proposition~\ref{prop:gkaction} we know that these operators satisfy the braid relations. Let $\Lambda'$ be any path in $\mathcal{S}_{[N+1, N-1, \tilde \la]}$ ending in 
\[\cdots \to [N-1,N-2,\tilde\la]\ \to\  [N, N-2, \tilde\la ]\ \to\   [ N, N-1, \tilde\la] \ \to\      [N+1,N-1,\tilde\la].\]
\textit{A-priori} we only know the operators $t_{2N + |\lambda| -2}$ and $t_{2N + |\lambda| -1}$ preserve the subspace spanned by the vectors $\{ v \psi_{\Lambda'}, v \psi_{s_{2N + |\lambda| -1}(\Lambda')}: v \in \operatorname{Cliff}( v_{{2N + |\lambda| -2}},v_{2N + |\lambda| -1},v_{{2N + |\lambda| }})[0]\}$. This is due to the existence of the second path $s_{2N + |\lambda| -1}(\La')$ from $[N, N-2,\tilde\la ]$ to $ [N+1,N-1,\tilde\la ]$ via the diagram $[N+1,N-2,\tilde\la ]$. However, for the path $\Lambda'$ we have $m_{2N + |\lambda| -1}=N-2$ and $m_{2N + |\lambda|}=N$. Hence $\beta_{2N + |\lambda| -1}=0$, by Lemma~\ref{betazero}. Hence the space $V'$ spanned by $\{ v \psi_{\La'}, v\in \operatorname{Cliff}( v_{{2N + |\lambda| -2}},v_{2N + |\lambda| -1},v_{{2N + |\lambda| }})[0]\}$ is preserved under the action of both $t_{2N + |\lambda| -2}$ and $t_{2N + |\lambda| -1}$.

Direct computation shows that for the path $\Lambda'$ we have $q_{2N+|\la| -2} = x_{N-1}$, $q_{2N+|\la| -1} = x_{N-2}$, and $q_{2N+|\la| } = x_{N}$. By Lemma~\ref{betazero} we have $x_{N} = -1 = x_{N-1}$. It follows from the definition that the matrices $t_{2N + |\lambda| -2}|_{V'}$ and $t_{2N + |\lambda| -1}|_{V'}$ agree with the matrices $t_{k}|_V$ and $t_{k+1}|_V$ respectively. As the operators $t_{2N + |\lambda| -2}$ and $t_{2N + |\lambda| -1}$ satisfy the braid relation, we obtain that $t_{k}|_V$ and $t_{k+1}|_V$ satisfy the braid relation. In particular, $t_kt_{k+1}t_{k}\psi_\La = t_{k+1}t_{k}t_{k+1}\psi_\La$ as desired.

The remaining cases fall to similar logic. If the $k$-th, $(k+1)$-st and $(k+2)$-nd edges of $\La$ are all generic, then we consider the operators $t_{|\La(k+2)|-2}$ and $t_{|\La(k+2)|-1}$ on $V_{\La(k+2)}$, and $\Lambda'$ any path ending in $\La(k-1) \to\La(k) \to \La(k+1) \to \La(k+2)$. If one of these edges is restricted, then we consider the operators $t_{N+|\La(k+2)|-2}$ and $t_{N+|\La(k+2)|-1}$ on $V_{[N,\La(k+2)]}$, and $\Lambda'$ the join of any path to $\La(k)$, and the generic path obtained from $\La(k-1) \to\Lambda(k)\to \Lambda(k+1)\to \Lambda(k+2)$ by padding the Young diagrams appearing after the restricted edge with a full row of $N$ boxes.



The other local relations are checked similarly, but the subspaces where the calculations take place are smaller dimensional.
\end{proof} 

\subsection{The commutant and weight spaces of $\Vlnk$}\label{decompositionsection} We generalize the definition of commuting operators on the module
$V_\la$ in \cite{JN} to our module $\Vlnk$ where we obtain additional operators in the root of unity setting.

Consider a path $\La\in\Slnk$ of length $n$ ending at $\lambda$. Then $\La$ has $f(\la) = (n-|\la|)/N$ restricted edges, and $f(\la)+\ell(\la)$ edges where we add a box for the first time in a new row. We define
$$\llan=\ell(\la)+2f(\la).$$

\begin{remark} \label{rem:ModifiedCounting}
If you follow along a path, $\llan$ counts the number of steps where you add the first box to a new row plus the number of steps along restricted edges. In the latter case, $\ell$ decreases by $1$ but $2f$ increases by $2$ to compensate.
\end{remark}


\begin{defn}\label{def:coms}

    Define linear maps $\kappa_i$, $1\leq i\leq f(\la)+\ell(\la)$ and $\tilde\kappa_j$,  $1\leq j\leq f(\la)$ from $\Vlnk$ to itself
    by
\begin{equation}\label{rhodef}
\kappa_iv\psi_\La=vv_{d(i)}\psi_\La,\hskip 3em \tilde\kappa_jv\psi_\La=vv_{\tilde{d}(i)}\psi_\La,
\end{equation}
where $d(i)$ and $\tilde{d}(i)$ are the numbers such that
 $\La(d)\to^{(N)} \La(d+1)$ results in adding a box in a new row for the $i$-th time,
and $\La(\tilde d) \to^{(N)} \La(\tilde d+1)$ results in
removing a row of $N-1$ boxes for the $i$-th time.
Finally, we define $\kappa_0$ by $\kappa_0v\psi_\La=\al(v)\psi_\La$, where $\al(v)=\pm v$ depending on the degree of $v$, see Equation~\ref{autom}.
\end{defn}


\begin{prop}\label{commutant}

Let $q$ be a primitive $4N$-th root of unity. Then the map
\[  \rho(v_i) := \begin{cases}  \kappa_i &\quad  1\leq i \leq f(\lambda) + \ell(\lambda)\\
\tilde\kappa_{i - f(\lambda) - \ell(\lambda)} &\quad  i > f(\lambda) + \ell(\lambda)
\end{cases}
:\mathrm{Cliff}\left(\llan\right) \to \End_{G_n}\left(\Vlnk\right)     \]
is an injective algebra homomophism. Further, this map extends to an injective algebra homomorphism
\[  \rho_0 :  \mathrm{Cliff}\left(\llan+1\right) \to \End_{G_n[0]}\left(\Vlnk\right)   \]
where the additional generator is mapped to $\kappa_0$.
\end{prop}

\begin{proof}  We first claim that each of the $\kappa_i$ and $\tilde\kappa_j$ are elements of $\End_{G_n}\left(\Vlnk\right)$. This is mostly analogous to the proof of \cite[Proposition 6.3]{JN} for $G_n$. The nontrivial part consists of showing that each $\kappa_i$ and $\tilde\kappa_j$ commutes with the action of  all the $t_k$.

We first look at the case of the $\kappa_i$. Since left multiplication commutes with right multiplication, the $\kappa_i$ clearly commute with the action of the Clifford algebra. Thus it is enough to check that $\kappa_i (t_k \psi_\Lambda) = t_k \kappa_i(\psi_\Lambda) = t_k v_{d(i)} \psi_\Lambda.$ This is easy to see if $k \neq d(i),\; d(i)-1$ from the definition of the action and using that distant generators commute.

We now consider the case where $k=d(i)$. The first time a new box is added in each row, it will occur along the main diagonal in the shifted diagram. Thus, $m_k = 0$ and $q_k=1$. We can now follow the calculation in \cite[Proposition 6.3]{JN}, which relies only on $q_k=1$. Similarly, when $k = d(i)-1$ we have that $q_{k+1} = 1$ so again we can follow \cite[Proposition 6.3]{JN}.

Proving the claim for $\tilde\kappa_j$ is completely analogous. Again we need only check $\tilde{\kappa}_i (t_k \psi_\Lambda) = t_k \tilde{\kappa}_i(\psi_\Lambda) = t_k v_{\tilde{d}(i)} \psi_\Lambda,$ and again this is easy to see if $k \neq \tilde{d}(i),\; \tilde{d}(i)-1$. For these indices we have $q_k = -1$ and $q_{k+1} = -1$ (respectively), and again the calculation in \cite[Proposition 6.3]{JN} goes through.



It follows that for each subspace $\mathrm{Cliff}(n)\psi_\La$, the $\kappa_i$ and $\tilde\kappa_j$ act via the right multiplication \ref{rhodef}
of $\llan=\ell(\la)+2f(\la)$ different generators $v_j$ of $\mathrm{Cliff}(n)$. Hence they generate an algebra isomorphic to $\mathrm{Cliff}(\llan)$.

It is well known that the centre of $\mathrm{Cliff}(m)$ is trivial if $m$ is even, and generated by $v_1v_2\cdots v_m$ if $m$ is odd. But 
\[\rho(v_1v_2\cdots v_m)\psi_\Lambda = \kappa_1 \ldots \kappa_{f(\lambda)+\ell(\lambda)}\tilde{\kappa}_1 \ldots \tilde{\kappa}_{f(\lambda)}\psi_\Lambda = v_{\tilde{d}_{f(\lambda)}} \ldots v_{\tilde{d}_{1}} v_{d_{\ell(\lambda)+f(\lambda)}} \ldots v_{d_1} \psi_\Lambda, \]
but this is non-zero because all of the $d(i)$ and $\tilde{d}(i)$ are distinct. Hence the map $\rho$ is injective.

It follows from the definitions that $\kappa_o$ commutes with the action of $\Gno$ and that it anti-commutes with $\kappa_i$ and $\tilde\kappa_j$ for $i,j\geq 1$. This proves that we have the required map for $\Gno$. Finally we need to check injectivity, which follows by the same calculation, since $\kappa_0$ will only change the answer by a scalar.
\end{proof}

It will be useful at several points in this section to have an explicit description of the weight spaces of the modules $V_\lambda$. Towards that goal we define the following.

\begin{defn}
   Let $q$ a primitive $4N$-th root of unity, $\lambda$ a Young diagram, and $\Lambda \in \Slnk$ a path of length $n$. We define
    \[     \mathrm{Cliff}_\Lambda := \langle  v_j : q_j^2 = 1\rangle  \subseteq \mathrm{Cliff}(n).    \]
    We define $  \mathrm{Cliff}_\Lambda[0]$ as the even subalgebra of $ \mathrm{Cliff}_\Lambda$.
\end{defn}

The following lemma is an easy consequence of the above definition.

\begin{lem}
    Let $q$ a primitive $4N$-th root of unity, $\lambda$ a Young diagram, and $\Lambda \in \Slnk$ a path of length $n$. Then
    \begin{enumerate}[(a)]
        \item $\mathrm{Cliff}_\Lambda \cong \mathrm{Cliff}({\ell(\lambda)^{(n)}})$, and $\mathrm{Cliff}_\Lambda[0] \cong \mathrm{Cliff}({\ell(\lambda)^{(n)}}-1)$.
        \item Each weight space of $V_\la^{(n)}$ is equal to  $\mathrm{Cliff}_\Lambda v\psi_\Lambda$ for some $v\in \mathrm{Cliff}(n)$.
    \end{enumerate}
\end{lem}
\begin{proof}
    For (a) we observe that $q_k^2 =1$ if and only if the $k$th edge is restricted or $m_k = 0 \mod 2N$.
    Hence we only obtain $q_k^2 =1$ for an edge if the added box is in the first column, or if the edge is a restricted edge. Thus the result follows from Remark \ref{rem:ModifiedCounting}. 

     For (b) let $\Omega  = \{q_k^{\pm 1}\}$ be a weight, and let $W_\Omega\subseteq V_\lambda^{(n)}$ be the corresponding weight space. We first claim that $q_k^{\pm 1}$ uniquely determines $m_k$. For a generic edge this follows from Lemma~\ref{xmdef1}(b), along with the fact that $0 \leq m_k < N-1$, and for a restricted edge (where $q_k = -1$) this follows from Lemma~\ref{xmdef1} (c). Now the same induction as in Lemma~\ref{lem:weights} shows that $\Lambda$ is uniquely determined by $q_k^{\pm 1}$. We thus have that $W_\Omega \subseteq \operatorname{Cliff}(n) \psi_\Lambda$.

    Let $v_\Omega \in \operatorname{Cliff}(n)$ be such that $v_\Omega\psi_\Lambda$ is a weight vector for $\Omega$. It is immediate from the definitions that $\mathrm{Cliff}_\Lambda v_\Omega\psi_\Lambda\subseteq W_\Omega$. From (a) we have that the dimension of $\mathrm{Cliff}_\Lambda v_\Omega\psi_\Lambda$ is $2^{\ell(\lambda)^{(n)}}$. For a fixed path $\Lambda$, there are $2^{n- \ell(\lambda)^{(n)}}$ different possible weights compatible with $\Lambda$. One for each choice of signs in the sequence $\{q_k^{\pm 1}\}$ (recalling that there are $\ell(\lambda)^{(n)}$ cases where $q_k= q_k^{-1}$). A dimension count implies
    \[   \bigoplus_{\Omega} \mathrm{Cliff}_\Lambda v_\Omega\psi_\Lambda \cong \operatorname{Cliff}(n) \psi_\Lambda.    \]
    As the weight spaces corresponding to a fixed $\Lambda$ partition the space $\operatorname{Cliff}(n) \psi_\Lambda$, and as each $\mathrm{Cliff}_\Lambda v_\Omega\psi_\Lambda$ is a subspace $W_\Omega$, it follows that $\mathrm{Cliff}_\Lambda v_\Omega\psi_\Lambda=W_\Omega$.
    
    
    
    
\end{proof}




\begin{cor} \label{cor:WeightCommutant}
The algebra homomorphisms from Proposition~\ref{commutant} restrict to faithful homomorphisms $\rho: \mathrm{Cliff}\left(\ell(\la)^{(n)} \right) \rightarrow \operatorname{End}_{\mathrm{Cliff}_\La}(\mathrm{Cliff}_\La \psi_\Lambda)$ and $\rho_0: \mathrm{Cliff}\left(\ell(\la)^{(n)+1} \right) \rightarrow \operatorname{End}_{\mathrm{Cliff}_\La[0]}(\mathrm{Cliff}_\La[0] \psi_\Lambda)$. 
\end{cor}

\begin{proof}

 In the case that $\ell(\lambda)^{(n)}$ is even, the claim is immediate, as even Clifford algebras are simple. If $\ell(\lambda)^{(n)}$ is odd, a direct computation shows that the central element $\kappa_1\cdots \kappa_{f(\lambda)+\ell(\lambda)}\tilde\kappa_1\cdots \tilde\kappa_{f(\lambda)}$ does not act as a scalar. This implies that the two simple summands of the Clifford algebras also act nonzero. 
\end{proof}
 

\subsection{Submodules of $V_\la^{(n)}$}\label{Ulasection}
Recall that $\mathrm{Cliff}(m)$ is isomorphic to the full matrix algebra  $M_{2^{m/2}}(\C)$ for $m$ even, and it is isomorphic to
$M_{2^{(m-1)/2}}(\C) \oplus M_{2^{(m-1)/2}}(\C)$ for $m$ odd. The 2 irreducible summands for $m$ odd can be distinguished by which eigenvalue
the central element $z_m=v_1v_2\ ...\ v_m$ acts; the eigenvalues are $\pm 1$ for $m\equiv 1 \mod 4$, and $\pm i$ for
$m\equiv 3 \mod 4$.   Moreover, if $m$ is odd and $\tilde e$ is a minimal
idempotent of $\mathrm{Cliff}(m-1)\subset \mathrm{Cliff}(m)$,
$\tilde e$ is a rank 1 idempotent in each of the two summands of $\mathrm{Cliff}(m)$. 

\begin{defn}\label{Udef}
Let $q$ be a primitive $4N$-th root of unity, and let 
\[\rho: \mathrm{Cliff}(\llan)\to \End_{G_n}(\Vlnk)\quad \text{ and } \quad\rho_0:\mathrm{Cliff}(\llan+1)\to \End_{\Gno}(\Vlnk)\]
be the algebra homomorphisms defined in Proposition \ref{commutant}.

\begin{enumerate}[(a)]
    \item We define the $G_n$-modules
$$\ell(\la)\ {\rm even:}\hskip 1em \Ulnk=\rho(e)\Vlnk,\hskip 3em \ell(\la)\ {\rm odd:}\hskip 1em U^{(n)}_{\la,\pm}=\rho( e_\pm)\Vlnk,$$
Here $e$ is a minimal projection in $\mathrm{Cliff}(\llan)$ for $\ell(\la)$ even, and $e_{\pm}$ are inequivalent minimal projections in $\mathrm{Cliff}(\llan)$ for $\ell(\la)$ odd. 

\item We define the $G_n[0]$ modules
$$\ell(\la)\ {\rm odd:}\hskip 1em \Ulnk=\rho(e)\Vlnk,\hskip 3em \ell(\la)\ {\rm even:}\hskip 1em U^{(n)}_{\la,\pm } =\rho(e_\pm)\Vlnk,$$
Here $e$ is a minimal projection in $\mathrm{Cliff}(\llan+1)$ for $\ell(\la)$ odd, and $e_{\pm}$ are inequivalent minimal projections in $\mathrm{Cliff}(\llan+1)$ for $\ell(\la)$ even. 
\end{enumerate} 

\begin{remark}\label{rem:signs}
    It is important to note that for each $\lambda$ and $n$ we have the free symmetry between minimal idempotents $e_{\pm}$ in the above definition (by exchanging the sign by which the central element acts). This means that for each $\lambda$ and $n$ we currently have the freedom to interchange the modules $W_{\la,\pm}^{(n)}$. This freedom will save us significant work. For example, later we will deduce that the cutdown of $\operatorname{End}_{\overline{\mathcal{SE}_N}}(+^N)$ by the projection $p_{[1^N]}$ is isomorphic to $W_{\emptyset,\pm}^{(N)}$ for some sign choice. If we were to fix the scalars by which the $e_\pm$ act, then we would have to determine the scalar by which $e_1\cdots e_{N-1} = v_1v_N$ acts on $p_{[1^N]}$ in $G_N[0]$ to pin down the correct choice.
    For general $N$, this is a rather difficult computation. To avoid this computation (and several others) we will not break the symmetry between the modules $W_{\la,\pm}^{(n)}$ until later. Our choice to not break the symmetry is justified in that our upcoming results on the modules $W_{\la,\pm}^{(n)}$ are independent of the sign choice.

    To illustrate let $N=2$, and suppose we had made the choice that the central element acts on $e_+$ by $\mathbf{i}$. By \cite[Remark 4.16]{ConformalA} we have $-\frac{\mathbf{i}}{\sqrt{2}} e_1 - t_1 + \frac{\mathbf{i}}{\sqrt{2}} = 0$, or equivalently, $e_1 = \mathbf{i} \sqrt{2} t_1 +1$. So letting $D$ denote the determinant map, and recalling that $t_1 D = -q^{-1} D = \frac{-1+\mathbf{i}}{\sqrt{2}} D$, we see $e_1 D = \mathbf{i} \sqrt{2} \frac{-1+\mathbf{i}}{\sqrt{2}} + 1 = - \mathbf{i}$. So in this case we would have $p_{[1^2]}\operatorname{End}_{\overline{\mathcal{SE}_N}}(+^N) \cong W_{\emptyset,-}^{(2)}$ \footnote{We thank Noah Snyder for sharing this computation with us.}. Note that this calculation required an explicit formula for the smallest negligible morphism, which is why it is difficult in general.
\end{remark}
\end{defn}
\begin{defn}\label{def:Utilde}
Let $q$ be a primitive $4N$-th root of unity.

    \begin{enumerate}[(a)]
        \item Let $U$ be a $G_n$ module. Then the $G_n$ module $\tilde U$ is the same vector space as $U$, but with the $G_n$ action defined by $c.u=\al(c)u$, where $\al$ is the automorphism of $G_n$ defined in \ref{autom}.

        \item  Let $W$ be a $G_n[0]$ module, and let $\gamma_n$ be the automorphism defined in Remark \ref{q1lemma}. Then the $G_n[0]$ module $\tilde W$ is the vector space $W$ with the $G_n[0]$ action defined by $c.w=\gamma_n(c)w$.
    \end{enumerate} 

\end{defn}
As $\al(z_m)=(-1)^mz_m$,
this automorphism switches the two simple components of $\mathrm{Cliff}(m)$ for $m$ odd. In both cases, tilde implements a swap of parity, but take care that this notation is potentially confusing since the restriction from $G_n$ to $G_n[0]$ does not commute with tilde.

\begin{lem}\label{Uladiff}
Let $q$ be a primitive $4N$-th root of unity. We have the following:

\begin{enumerate}[(a)]
    \item  If $\ell(\la)^{(n)}$ is even, then $\tilde U_\la^{(n)}\cong \Ulnk$ as a $G_n$ module.

    \item  Let $\ell(\la)^{(n)}$ be odd.
Then $\tilde U^{(n)}_{\la,+}\cong U^{(n)}_{\la,-}$ as a $G_n$ module.
In particular, $U^{(n)}_{\la,+}$ and $U^{(n)}_{\la,-}$ have the same character.
But  $U^{(n)}_{\la,+}\not\cong U^{(n)}_{\la,-}$ as $G_n$-modules.

\item The multiplicity of a weight $q_k^{\nu(k)}$ 
  in $\Ulnk$ (respectively
$U_{(\la,\pm)}^{(n)}$)
is equal to the dimension of a simple $\mathrm{Cliff}(\llan)$ module. In particular, every weight of $V_{\la}^{(n)}$ is also a weight of $\Ulnk$ (respectively
$U_{(\la,\pm)}^{(n)}$).

\item The modules $U^{(n)}_{\la,+}$ and  $ U^{(n)}_{\la,-}$, viewed as $\Gno$ modules, are both isomorphic to  $W_\la^{(n)}$ for $\ell(\la)$  odd.
If $\ell(\la)$ is even, $U^{(n)}_\la\cong W_{\la,+}^{(n)}\oplus W_{\la,-}^{(n)}$ as $\Gno$ modules
and  $W_{\la,+}^{(n)}\not\cong W_{\la,-}^{(n)}$. The multiplicity of a weight $(q_k^{\nu(k)})$, $\nu(k)=\pm 1$,  in $W_\la^{(n)}$ resp
$W_{(\la,\pm)}^{(n)}$
is equal to the dimension of a simple $\mathrm{Cliff}(\llan-1)$ module.
\end{enumerate}

\end{lem}

\begin{proof} 
We begin with c). Let $\Lambda \in \Slnk$ and $(q^\La_k)_k=(q_k)_k$ be the associated $n$ -tuple such that $J_k\psi_\La=q_k^{-1}\psi_\La$. We denote the corresponding weight by $\Omega_0$.
Recall that the weight space for this weight is $\mathrm{Cliff}_\La \psi_\Lambda$. Hence the action of $\mathrm{Cliff}_\La$ on that weight space is isomorphic to the left regular
representation of $\mathrm{Cliff}_\La$. It follows that there is an isomorphism $\tau:  \mathrm{Cliff}_\La\to \operatorname{End}_{\mathrm{Cliff}_\La}(\mathrm{Cliff}_\La \psi_\Lambda)$. By Corollary~\ref{cor:WeightCommutant} the restriction of the idempotents $\rho(e)$ and $\rho( e_{\pm})$ from Definition~\ref{Udef} are minimal (and non-zero) in $\operatorname{End}_{\mathrm{Cliff}_\La}(\mathrm{Cliff}_\La \psi_\Lambda)$. Thus the multiplicity of the weight $\Omega_0=(q_k^{-1})_k$ in either $U_\lambda$ or $U_{\lambda, \pm}$ is equal to the dimension of a simple $\mathrm{Cliff}_\La \cong \mathrm{Cliff}(\ell(\lambda)^{(n)})$ module. This proves (c) for the weight $\Omega_0$. If $\Omega_\eta$ is a weight such that $J_kv\psi_\La=q_k^{\eta(k)}v\psi$ for a corresponding weight vector, we observe that multiplication by $v^\eta=\prod_{\eta(k)=+}v_k$ permutes the weight spaces of $\Omega_0$ and $\Omega_\eta$.
We deduce (c) for general weights from this.

Define the linear map $\Phi: V_\la^{(n)}\to V_\la^{(n)}$ by $v\psi_\La \mapsto \al(v)\psi_\La$ for all $v\in \mathrm{Cliff}(n),\ \La\in\Slnk$.
It follows that $\Phi c \Phi=\al(c)$ 
for all $c\in G_n$, and
$\Phi(U)\cong \tilde U$ for any $G_n$ submodule $U$ of $V_\la^{(n)}$
. Hence, if $\llan$ is even
and $U_\la^{(n)}=\rho(e)V_\la$, we obtain
$$\tilde U_\la^{(n)}\cong \Phi(U_\la^{(n)})=\rho(\al(e))V_\la^{(n)}\cong \rho(e)V_\la^{(n)}=U_\la^{(n)},$$
as the idempotents $e$ and $\alpha(e)$ are equivalent in $\mathrm{Cliff}(\llan)$. This proves (a). 

If $\llan$ is odd
and $U_{\la,+}=\rho(e)V_\la$, we similarly obtain
$$\tilde U_{\la,+}^{(n)}\cong \Phi(U_{\la,+}^{(n)})=\rho(\al(e))V_\la^{(n)} \cong U_{\la,-}^{(n)},$$
where the last isomorphism follows from the fact that $\al$ switches the simple components 
of $\mathrm{Cliff}(\llan)$. The character of a $G_n$-module is defined via the even elements $J_k$,
which are fixed by $\al$. Hence the modules $U^{(n)}_{\la,\pm}$ have the same characters. To show that the modules $U^{(n)}_{\la,\pm}$ are distinct recall the minimal idempotents $e_{\pm} \in \mathrm{Cliff}(\llan)_\pm$ from Definition~\ref{Udef}, and let $z_\La=\prod_{q_k^2 = 1} v_k$ the central element of $\mathrm{Cliff}_\La$. Then $\tau(z_\La)$ acts diagionally on the two elements $\rho(e_\pm)$ by scalars which differ by sign. Let us denote these two scalars by $\gamma_{\pm}$. Observe that we have a right action of $\mathrm{Cliff}_\Lambda$ on $\mathrm{Cliff}_\Lambda \psi_\Lambda$ using the map $\tau$. As $z_\La$ is central in $\mathrm{Cliff}_\Lambda$, the left and right action coincide. Hence the left action of $z_\La$ on the weight space $\mathrm{Cliff}_\La\psi_\La \cap U_{\la,\pm} = \mathrm{Cliff}_\La \psi_\Lambda \rho(e_\pm)$ is given by $ \ga_{\pm}$. As these two scalars differ we get that
$U_{\la,+}^{(n)}\not\cong U_{\la,-}^{(n)}$, completing the proof of (b).


The proof of (d) is an easy modification of the proof of the previous parts. The restriction rules from $G_n$ to $\Gno$ modules follow from the restrictions rules from $\mathrm{Cliff}(\ell(\lambda)+1) \cong \langle \kappa_0, \kappa_i, \tilde \kappa_j : 1 \leq i \leq f + \ell(\lambda), 1\leq j \leq f \rangle$ to $\mathrm{Cliff}(\ell(\lambda)) \cong  \langle  \kappa_i, \tilde \kappa_j : 1 \leq i \leq f + \ell(\lambda), 1\leq j \leq f \rangle$. We note that as $J_k\in\Gno$ for all $1\leq k\leq n$, we have the same weights also with respect to $\Gno$. We now have that $\operatorname{End}_{ \mathrm{Cliff}_\La[0]}(\mathrm{Cliff}_\La \psi_\Lambda) \cong \mathrm{Cliff}(\ell(\lambda)+1)$, and it follows that $\operatorname{End}_{ \mathrm{Cliff}_\La[0]}(\mathrm{Cliff}_\La \psi_\Lambda)$ is generated by the restriction of the operators $\kappa_i$, $\tilde \kappa_j$, and now also $\kappa_0$. The statement on weight multiplicities in the modules $W_{\lambda}$ and $W_{\lambda, \pm}$ now follows as in part (c).

\end{proof}

\subsection{Irreducibility and isomorphisms of representations} 
With the weight spaces of our representations now understood, we can show that these representations are irreducible and non-isomorphic using standard techniques.

\begin{prop}\label{irred}

Let $q$ a primitive $4N$-th root of unity. We have the following:

\begin{enumerate}[(a)]
    \item The $G_n$ modules $U^{(n)}_\la$ resp $U^{(n)}_{(\la,\pm)}$ are irreducible and mutually non-isomorphic. 

    \item  The $\Gno$ modules $W^{(n)}_\la$ resp $W^{(n)}_{(\la,\pm)}$ are irreducible and mutually non-isomorphic.

    \item The automorphism $\al$, as defined in Equation~\ref{autom}, interchanges the simple components $(G_n)_{(\la,\pm)}$ for $\ell(\la)$ odd, and fixes $(G_n)_{\la}$  for $\ell(\la)$ even. Similarly, the automorphism $\ga_n$, as defined in Remark~\ref{q1lemma} interchanges the simple components $G_n[0]_{(\la,\pm)}$ for $\ell(\la)$ even, and fixes $G_n[0]_{\la}$  for $\ell(\la)$ odd.

\end{enumerate}
\end{prop}

\begin{proof} (a) Assume that $U$ is a nonzero submodule of $\Ul$, for $\nu=\la$ or $\nu=(\la,\pm)$.  It is a well-known result that if $\Ul$ has a basis of weight vectors,
then so has $U$. Assume we have a weight vector $v\psi_\La\in U$,  $v\in \mathrm{Cliff}(n)$ of weight $\om$.
 As the weight space of $\om$ is a
$\mathrm{Cliff}_\La$-module, it follows that the multiplicity of $\om$ is at least equal to the dimension of
such a module, hence it coincides with the multiplicity of $\om$ in $\Ul$. If $\omega = (q_k^{\nu(k)})_k$ is the character for $v\psi_\La$,
it follows from the action of $\mathrm{Cliff}(n)$ on $v\psi_\La$ that also the multiplicities for any weight $(q_k^{\pm 1})$ 
for $\Ul$ and $U$ also agree, see Relation \ref{jkrelation} and Definition \ref{def:weights}. 

By Lemma \ref{betazero}(b) and the definition of $\Slnk$, the quantity $\beta_k^\La$ is nonzero for any
$\La\in \Slnk$ and $k$, $1\leq k<n$. Hence it follows from Definition~\ref{def:gnreproot} 
that $\mathrm{Cliff}(n)\psi_\La\subset U$ also implies $\mathrm{Cliff}(n)\psi_{s_k(\La)}\subset U$.
As any $\La'\in \Slnk$ can be obtained from a given $\La\in \Slnk$ via a sequence of transpositions $s_{k_i}$,
it follows that $U_\nu=\oplus_{\La\in\Slnk} \mathrm{Cliff}(n)\psi_\La\subset U$. 

It was shown in Lemma \ref{Uladiff},(b) that $U^{(n)}_{(\la,+)}\not\cong U^{(n)}_{(\la,-)}$ for $\ell(\la)$ odd. If two simple modules $U^{(n)}_\la$ and $U^{(n)}_{\mu}$ or $U^{(n)}_{(\mu,\pm)}$ are labeled by different Young diagrams $\la\neq\mu$, their weights are disjoint by Lemma \ref{lem:weights}. This proves
the last statement in (a).
 
(b) As the elements $J_k$ are in $\Gno$, we can define weights for $\Gno$ modules in the same way as for $G_n$ modules. We can therefore adapt the proof of (a) to this setting, by replacing $\mathrm{Cliff}_\La$ by $\mathrm{Cliff}_\La[0]$, and $\mathrm{Cliff}(n)$ by $\mathrm{Cliff}(n)[0]\cong \mathrm{Cliff}(n-1)$.

(c) We use the notations from the proof of Lemma \ref{Uladiff}. Let $p$ be a minimal idempotent in  $(G_n)_{(\la,+)}$ for $\ell(\la)$ odd, and let $v\in U_{(\la,+)}$ such that $pv\neq 0$. Then $\Phi(v)\in U_{(\la,-)} $ and $\al(p)\Phi(v)=\Phi p \Phi^2(v)=\Phi(p(v))\neq 0$. Hence $\al(p)\in(G_n)_{(\la,-)} $.
The proof of the corresponding statement for $G_n[0]$ is similar, using the automorphism $\ga_n$.
\end{proof}



\subsection{Restriction rules} 

We now move on to determining the restriction rules of our modules along the standard embeddings $G_n \subset G_{n+1}$ and $G_n[0] \subset G_{n+1}[0]$. By Corollary~\ref{cor:RestrictionFusion} these results will tell us the tensor product rules for the corresponding simple objects of $\operatorname{Ab}(\overline{\mathcal{SE}_N})$, and also for deducing that our representations form a complete set for the quotient algebra $\End_{\operatorname{Ab}(\overline{\mathcal{SE}_N})}(+^n)$ which we recall is a quotient of $G_n[0]$. These consequences will be applied in the next section. 

Recall that $Y_N^<$ denotes the set of strict Young diagrams with $\la_1<N$.
\begin{defn}\label{Gndef} (a) The directed graph $\Glnk$ is defined by the set of vertices
$V(N)$ given by 
$$V(N)\ =\ \{\la\in  Y^<_N,\ \ell(\la)\ even\}\ \cup\  \{ (\la, \pm), \la\in  Y^<_N,\ \ell(\la)\ odd\}.$$
There's either a single or double edge from $\la$ or $(\la,\pm)$ to $\mu$ or $(\mu,\pm)$ if and only if $\la\ \to^{(N)}\ \mu$, with a double edge if $\ell(\la)$ and $\ell(\mu)$ are even and a single edge otherwise.
We define $\Slk$ to be the set of all paths of length $n$ in the graph $\Glnk$.

(b)  The directed graph $\Glnk[0]$ is defined as $\Glnk$, except that we interchange the roles of odd and even. That is, the vertices are given by
$$V(N)[0]\ =\ \{\la\in  Y^<_N,\ \ell(\la)\ odd\}\ \cup\  \{ (\la, \pm), \la\in  Y^<_N,\ \ell(\la)\ even\},$$
and edges are defined as in (a), except that now we only have double edges between $\la$ and $\mu$
when both $\ell(\la)$ and $\ell(\mu)$ are odd.
\end{defn}

These vertices are defined in such a way that if $\mu$ is a vertex of $\Gamma(N)$ (resp. $\Gamma(N)[0]$), we can refer to the corresponding representation $U_\mu^{(n)}$ of $G_n$ (resp. $U_\mu^{(n)}$ of $G_n[0]$) using a single symbol rather than needing a sign as well. We will also use the notation $\nu \to \gamma \in \Gamma(N)$ to refer to the collection of directed edges from $\nu \to \gamma$.


\begin{lem}\label{restrictlemma} 
Let $q$ a primitive $4N$-th root of unity.

\begin{enumerate}[(a)]
    \item Let $\mu\in Y^<_N $, and let $n+1=|\mu|+fN$ for some integer $f\geq 0$.
Then the $G_{n+1}$ module $V_\mu^{(n+1)}$, viewed as a $G_n$ module, decomposes as 
$$V_\mu^{(n+1)}\ \cong\ \bigoplus_{\la\ \to^{(N)} \mu} \Vlnk\oplus v_{n+1}\Vlnk.$$

\item Let $\gamma =\mu$ or $\gamma=(\mu,\pm)$ be in $V(N)$, and let $n+1=|\mu|+fN$ for some $f\geq 0$.
Then the $G_{n+1}$ module $U_\gamma^{(n+1)}$, viewed as a $G_n$ module, decomposes as
$$U_\gamma^{(n+1)}\ \cong\ \bigoplus_{\nu \to \gamma \in \Gamma(N)} U_\nu^{(n)},$$
where the sum is over all edges of the graph $\Gamma(N)$ taken with multiplicity. 

\item The $G_{n+1}[0]$ module $W_\gamma^{(n+1)}$, viewed as a $G_n[0]$ module, decomposes as
$$W_\gamma^{(n+1)}\ \cong\ \bigoplus_{\nu \to \gamma \in \Gamma(N)[0]} W_\nu^{(n)},$$
where the sum is again with multiplicity. 
\end{enumerate}

\end{lem}

\begin{proof} Observe that we have a bijection
\begin{equation*}\label{restrictbijection}
Paths(\mu, n+1) \ \leftrightarrow\  \bigcup_{\la\ \to^{(N)} \mu}\ \Slnk, 
\end{equation*}
given by removing the last edge in the path corresponding to a tableau in $\operatorname{Paths}(\mu, n+1)$.
As $\mathrm{Cliff}(n+1)=\mathrm{Cliff}(n)\oplus v_{n+1}\mathrm{Cliff}(n)$, this bijection induces the decomposition of
$V_\mu^{(n+1)}$ into a direct sum of $\mathrm{Cliff}(n)$-modules as claimed. They are  preserved by the $t_i$ generators of $G_n$, which follows from Relation \ref{cross3}. Hence they are also $G_n$-modules.

Before proving part (b), observe that $cv_{n+1}=v_{n+1}\al(c)$ for all $c\in G_n$.
Hence $\tilde V_\la^{(n)}\cong v_{n+1}V_\la^{(n)}$ as a $G_n$ module, where $\tilde V_\la^{(n)}$ 
is the vector space $V_\la^{(n)}$ with $G_n$ action $c.v=\al(c)v$. 

We now prove part (b) by splitting into cases, first on whether  $\llan = \lmun$ or $\llan<\lmun$, and then based on the parity of $\llan$ and $\lmun$. We observe from the definition that the parity of $\ell(\nu)^{n}$ is the same as the parity of $\ell(\nu)$. Hence our formulae for the multiplicity of the summands of $W_{\gamma}^{(n+1)}$ are independent of $n$. 

If $\llan=\lmun$,
it follows that $\rho(e)(V_\la+v_{n+1}V_\la)\cong 2U_\la$ resp $\cong U_{\la,+}\oplus U_{\la,-}$,
depending on whether $\lmun = \llan$ is even or odd.

If $\llan<\lmun$ even, a minimal projection $e$ in $\mathrm{Cliff}(\llan)$ is also a minimal projection in $\mathrm{Cliff}(\lmun)$.
The claim now follows from the fact that $\rho(e)(V_\la^{(n)}+v_{n+1}V_\la^{(n)})\cong U^{(n)}_{\la,+}+U^{(n)}_{\la,-}$, by  Lemma \ref{Uladiff}. 


If $\llan<\lmun$ odd , we pick a minimal idempotent $\tilde e\in \mathrm{Cliff}(\llan)$.
Hence $\rho(\tilde e)V_\mu^{(n+1)}\cong U^{(n+1)}_{\mu,+}\oplus  U^{(n+1)}_{\mu, -}$.
But then $\rho(\tilde e)(V_\la^{(n)}\oplus v_{n+1}V_\la^{(n)})\cong 2 \rho(\tilde e)V_\la^{(n)}\cong 2U_\la^{(n)}$, by Lemma \ref{Uladiff}, (b).
As $U_{\mu,\pm}$ have the same characters, 
it follows that $U^{(n+1)}_{\mu,+}$, viewed as a $G_n$-module, contains a summand isomorphic to $U_\la^{(n)}$.

Part (c) is proved similarly, using the fact that the elements $J_k$ are in $\Gno$.
\end{proof}

\subsection{Semisimple quotients of $G_n$ and $G_n[0]$} 
We finish our results in the root of unity case by defining semisimple quotients of $G_n$ and $G_n[0]$. We will see in the following section that this

\begin{defn}
    Suppose $q$ is a primitive $4N$-th root of unity. Let
$Y^<_{N,n}$ be the set of strict Young diagrams $\la$ such $|\la|\equiv n$ mod $N$ and $|\la|\leq n$. Let $\bar G_n$ be the quotient of $G_n$ by the intersection of the annihilators of the modules $U_{\lambda}^{(n)}$ and $U_{\lambda,\pm}^{(n)}$. Let $\bar G_n[0]$ be defined in the same fashion with the corresponding modules $W_{\lambda}^{(n)}$ and $W_{\lambda,\pm}^{(n)}$.  
\end{defn}


The following result follows immediately from the definition and our results in the previous subsection. 

\begin{prop}\label{prop:quotient}
    We have that $\bar G_n$ is a semisimple algebra whose simple modules are indexed exactly by the  $U_{\lambda}^{(n)}$ and $U_{\lambda,\pm}^{(n)}$ for $Y^<_{N,n}$. Similarly, $\bar G_n[0]$ is a semisimple algebra whose simple modules are indexed exactly by the $W_{\la}^{(n)}$ and, if $\ell(\la)$ is even, by
    $W_{(\la,\pm)}^{(n)}$, with $\la\in Y^<_{N,n}$.
\end{prop}






When $n<N$ we have a stronger result.

\begin{prop} \label{prop:CountingReps}
Suppose $q$ be a primitive $4N$-th root of unity. For $n<N$, the modules $W_\nu$ with $\nu = \lambda$ with $|\lambda|=n $ and $\ell(\lambda)$ odd, or $\nu = (\lambda, \pm)$ with $|\lambda|=n $ and $\ell(\lambda)$ even form a complete set of mutually non-isomorphic simple representations of $G_n[0]$. In particular, when $n<N$ we have that $G_n[0]\cong \tilde{G}_n[0]$ is semisimple.
\end{prop}
\begin{proof}
    When $n < N$ our representations agree exactly with those from \cite[Proposition 6.7]{JN} since there are no restricted edges. Then we have 
    \[ \sum_\nu \dim(U_\nu)^2 = 2^n\cdot  n! = \dim(G_n) \quad  \text{ and } \quad \sum_\nu \dim(W_\nu)^2 = 2^{n-1}\cdot  n! = \dim(G_n[0]),  \]
    where the first combinatorial identity was proved in e.g. \cite[Proposition 6.7]{JN}, and the second identity follows from this and the restriction rule in Lemma~\ref{Uladiff} (d).
    Hence $G_n$ and $G_n[0]$ are semisimple, and the listed irreducible modules form a complete set.
    
\end{proof}


\subsection{Representations of $G_n$ at $q$ not a root of unity.}


In this subsection we will work with $q\in \mathbb{C}$ not a root of unity. This case ends up being significantly easier than the one with $q$ a primitive $4N$-th root of unity case. Nearly all the proofs of the previous section work verbatim. The main difference is that in the current case, we do not have to deal with the restricted graphs $\Gamma(N)$, and are instead working on the generic graph $\Gamma(\infty)$. 
Hence the proofs here are identical to the special case with $N>n$ in the previous section.

\begin{remark}
Note that our results here are a slight strengthening of \cite{JN}. They work over a finite extension of the field of rational functions $\mathbb{C}(q)$, it follows that their results hold for ``generic $q$'' in the sense that they hold for $q$ in an open dense subset of $\mathbb{C}$, but there's little control over this open dense subset.
\end{remark}

Recall from Prop. \ref{prop:gkaction} that the representations $V_\la$ are well-defined for all $q$ which are not roots of unity. For generic $q$ it is shown in \cite{JN} that $\operatorname{End}_{G_n}(V_\lambda) \cong \operatorname{Cliff}(\ell(\lambda))$. However, once $q$ is specialised, we may have that $\operatorname{End}_{G_n}(V_\lambda)$ gets larger. For now we have the following. Recall the definition of the commuting operators $\kappa_i$ from Definition~\ref{def:coms}. Note that these operators are still well-defined in the current setting.
\begin{prop}
    Let $q\in \mathbb{C}-\{0\}$ not be a root of unity. Then the map
\[  \rho(v_i) :=  \kappa_i 
:\mathrm{Cliff}\left(\ell(\lambda)\right) \to \End_{G_n}\left(V_\lambda   \right)     \]
is an injective algebra homomophism. Further, this map extends to an injective algebra homomorphism
\[  \rho_0  :\mathrm{Cliff}\left(\ell(\lambda)+1\right) \to \End_{G_n}\left(V_\lambda   \right)   \]
where the additional generator is mapped to $\kappa_0$.
\end{prop}

Using this explicit subalgebra of the commutant, we can define submodules of the $V_\la$.
\begin{defn}\label{def:creation}
Let $q\in \mathbb{C}-\{0\}$ not a root of unity.

\begin{enumerate}[(a)]
    \item We define the $G_n$-modules
$$\ell(\la)\ {\rm even:}\hskip 1em U_\lambda =\rho(e)V_\lambda,\hskip 3em \ell(\la)\ {\rm odd:}\hskip 1em U_{\la,\pm} = \rho( e_\pm)V_\lambda,$$
where $e$ is a minimal projection in $\mathrm{Cliff}(\ell(\lambda))$ for $\ell(\lambda)$ even, and $e_\pm$ are two inequivalent minimal projections in 
$\mathrm{Cliff}(\ell(\lambda))$ for $\ell(\la)$ odd. Again we distinguish the sign $e_\pm$ based on the action of the canonical non-trivial central element.

\item We define the $\Gno$-modules
$$\ell(\la)\ {\rm odd:}\hskip 1em W_\la=\rho_0(e)V_\lambda,\hskip 3em \ell(\la)\ {\rm even:}\hskip 1em W_{\la,\pm} = \rho_0( e_\pm)V_\lambda$$
where $e$ is a minimal projection in $\mathrm{Cliff}(\ell(\la)+1)$ for $\ell(\la)$ odd, and $\tilde e_\pm$ are inequivalent minimal projections (as usual distinguished by the action of the canonical central element) in 
$\mathrm{Cliff}(\ell(\la)+1)$ for $\ell(\la)$ even.
\end{enumerate} 
\end{defn}
A slight alteration of the results of the previous subsection prove the following. To adapt the proof, we need that $x_m = \pm 1$ if and only if $m=0$ or $m=-1$. This follows from Lemma~\ref{xmdef1} (c).
\begin{lem}\label{lem:irr}

Let $q\in \mathbb{C}-\{0\}$ not a root of unity. Then

\begin{enumerate}[(a)]
    \item The $G_n$ modules $U_\la$ resp $U_{(\la,\pm)}$ are irreducible.

    \item  The $\Gno$ modules $W_\la$ resp $W_{(\la,\pm)}$ are irreducible.
\end{enumerate}
\end{lem}

We can also obtain restriction rules for these modules in terms of the generic graphs $V(\infty)$ and $V(\infty)[0]$ as in Definition~\ref{Gndef}.

\begin{lem}\label{lem:restrict}
Let $q\in \mathbb{C}-\{0\}$ not a root of unity, and $n\in \mathbb{N}$.

\begin{enumerate}[(a)]
   
\item Let $\gamma =\mu$ or $\gamma=(\mu,\pm)$ be in $V(\infty)$ with $|\mu| = n+1$.
Then the $G_{n+1}$ module $U_\gamma$, viewed as a $G_n$ module, decomposes as
$$U_\gamma \cong\ \bigoplus_{\nu \to \gamma \in \Gamma(\infty)} U_\nu,$$
where the sum is over all edges of the graph $\Gamma(N)$ taken with multiplicity.

\item  Let $\gamma =\mu$ or $\gamma=(\mu,\pm)$ be in $V(\infty)[0]$ with $|\mu| = n+1$. Then the $G_{n+1}[0]$ module $U_\gamma$, viewed as a $G_n[0]$ module, decomposes as
$$W_\gamma\ \cong\ \bigoplus_{\nu \to \gamma \in \Gamma(\infty)[0]} W_\nu,$$
where the sum is again with multiplicity.

\item  The $G_n$ module $U_\la$ viewed as as $G_n[0]$ module decomposes as $W_{\lambda,+} \oplus W_{\lambda,-} $. The $G_n$ module $U_{\la,\pm}$ viewed as as $G_n[0]$ module decomposes as $W_{\lambda} $.

\end{enumerate}

\end{lem}

We can then obtain the following classification result.
\begin{thm}\label{gkactionthm} 
Let $q\in \mathbb{C}-\{0\}$ not a root of unity, and $n\in \mathbb{N}$. Then
\begin{enumerate}[(a)]
    \item The modules $U_\nu$ with $\nu = \lambda$ with $|\lambda|=n $ and $\ell(\lambda)$ even, or $\nu = (\lambda, \pm)$ with $|\lambda|=n $ and $\ell(\lambda)$ odd form a complete set of mutually non-isomorphic simple representations of $G_n$. In particular, $G_n$ is semisimple.

    \item The modules $W_\nu$ with $\nu = \lambda$ with $|\lambda|=n $ and $\ell(\lambda)$ odd, or $\nu = (\lambda, \pm)$ with $|\lambda|=n $ and $\ell(\lambda)$ even form a complete set of mutually non-isomorphic simple representations of $G_n[0]$. In particular, $G_n[0]$ is semisimple.
\end{enumerate}
\end{thm}
\begin{proof}
    The proof that the modules $U_\nu$ (resp. $W_\nu$) are non-isomorphic mirrors the root of unity argument. Then that they exhaust all representations follows the argument of Proposition~\ref{prop:CountingReps}.
     Hence $G_n$ and $G_n[0]$ are semisimple, and the listed irreducible modules form a complete set.
    
\end{proof}

\section{Description of the categories $\operatorname{Ab}(\E_q)$ and $\operatorname{Ab}(\overline{\mathcal{SE}_{N}})$ }\label{sec:description}

The main goal in this section is to determine the structure of the categories $\operatorname{Ab}(\E_q)$ and $\operatorname{Ab}(\overline{\mathcal{SE}_{N}})$. To do this, we will need to determine the structure of the endomorphism algebras $\operatorname{End}_{\mathcal{E}_q}(+^n)$ and $\operatorname{End}_{\overline{\mathcal{SE}_N}}(+^n)$. We will show that these algebras are isomorphic to the algebras of Theorem~\ref{gkactionthm} and Proposition~\ref{prop:quotient}. Our main technique for achieving this will be via analysing the pullback of the categorical trace on the algebras $G_n[0]$.



\begin{remark}\label{condexpexample} Let $\Ca$ be a spherical tensor category, such that $\bar\Ca=\Ca/\operatorname{Neg}(\Ca)$ is semisimple, see Subsection~\ref{sec:Cauchy}. By definition, the categorical trace is nondegenerate on $\End(Y)$ for any object $Y$ in $\bar\Ca$. Fix an object $X$ in $\bar\Ca$ with $\dim X\neq 0$.
Then we define the normalized trace $tr$ on $\End(X^{\otimes n})$ by $tr(a)=\frac{1}{(\dim X)^{n}}Tr(a)$.
 This functional is compatible with the standard embedding $a\in \End(X^{\otimes n})\to a\otimes 1\in \End(X^{n+1})$. 
It follows from the definitions that $\mathrm{tr}_{X^{n+1}}(a\otimes 1)=\mathrm{tr}_{X^{\otimes n}}(a)$ for $a\in \End(X^{\otimes n})$.
We therefore will only write $tr$ for $\mathrm{tr}_{X^{\otimes n}}$.
\end{remark}
Recall that we have an isomorphism $\Gno \to \End_{\E_q}(+^n)$. We use this map to define a trace on $\Gno$.
\begin{defn}
    Let $q \in \mathbb{C}- \{-1,0,1\}$. We denote the pullback to $\Gno$ of the normalized
 trace on $\End(+^n)$ by $tr$.
\end{defn}

We begin by analyzing the case where $q$ is not a root of unity. From the previous section we know that the algebras $G_n[0]$ are semisimple, which makes this case significantly easier. We will also obtain results which are key for the $q$ a root of unity case later.




\subsection{Traces for the algebras $\Gno$ at $q$ not a root of unity}
Recall that if $A=\oplus_\ga A_\ga$ is a direct sum of full matrix rings $A_\ga$, any trace functional $tr:A\to\C$
is completely determined by its weight vector $(\om_\ga)$, where $\om_\ga=tr(p_\ga)$ for a minimal projection $p_\ga\in A_\ga$.

In the case of $q$ not a root of unity we can determine the weight vector of the categorical trace on the algebras $\Gno$. We do this by using the isomorphism $\Gno \to \End_{\E_q}(+^n)$. From the previous section we know that $\Gno$ is semisimple for $q$ not a root of unity. It then follows that the tensor product rules between the simple summands of $+^n$ in $\mathcal{E}_q$ agree with the special case of $q=1$. The tensor product rules in the $q=1$ case are well-known thanks to Sergeev duality. Using the fact that the categorical trace respects tensor products and direct sums, we can inductively compute the dimensions of our simple summands. This in turn gives us our weight vector.


We first summarize what our previous results say about $\End_{\E_q}(+^{ n})$.

\begin{prop}
Let $q\in \mathbb{C}- \{0\}$ be not a root of unity. Then $\End_{\E_q}(+^{ n})$ is semisimple and isomorphic to $\Gno$. In particular, the simple subobjects of $+^{ n}$ are labeled by strict Young diagrams $\la$ with $n$ boxes, with two non-isomorphic objects ${(\la,\pm)}$ if $\ell(\la)$ is even. 
\end{prop}
\begin{proof}
    By \cite[Corollary 5.12]{ConformalA} we have that $\operatorname{End}_{\mathcal{E}_q}(+^n)$ is isomorphic to $G_n[0]$. It follows from Theorem~\ref{gkactionthm} that $\operatorname{End}_{\mathcal{E}_q}(+^n)$ is semisimple. 
\end{proof}

Let $\operatorname{Ab}(\mathcal{E}_q)^\mathrm{poly}$ denote the full subcategory of subobjects of $+^{ n}$. Note that although $\mathcal{E}_q$ does not make sense at $q=\pm 1$, the subcategory $\operatorname{Ab}(\mathcal{E}_q)^\mathrm{poly}$ does have a sensible interpretation at $q=\pm 1$ since no circles appear when composing endomorphisms of $+^n$. Indeed $\operatorname{Ab}(\mathcal{E}_1)^\mathrm{poly}$ agrees with $\mathrm{Rep}^{\mathrm{poly}}(q_\infty)$ by Sergeev duality \cite{Ser}.

\begin{prop} \label{prop:GrothendieckDeform}
 Let $q\in \mathbb{C}- \{0\}$ be not a root of unity. We have an isomoprhism of Grothendieck rings $K_0(\operatorname{Ab}(\mathcal{E}_q)^\mathrm{poly}) \cong K_0(\mathrm{Rep}^{\mathrm{poly}}(q_\infty))$ sending $p_{\lambda}$ to $p_{\lambda}$ and $p_{\lambda, \pm}$ to $p_{\lambda,\pm}$.
\end{prop}
\begin{proof}

Choose a path connecting $1$ to $q$ which does not go through $0$ or any roots of unity. Choose the square roots used to define the representations of $G_n[0]$ to vary continuously over this path. Now the representations $W_\lambda$ and $W_{\lambda,\pm}$ vary continuously along this path, and the corresponding projections can also be chosen to vary continuously. Hence the coefficients of the fusion rules for tensoring these projections (i.e the structure constants of $K_0(\operatorname{Ab}(\mathcal{E}_q)^\mathrm{poly})$) are natural numbers which vary continuously. Since a natural number which varies continuously is constant, the fusion rules  do not change as we move along this path. In particular, all of the $K_0(\operatorname{Ab}(\mathcal{E}_q)^\mathrm{poly})$ are isomorphic to $K_0(\operatorname{Ab}(\mathcal{E}_1)^\mathrm{poly})$ which is $K_0(\mathrm{Rep}^{\mathrm{poly}}(q_\infty))$ by Sergeev duality.

\end{proof}

Let $d=q_{[1]}=2i/(q-q^{-1})$. Note that $d$ is the categorical dimension of the generating object $+\in\mathcal{E}_q$.

\begin{prop}\label{traceweights} 
 Let $q\in \mathbb{C}- \{0\}$ be not a root of unity. The dimensions of $(\la)$ resp $(\la,\pm)$ are equal to $q_\la/2^{\lfloor \ell(\la)/2\rfloor}$. Hence the weight vector for the normalized diagram trace $tr$ on $\End_{\E_q}(+^{ n})$ is given by $(q_\la/2^{\lfloor \ell(\la)/2\rfloor}d^n)$, which is also well-defined for $q=\pm 1$. 
\end{prop}

\begin{proof}
As explained in Subsection~\ref{Gn0def}, the restriction of the normalized diagram trace to the Hecke algebra is the unique Markov trace with $\eta=1/2$, and so by Proposition~\ref{prop:HansDimensions} we have $tr(p_{(n)})= a_n/d^n = q_{(n)}/d^n$. This implies
$\dim ([n])=q_{(n)}$ 
\footnote{Note that this discussion is the motivation behind the specialisation of the Hall-Littlewood polynomials $Q_\lambda$ given in Subsection~\ref{sub:formal}.}. Furthermore, in the case that $\ell(\lambda)$ is even, we have from Definition~\ref{def:Utilde} and Lemma~\ref{Uladiff} b) that the automorphism $\gamma_{|\lambda|}$ exchanges the modules $W_{\la, +}$ and $W_{\la, -}$. This implies that $d(\lambda, +) =d(\lambda, -)$. The general dimension formula then follows from Proposition~\ref{prop:GrothendieckDeform} and Proposition~\ref{prop:dimFun}. 

It is easy to check that $\lim_{q\to\pm 1} q_\la/2^{\lfloor \ell(\la)/2\rfloor}d^n$ exists. 

\end{proof} 

As a corollary we obtain a formula for the categorical trace $tr$ in terms of the standard matrix trace on the modules $V_\lambda$. Note that 
a slightly modified formula will be used to define a trace on the quotient $\bGn$ in the next section.
\begin{cor}\label{cor:tracereg}
    Let $q\in \mathbb{C}-\{0\}$ be not a root of unity, and let $\mathrm{Tr}_\la$ be the usual trace on $V_\la$ and let $c\in G_n[0]$. We have
    $$\mathrm{tr}(c)\ =\ \sum_{|\la|=n} \frac{q_\la}{2^{\ell(\la)}d^n}\mathrm{Tr}_\la(c).$$
\end{cor}

\begin{proof} It follows from Definition~\ref{def:creation} and Lemma~\ref{lem:irr} that $V_\la$ is isomorphic to the direct sum of $2^{ \ell(\la)/2}$ copies of $W_{(\la,+)}\oplus W_{(\la,-)}$
for $\ell(\la)$ even, and of $2^{ (\ell(\la)+1)/2}$ copies of $W_{\la}$ for $\ell(\la)$ odd.
These multiplicities give us the rank of a minimal idempotent $p=p_\la$ or $p=p_{(\la,\pm)}$ of $G_n[0]$ in these representations. This allows us to check that the right hand side in our claim with $c=p$, and since the $p$'s generate the whole algebra the result follows.
\end{proof}

Our next goal is to give an recursive formula for the trace $\mathrm{tr}$ for $q$ not a root of unity. Our eventual goal is to show that when $q$ is a root of unity that the trace is given by a specialization of this same recursive formula, which will allow us to prove an identity like Corollary~\ref{cor:tracereg} when $q$ is a root of unity.

We now prove some useful recursive identities regarding the matrix coefficients for the trace $\mathrm{Tr}_\lambda$. First we need the following technical lemma. Since we will need this result in both the root of unity and non-root of unity cases and since the proofs are entirely parallel we state both here.

\begin{lem}\label{lem:matrixcoeff} 
Let $q \in \mathbb{C} - \{-1,0,1\}$, and let $n\in \mathbb{N}$, and let $\chi\in \{ 1, v_nv_{n+1}, t_n, v_nv_{n+1}t_n\}$. If $q$ is not a root of unity let $\mu$ be a strict Young diagram with $|\mu| = n+1$, and let $v\psi_\Gamma \in V_\mu$ be a basis vector with weight $(q_k^{\nu(k)})_{i=1}^n$. If $q$ is a $4N$-th root of unity let $\mu$ be a strict Young diagram with $|\mu| \equiv n+1 \mod N$, and let $v\psi_\Gamma\in V_\mu^{(n)}$ be a basis vector with weight $(q_k^{\nu(k)})_{i=1}^{n}$.
In either case, we have that $\chi v\psi_\Ga$ is a linear combination of $v\psi_\Ga$, $v_nv_{n+1}v\psi_\Ga$ and $\tilde v\psi_{s_n(\Ga)}$ for some $\tilde v\in \mathrm{Cliff}(n+1)$ which depends on $\chi$ and $v$. Furthermore, the coefficient $a(\chi, v\psi_\Ga)= a(q_n^{\nu(n)},q_{n+1}^{\nu(n+1)};\chi)$ is the evaluation of a fixed rational function in $q$ of $v\psi_\Ga$ which only depends on $\chi$. These algebraic functions are as follows:
\[\begin{tabular}{c|c}
   $\chi$  &  $a(x,y;\chi)$ \\\hline
    $1$ & $1$\\
    $v_nv_{n+1}$ &$0$\\
    $t_n$ & $-\frac{q-q^{-1}}{x y^{-1}-1}$\\
    $v_nv_{n+1}t_n$ & $-\frac{q-q^{-1}}{x^{-1}y-1}$.
\end{tabular}\]


\end{lem}

\begin{proof}
    As each $\chi$ commutes with $v_i$ for $i<n$ it suffices to check the first statement for  $v\psi_\Ga$ with $v\in\{ 1, v_n,v_{n+1}, v_nv_{n+1}\}$.
    For $\chi=v_nv_{n+1}$ the first part of the statement is obvious (in fact, we don't need a $\tilde{v}$ term), The second part is also immediate, since the coefficient is always zero.
    
    If $\chi=t_n$ and $v=1$, the claim follows from \ref{tkrep} taking $\tilde{v}=1$. The other cases can be easily deduced from that, using relations \ref{cross1} and \ref{cross2}. E.g. if $v = v_n v_{n+1}$ we have 
    \begin{align*}
        t_nv_nv_{n+1}\psi_\Ga\ =&\ v_{n+1}t_nv_{n+1}\psi_\Ga\ =\ (v_{n+1}v_nt_n-(q-q^{-1})v_{n+1}(v_n-v_{n+1}))\psi_\Ga\ =\cr
        =&\ (v_{n+1}v_n(\hat\beta  \psi_{s_n(\Ga)}-\frac{q-q^{-1}}{q_n^{-1}q_{n+1}-1}\psi_\Ga+\frac{q-q^{-1}}{q_nq_{n+1}-1}v_nv_{n+1}\psi_\Ga)\cr &-\ (q-q^{-1})v_{n+1}(v_n-v_{n+1}))\psi_\Ga\ =\cr
        =&\ \hat\beta v_{n+1}v_n \psi_{s_n(\Ga)}+\frac{(q-q^{-1})q_nq_{n+1}}{q_nq_{n+1}-1}\psi_\Ga-\frac{(q-q^{-1})q_n^{-1}q_{n+1}}{q_n^{-1}q_{n+1}-1}v_{n+1}v_{n}\psi_\Ga\ =\cr
        =&\ -\hat\beta v_nv_{n+1} \psi_{s_n(\Ga)} -\frac{(q-q^{-1})}{q_n^{-1}q_{n+1}^{-1}-1}\psi_\Ga-\frac{(q-q^{-1})}{q_nq_{n+1}^{-1}-1}v_nv_{n+1}\psi_\Ga.
    \end{align*}




    The other cases for $v$, and the case $\chi=v_nv_{n+1}t_n$ are checked similarly.

\end{proof}


Recall that we write $\lambda \rightarrow \mu$ if $\mu$ can be obtained from $\lambda$ by adding a single box.



\begin{remark}\label{rem:weightbasis} For each $n$, and for each label $\nu \in V(\infty)[0]$ 
of the form $\lambda$ or $(\lambda, \pm)$ with $n = |\lambda|$, we have that $W_\nu$ has a basis of weight vectors $w_i^\nu$. Let $e_{ij}^\nu$ denote the image of the corresponding matrix element of $\End(W_\nu)$ in the corresponding matrix factor of the semisimple algebra $G_n[0]$. We will always assume in the following that any system of matrix units for the semisimple algebra $G_n[0]$ will be of that form.
Using our definitions, we obtain the following equation
\begin{equation}\label{eq:weightbasis}
tr(e_{ij}^\nu)tr(\chi)\ =\ tr(e_{ij}^\nu\chi)\ =\ \sum_{\la\to\mu}\frac{q_\mu}{d^{n+1}2^{\ell(\mu)}}\mathrm{Tr}_\mu(e_{ij}^\nu\chi);
\end{equation}
indeed, it follows from our restriction rules that $\mathrm{Tr}_\mu(e_{ij}^\nu\chi)$ is nonzero only if $\la\to\mu$.
We will use this equation to obtain identities of rational functions involving matrix entries in our representations which will be useful for studying traces for $q$ a root of unity. 

Preempting our work in the root of unity setting, we make an analogous definition for the algebras $\bar G_n[0]$ at $q$ a $4N$-th root of unity. For each $n$, and for each label $\nu \in V(N)[0]$ of the form $\lambda$ or $(\lambda, \pm)$ with $n \equiv |\lambda| \pmod N$, write $e_{ij}^\nu$ for the image of the corresponding matrix element of $\End(W_\nu^{(n)})$ in the corresponding matrix factor of the semisimple algebra $\bar G_n[0]$. 
\end{remark}


\begin{lem}\label{lem:condexpcond} Let $q\in \mathbb{C}-\{0\}$ be not a root of unity. 
Fix a label $\nu \in V(\infty)[0]$ of the form $\lambda$ or $\lambda, \pm$, and fix an index $i,j$ of the weight basis of $W_\nu$ corresponding to the weight $(q_k^{\eta(k)})_{k=1,\ldots n}$. 
    Let $\chi\in \{1, t_n, e_n, e_nt_n\}$, then we have the following trace identities in $G_{n+1}[0]$.
    \begin{enumerate}[(a)]
        \item If $\mu \in V(\infty)[0]$ with $|\mu| = n+1$, then 
          \[\mathrm{Tr}_\mu(e_{ij}^\nu\chi)=\delta_{i,j}(a(q_n^{\eta(n)},q_{n+1};\chi)+a(q_n^{\eta(n)},q_{n+1}^{-1};\chi))\mathrm{Tr}_\la(e_{ii}^\nu)\]
        if there's an edge $\lambda \rightarrow \mu$ in $\tilde{\Gamma}(\infty)[0]$, and $\mathrm{Tr}_\mu(e_{ij}^\nu\chi) = 0$ otherwise.
        \item The diagram trace satisfies \[tr(\chi)d q_\la\ =\ \sum_{ \la\to\mu} (a(q_n^{\eta(n)},q_{n+1}(\la\to\mu);\chi)+a(q_n^{\eta(n)},q_{n+1}^{-1}(\la\to\mu);\chi))2^{\ell(\la)-\ell(\mu)}q_\mu.\] 
    \end{enumerate}
    Here $q_{n+1}$ is determined by the edge $\lambda \to \mu$.

\end{lem}

\begin{proof} 

To prove $\mathrm{Tr}_\mu(e^\nu_{ij}\chi)=0$ for $i\neq j$ and $\chi=v_nv_{n+1}$, observe that for a fixed weight vector $v\psi_\Gamma\in V_\mu$ the $G_n[0]$ modules generated by $v\psi_\Ga$ and by $v_nv_{n+1}v\psi_\Ga$ have zero intersection. Indeed, one of them is contained in span $\{ \mathrm{Cliff}(n)\psi_{\Ga'}, \Ga'\in {\mathcal S}_\mu\}$, while the other one is contained in span $\{ v_{n+1}\mathrm{Cliff}(n)\psi_{\Ga'}, \Ga'\in {\mathcal S}_\mu\}$ (depending on if $v_{n+1} \in v$ or not). It follows that $\mathrm{Tr}_\mu(cv_nv_{n+1})=0$ for any $c\in G_n[0]$. 

For the case $\chi=t_n$ and $i\neq j$, it follows from Lemma~\ref{lem:matrixcoeff} that
$t_n v\psi_\Ga$ is a linear combination of $v\psi_\Ga$, $v_nv_{n+1}v\psi_\Ga$ and $\tilde v\psi_{s_n(\Ga)}$ for some $\tilde v\in \mathrm{Cliff}(n+1)$, and that the coefficient of $v\psi_\Gamma$ is $a(q_n^{\nu(n)},q_{n+1}^{\nu(n+1)};t_n)$ which only depends on the weight of $v\psi_\Ga$. By the previous paragraph, the map $v\psi_\Ga\mapsto e_{ij}^{\nu }v_nv_{n+1}v\psi_\Ga$ has trace zero. The map $v\psi_\Ga\mapsto e_{ij} a(q_n^{\nu(n)},q_{n+1}^{\nu(n+1)};t_n)v \psi_\Ga$ is the composition of $v\psi_\Ga\mapsto  a(q_n^{\nu(n)},q_{n+1}^{\nu(n+1)};t_n) v \psi_\Ga$ with the operator $e_{ij}^\nu$. In the basis of the first paragraph, the first operator is a diagonal matrix, while the second is off-diagonal. It follows that their composition is traceless. Hence it suffices to check that the map $v\psi_\Ga\mapsto c \tilde v\psi_{s_n(\Ga)}$ has zero trace for all $c\in G_n[0]$. This 
follows from the fact that $\mathrm{Cliff}(n+1)\psi_{\Ga}$ and $\mathrm{Cliff}(n+1)\psi_{s_n(\Ga)}$ generate $G_n[0]$ modules with zero intersection; indeed these modules are $V_\la\oplus v_{n+1}V_\la$ and $V_{\la'}\oplus v_{n+1}V_{\la'}$, where $\la=\Ga(n)$ and $\la'=s_n(\Ga)(n)$. The proof for $\mathrm{Tr}_\mu(e_{ij}^\nu  v_nv_{n+1}t_n)=0$ is near identical using Lemma~\ref{matrixcoeff}.

Observe that if $e_{ii}^\nu$ has weight $\Omega=(q_k^{\eta(k)})_{k=1}^n$ and $\la\to\mu$,
    then $e_{ii}^{\nu} V_\mu$ is a direct sum of weight spaces with weights $(\Omega,q_{n+1}^{\pm 1})$, where $q_{n+1}$ is determined by the edge $\la\rightarrow \mu$. Recall from the above discussion that the maps $v\psi_\Ga\mapsto c \tilde v\psi_{s_n(\Ga)}$ are traceless. By Lemma~\ref{lem:matrixcoeff} we obtain that
     \[\mathrm{Tr}_\mu(e_{ii}^\nu\chi)=(a(q_n^{\eta(n)},q_{n+1};\chi)+a(q_n^{\eta(n)},q_{n+1}^{-1};\chi))\mathrm{Tr}_\la(e_{ii}^\nu).\]
     
     Observe that the idempotent $e_{ii}^\nu\in \Gno$ acts nonzero only in the modules $V_\mu$ of $\Gnoo$ for which $\la\rightarrow \mu$. Using Corollary \ref{cor:tracereg}, we obtain
     \begin{equation}\label{tracecompare}
    \frac{q_\la}{d^n2^{\ell(\la)}}\mathrm{Tr}_\la(e_{ii}^\nu)tr(\chi)\ =\ tr(e_{ii}^\nu)tr(\chi)\ =\ tr(e_{ii}^\nu\chi)\ =\ \sum_{\la\to\mu}\frac{q_\mu}{d^{n+1}2^{\ell(\mu)}}\mathrm{Tr}_\mu(e_{ii}^\nu\chi).
    \end{equation}
     Multiplying this equation by $d^{n+1}2^{\ell(\la)}$ and applying part (a) now gives the second claim. 
\end{proof}

The right hand side of the expression in (b) above involves choices of signs for square roots. It is independent of these choices, as the left hand side is a rational function in $q$. We have the equality in (b) for infinitely many specialisations. Hence we obtain the following. This result will be key in the root of unity setting.

\begin{cor}\label{cor:rationalEq}
   Let $\gamma\to \lambda$ be strict Young diagrams, define $n= |\lambda|$, let $\varepsilon\in \{1,-1\}$, and let $\chi \in \{1 ,e_n,t_n,e_nt_n\}$. We have the equality of rational functions in the formal variable $q$
    \[tr(\chi)d q_\la\ =\ \sum_{\la\to\mu} (a(q_n^{\varepsilon},q_{n+1};\chi)+a(q_n^{\varepsilon},q_{n+1}^{-1};\chi))2^{\ell(\la)-\ell(\mu)}q_\mu,\]
    where $q_n$ is determined by the edge $\gamma\to \lambda$, and $q_{n+1}$ is determined by the edge $\lambda \to \mu$. In particular if $\lambda_1 < N$ then the above equality holds when $q$ is specialised to a primitive $4N$-th root of unity, for any choices of square roots in the expressions for $q_n$ and $q_{n+1}$ for the summands on the right hand side.
\end{cor}
\begin{proof}
    By Corollary~\ref{Nadcor} we have that $q_\lambda$ and the $q_\mu$ are well-defined when specialised to $q$. Looking at the table in Lemma \ref{lem:matrixcoeff} 
    The specialisations of $a(q_n^{\nu(n)},q_{n+1};\chi)$ and $a(q_n^{\nu(n)},q_{n+1}^{-1};\chi)$ are well-defined by Lemma~\ref{lem:betaWellDefined}. 
\end{proof}



\subsection{Traces for the quotients $\bGn$ at $q$ a root of unity}\label{sub:trace} 
In the case that $q$ is a root of unity, we do not have that $G_n[0]$ is semisimple. Hence the techniques of the previous subsection do not work to deduce the structure of $\operatorname{tr}$. Instead we will define an explicit trace on the semisimple quotient $\bar G_n[0]$, and inductively show that its pullback to $G_n[0]$ agrees with $\operatorname{tr}$. A key ingredient of the proof showing these traces agree will be Corollary~\ref{cor:rationalEq}.

\begin{defn}
    Let $q$ be a primitive $4N$-th root of unity, and let $\mathrm{Tr}_{\la^{(n)}}$ be the usual trace on $\Vlnk$, where $n=|\la|+f_\la N$. Then
we define a trace $\tilde{\mathrm{tr}}_n $ on $\Gno$ by
    $$\tilde{\mathrm{tr}}_n(c)\ :=\ \sum_{\la\in Y^<_{N,n}} \frac{q_\la}{2^{\ell(\la)+f_\la}d^n}\mathrm{Tr}_{\la^{(n)}}(c).$$
\end{defn}

The following result on the matrix entries of $\mathrm{Tr}_{\la^{(n)}}$ will be required to show $\tilde{\mathrm{tr}}_n = \operatorname{tr}$.

\begin{lem}\label{lem:traceFormula}
    Let $q$ be a primitive $4N$-th root of unity, let $n\in \mathbb{N}$, let $\nu \in V(N)[0]$, and fix an index $i,j$
    of the weight basis of $W_\nu^{(n)}$ corresponding to the weight $(q_k^{\eta(k)})_{k=1,\ldots n}$. Let $\chi\in \{1, t_n, e_n,e_nt_n\}$, then we have
    \[ \mathrm{Tr}_{\mu^{(n+1)}}(e_{ij}^\nu\chi)=\delta_{i,j}(a(q_n^{\eta(n)},q_{n+1};\chi)+a(q_n^{\eta(n)},q_{n+1}^{-1};\chi))\mathrm{Tr}_{\la^{(n)}}(e_{ii}^\nu)\]
    for all $\mu\in Y^<_{N,n+1}$ with $\lambda \to^{(N)} \mu$ where $\nu = \lambda$ or $\nu = (\lambda, \pm)$ depending on the parity of $\ell(\lambda)$. The value $q_{n+1}$ is determined by the edge $\lambda \to^{(N)} \mu$, and the value $q_n$ is determined by the weight of $e_{ii}^\nu$
\end{lem}
\begin{proof}
    This proof is near identical to the proof of Lemma~\ref{lem:condexpcond} (a).
\end{proof}

With this lemma in hand we can now show that $\mathrm{tr}_n = \operatorname{tr}$.


\begin{prop}\label{negligquot} Let $q$ be a primitive $4N$-th root of unity with $q^N = i$. Then $\tilde{\mathrm{tr}}_n=\mathrm{tr}$, and $ \End_{\operatorname{Ab}(\overline{\mathcal{SE}_{N}})}(+^n)\cong\bGn$.
\end{prop}    

\begin{proof} 
The claim that $\tilde{\mathrm{tr}}_n=\mathrm{tr}$ will be proved by induction on $n$. Obviously $\mathrm{tr}$ and $\tilde{\mathrm{tr}}_2$ coincide on $\bar G_2[0]$.
For the induction step, 
we first claim that $\tilde{\mathrm{tr}}_{n+1}(c\chi)=\tilde{\mathrm{tr}}_n(c)\tilde {\mathrm{tr}}_{n+1}(\chi)$ for $c\in \bGn$ and $\chi\in \{ 1, e_n, t_n, e_nt_n\}$.
It will be enough to prove this for $c=e_{ij}^\nu$, where we assume matrix units $(e_{ij}^\nu)$ for $\bar G_n[0]$ associated to weight vectors as in Remark \ref{rem:weightbasis}. 
Fix such an $e_{ij}^\nu$, with $\nu = \lambda$ or $\nu = (\lambda, \pm)$ and corresponding weight vector $( q_k^{\eta(k)}   )_{k=1}^n$. If $i\neq j$, our claim follows from  Lemma~\ref{lem:traceFormula} and Remark \ref{rem:weightbasis}. 
Our claim for $c=e_{ii}^\nu\in \bGn_\la$ translates to
\begin{equation}\label{eq:goal}
tr(\chi)\frac{q_\la}{2^{\ell(\la)+f_\la}d^{n}}\mathrm{Tr}_{\la^{(n)}}(e_{ii}^\nu)\ =\sum_{\mu \in Y_{N, n+1}}\frac{q_\mu}{2^{\ell(\mu)+f_\mu}d^{n+1}}\mathrm{Tr}_{\mu^{(n+1)}}(e_{ii}^\nu\chi).\end{equation} 

Using Lemma~\ref{lem:traceFormula} 
we can write the right hand side as
\[    
\frac{1}{d^{n+1}}\sum_{\la\to^{(N)}\mu}\frac{q_\mu}{2^{\ell(\mu)+f_\mu}}(a(q_n^{\eta(n)},q_{n+1};\chi)+a(q_n^{\eta(n)},q_{n+1}^{-1};\chi))\mathrm{Tr}_{\la^{(n)}}(e_{ii}^\nu).              \]
Plugging this into \ref{eq:goal} and multiplying the resulting equation by $2^{f_\la+\ell(\la)}d^{n+1}/\mathrm{Tr}_{\la^{(n)}}(e_{ii}^\nu)$,
our claim \ref{eq:goal} is equivalent to
    \[tr(\chi)d q_\la\ =\ \sum_{\la\to^{(N)}\mu} (a(q_n^{\eta(n)},q_{n+1};\chi)+a(q_n^{\eta(n)},q_{n+1}^{-1};\chi))2^{\ell(\la)+f_\la-\ell(\mu)-f_\mu}q_\mu.\]
Applying Corollary \ref{cor:rationalEq} to the left hand side and multiplying both sides by $2^{\ell(\mu)-\ell(\la)}$, the last equation becomes
\begin{equation}\label{eq:itsnice}\sum_{\la\to^{(N)}\mu} (a(q_n^{\eta(n)},q_{n+1};\chi)+a(q_n^{\eta(n)},q_{n+1}^{-1};\chi))2^{f_\la-f_\mu}q_\mu\ =\  \sum_{\la\to\mu} (a(q_n^{\eta(n)},q_{n+1};\chi)+a(q_n^{\eta(n)},q_{n+1}^{-1};\chi))q_\mu.\end{equation}


This equality obviously holds if all relevant edges, i.e. the edge into $\la$ (which determines $q_n^{\nu(n)}$) and all edges $\la\to^{(N)}\mu$ (which determine the $q_{n+1}$) are regular, as then $f_\la=f_\mu$ in all cases.
If we have a restricted edge $\la\to^{(N)} \mu$ (there can only be one), Eq \ref{eq:itsnice} holds if we can show that the summand $s_\mu$ corresponding to the restricted edge $\la\to^{(N)} \mu=(\la_2,\la_3...)$
is equal to the summand $s_{\mu'} $  corresponding to the regular edge $\la\to \mu'=(N,\mu)$.
As $q_n$ only depends on $\la$, it is the same in both cases. In the generic case, we have $q'_{n+1}=x_{N-1}=-1$, which coincides with $q_{n+1}$ for the restricted edge. This proves equality of the matrix entries. Moreover, by Corollary~\ref{Nadcor} (c), we have $q_\mu=q_{\mu'}$. It is also easy to check that $f_{\mu'}+\ell(\mu')=f_{\mu}+\ell(\mu)$.
This shows $s_\mu=s_{\mu'}$, and hence also \ref{eq:goal} in this case. 

Finally, consider the case when we have a restricted edge $(N-1,\la)\to^{(N)} \la$. It follows from the definitions that the summands for the edges $\la\to\mu$ and $(N,\la)\to(N,\mu)$ agree in \ref{eq:itsnice}. Our claim will follow if we can show that the summand for the additional edge $(N,\la)\to (N+1,\la)$ on the right hand side of \ref{eq:itsnice} vanishes for our $q$. In the case of $\chi = 1$ or $\chi = e_n$ this follows immediately as $q_{(N+1,\lambda)}=0$ by Corollary~\ref{Nadcor}. In the case of $\chi = t_n$ it follows from Lemma~\ref{lem:matrixcoeff} that
\[a(q_n^{\eta(n)},q_{n+1};t_n)+a(q_n^{\eta(n)},q_{n+1}^{-1};t_n)\ =\ \frac{q-q^{-1}}{q_n^{-\eta(n)}q_{n+1}-1}-\frac{q-q^{-1}}{q_n^{\eta(n)}q_{n+1}-1}+(q-q^{-1}).\]
Let $q_n=x_{N-1}$ and $q_{n+1}=x_{N}$, and assume $q$ to be a variable for the moment. 
Using Lemma~\ref{xmdef1} the sum of the first two fractions, divided by $q-q^{-1}$, is equal to
\[\frac{x_{N-1}^{\eta(n)}-x_{N-1}^{-\eta(n)}}{x_N+x_N^{-1}-x_{N-1}-x_{N-1}^{-1}}\ =\ -\eta(n)\frac{\sqrt{\{N-1\}^2-1}}{\{N\}-\{N-1\}}\ =\ -\eta(n)\frac{\sqrt{(q^{2N}-q^{-2N})(q^{2N-2}-q^{2-2N})}}{(q-q^{-1})(q^{2N}-q^{-2N})}.\]
It follows from Theorem \ref{qdimensionsGn} and Corollary \ref{Nadcor}(a),
that the rational function $q_{(N+1,\la)}$ has a simple zero at any primitive $4N$-th root of unity,
coming from the factor $q^N+q^{-N}$.
This and our previous calculation show that the summand for the edge $(N,\la)\to(N+1,\la)$ is indeed zero for $q$ a primitive $4N$-th root of unity. The case of $\chi =e_nt_n$ is near identical 
As the matrix units span $\bGn$, we have that $\tilde{\mathrm{tr}}_{n+1}(c\chi)=\tilde {\mathrm{tr}}_{n}(c)\tilde{\mathrm{tr}}_{n+1}(\chi)$ as desired. Applying this for $\chi=1$ shows that we get the same result for $\tilde{\mathrm{tr}}_{n+1}(c) = \tilde{\mathrm{tr}}_{n}(c)$ for all $c \in \bar G_n[0]$. 

As $q_\la\neq 0$ also for $q$ a primitive $4N$-th root of unity for all $\la\in Y_N^<$ by Lemma~\ref{Nadcor}, it follows that $\tilde{\mathrm{tr}}_n$ is non-degenerate on $\bGn$ for all $n\in \N$.
Using our claim we also have that
\[  \tilde{\mathrm{tr}}_{n+1}(a\chi b) =  \tilde{\mathrm{tr}}_{n+1}(ba\chi)
    = \tilde{\mathrm{tr}}_{n}(ba) tr(\chi)
    =  \mathrm{tr}(ba)\mathrm{tr}(\chi) = \mathrm{tr}(ba \chi) = \mathrm{tr}(a \chi b) ,\]
where the third to last equality holds by induction, and the second to last equality is a straightforward exercise in skein theory. We thus have $ \tilde{\mathrm{tr}}_{n+1}=\mathrm{tr}$ on $\bar G_{n+1}[0]$ by Lemma \ref{HCSpanning}.

By definition, $\End_{\operatorname{Ab}(\overline{\mathcal{SE}_{N}})}(+^n)$ is isomorphic to $G_n[0]$ modulo the annihilator ideal of $\mathrm{tr}$. As $\mathrm{tr}=\tilde{\mathrm{tr}}_n$ is nondegenerate on $\bGn$, it follows that this annihilator ideal is equal to the kernel of the representation of $G_n[0]$ onto $\bar G_n[0]$. This proves that $\End_{\operatorname{Ab}(\overline{\mathcal{SE}_{N}})}(+^n)\cong \bGn$.
\end{proof}

\subsection{The category $\operatorname{Ab}(\overline{\mathcal{SE}_{N}})$} 

With the endomorphism algebras of $\overline{\mathcal{SE}_{N}}$ now explicitly identified, and their representation theory fully studied, we can now obtain the structure of the Cauchy completion $\operatorname{Ab}(\overline{\mathcal{SE}_{N}})$.


\begin{defn}
    Let $p_{(1^n)}$ be the minimal projection in the Hecke algebra corresponding to the Young diagram $(1^n)$, and let $\wedge_n :=G_n[0]p_{(1^n)}$ be the corresponding irreducible representation of $G_n[0]$.
\end{defn}

Note that the simple summand of $+^N$ in $\operatorname{Ab}(\overline{\mathcal{SE}_{N}})$ corresponding to $p_{(1^N)}$ is isomorphic to the trivial object, with isomorphism given by the additional generator defined in Definition~\ref{def:ext}. The main goal of the next proposition is to determine which of the $W_{\lambda,\pm}^{(N)}$ is isomorphic to $\wedge_N$.

\begin{prop}\label{prop:Ntensor} Let $q$ be a primitive $4N$-th root of unity. 

\begin{enumerate}[(a)]
    \item $\wedge_{n}=\operatorname{Cliff}(n)[0]p_{(1^n)}$ for $n\leq N$, and it is isomorphic to $W^{(n)}_{[n]}$ for $n<N$.

    \item The idempotent $p_{(1^{N-1})}$ acts as a rank 1 idempotent in
the modules $W^{(N)}_{[N-1,1],\pm}$ and $W^{(N)}_{\emptyset,\pm}$, and as zero in all the other simple $ G_N[0]$ modules constructed in Section \ref{decompositionsection}.

\item For $N > 3$ we may assume that $
\wedge_N \cong W_{\emptyset,+}^{(N)}$. 

\end{enumerate}

\end{prop}

\begin{proof} The first part of Statement (a) follows from
$G_n[0]p_{(1^n)}=\operatorname{Cliff}(n)[0]H_np_{(1^n)}=\operatorname{Cliff}(n)[0]p_{(1^n)}$, as $H_np_{(1^n)}=\mathbb{C}p_{(1^n)}$.  We prove the second statement in (a) by induction on $n$, with $n=2$ being trivially true. Note that $e_{n-1}\mathrm{Cliff}(n-1)[0]p_{(1^n)}$ is a $G_{n-1}[0]$ module by the relations $t_{n-2}e_{n-1} = e_{n-1}e_{n-2}t_{n-2}$ and $e_{n-2}e_{n-1} = -e_{n-1}e_{n-2}$. We have the isomorphisms of $G_{n-1}[0]$ modules
\[\wedge_{n} \cong \mathrm{Cliff}(n-1)[0]p_{(n)}\oplus e_{n-1}\mathrm{Cliff}(n-1)[0]p_{(1^n)}\cong \wedge_{n-1}\oplus\tilde \wedge_{n-1}.\]
In the first isomorphism we have used the already proved identity and $\operatorname{Cliff}(n)[0]=\operatorname{Cliff}(n-1)[0]\oplus e_{n-1}\operatorname{Cliff}(n-1)[0]$. In the second isomorphism we use the isomorphism of $G_{n-1}[0]$ modules $\mathrm{Cliff}(n-1)[0]p_{(1^n)}\cong \mathrm{Cliff}(n-1)[0]p_{(1^{n-1})}$, along with Definition~\ref{def:Utilde} and Lemma~\ref{q1lemma} for the second summand.
By induction assumption, we have the isomorphism $\wedge_{n-1}\cong W_{(n-1)}^{(n-1)}$. The modules $W_{(n-1)}^{(n-1)}$ and $\tilde W_{(n-1)}^{(n-1)}$ have the same character, as conjugation by $e_{n}$ fixes the $J_k$, $k<n$, and it does not change the eigenvalues and their multiplicities of $J_n$. Hence the modules $W_{[n]}$ and $\tilde W_{[n]}$ must be isomorphic. It follows from Lemma~\ref{restrictlemma}(c) and Proposition~\ref{prop:quotient} that $W_{(n)}^{(n)}$ is the only $G_n[0]$ module whose restriction to $G_{n-1}[0]$ coincides with the one of $\wedge_{n}$. 


Part (b) is a direct consequence of the restriction rules in Lemma \ref{restrictlemma}.

For (c) we recall that $p_{(1^N)}$ is minimal in $G_N[0]$. Thus $\wedge_N = \bar G_N[0]p_{(1^N)} $ is an irreducible $\bar G_N[0]$ module, and so isomorphic to one of the modules stated in Proposition~\ref{prop:quotient}. As $p_{(1^N)} = p_{(1^{N-1})}p_{(1^N)}$ we get from (b) that $\bar G_N[0]p_{(1^N)}$ is isomorphic to one of $W^{(N)}_{[N-1,1],\pm}$ or $W^{(N)}_{\emptyset,\pm}$. By (a), we have $G_N[0] p_{(1^N)}   =   \operatorname{Cliff}(N)[0]p_{(1^N)} $. In particular, this implies $\dim( \bar G_N[0] p_{(1^N)}  ) \leq 2^{N-1}$. We have that $\dim(W^{(N)}_{[N-1,1],\pm}) = (N-1)2^{N-2}$ and $\dim(W^{(N)}_{\emptyset,\pm}) =2^{N-2}$ and so $\wedge_N$ is isomorphic to either $W^{(N)}_{\emptyset,+}$ or $W^{(N)}_{\emptyset,-}$. By using up the free symmetry between these two modules described in Remark~\ref{rem:signs} we may assume $\wedge_N \cong W^{(N)}_{\emptyset,+}$.
\end{proof}

\begin{thm}\label{fusionCat}
Let $N\in \mathbb{N}_{\geq 4}$, the isomorphism classes of simple objects in $\operatorname{Ab}(\overline{\mathcal{SE}_{N}})$ are parameterised by the set
\[    \{  \lambda : \lambda\in Y^<_N :\:\ell(\lambda) \text{ odd}\} \cup  \{  (\lambda, \pm) :  \lambda\in Y^<_N :\:\ell(\lambda) \text{ even}\} .    \]
In particular, the labels $([0],+)$, $([0],-)$, $[1]$ and $[N-1] $ correspond to the trivial object, the object $g$, and the objects $+$ and its dual object $-$ in $\operatorname{Ab}(\overline{\mathcal{SE}_{N}})$. 
Furthermore, the decomposition of the tensor product of a simple object with $+$ is given by the graph $\Ga(N)[0]$, see Definition~\ref{Gndef}. 
Finally, the dimension of the objects $(\la)$ resp $(\la,\pm)$ are equal to the evaluations of the rational functions 
$q_\la/2^{\lfloor \ell(\la)/2\rfloor}$ at $q = e^{2\pi i \frac{1}{4N}}$, with $q_\la$ as in Theorem \ref{qdimensionsGn}.
\end{thm}

\begin{proof}
It is a well known fact that every simple object in $\overline{\operatorname{Rep}(U_q(\mathfrak{sl}_N))}$ is isomorphic to a subobject of $V_\square^{\otimes n}$ for some $n\in \mathbb{N}$. Recall that the free module functor $\mathcal{F}_A: \overline{\operatorname{Rep}(U_q(\mathfrak{sl}_N))} \to \overline{\operatorname{Rep}(U_q(\mathfrak{sl}_N))}_A\simeq \operatorname{Ab}(\overline{\mathcal{SE}_{N}})$ is dominant, and so every simple object of $\operatorname{Ab}(\overline{\mathcal{SE}_{N}})$ is isomorphic to a simple summand of $+^n$.

It follows from Proposition \ref{negligquot} that the simple objects in $+^n$, viewed as an object in $\operatorname{Ab}(\overline{\mathcal{SE}_{N}})$, are labeled by $\la$ or $(\la,\pm)$, depending on parity of $\ell(\la)$, with  $\la\in Y^<_{N,n}$. To show that this labeling is independent of $n$, we must produce an isomorphism between the object labeled by $\nu$ in $+^{N+n}$ and the object labeled by $\nu$ in $+^{n}$. This isomorphism will be constructed using the additional generator defined in Definition~\ref{def:ext}. 

Let $\La_0$ be the unique tableau of length $N$ for the Young diagram $(N)$. By Proposition~\ref{prop:Ntensor} c) we have that $p_{(1^N)}$ acts non-trivially on $W_{\emptyset,+}^{(N)}$, and hence on $V_{\emptyset}^{(N)} = \operatorname{Cliff}(N)\psi_{\Lambda_0}$ 
. As $\operatorname{Cliff}(N)$ has a basis of invertible elements (products of generators), there exists an invertible $v\in \operatorname{Cliff}(N)$ such that $p_{(1^N)}v\psi_{\La_0}\neq 0$. Hence there exists an idempotent $p:=v^{-1}p_{(1^N)}v$
such that $p\psi_{\La_0}\neq 0$. 
It follows from our definitions that this can be generalized to any path $\Ga$ of length $>N$, that is: 
\begin{equation*}\label{reductmodule}
p \psi_\Ga\neq 0\quad \Leftrightarrow\quad \Gamma_{|[1,N]} = \La_0.
\end{equation*} 
We claim that the isomorphism between the simples labeled by $\nu$ in $+^{n}$ and $+^{N+n}$ is given by applying the distinguished morphism $\Hom_{\mathcal{SE}_N}(+^N \to \mathbf{1})$ given by the composition of conjugation by the $v\in \operatorname{Cliff}(N)$ above, with the defining morphism in $\Hom_{\mathcal{SE}_N}(+^N \to \mathbf{1})$ from Definition~\ref{def:ext}. We denote this distinguished morphism $\hat{v}$. 

We will begin with the slightly simpler case of $\nu = \la$ with $|\la|$ odd for now.  Note that there is an isomorphism of algebras $\rho: p\bar G_{n+N}[0] p\cong \bGn$ given by composition with $\hat{v}$. To show that $\hat{v}$ gives an isomorphism between the two objects labeled by $\la$, it suffices to show that $\rho$ maps the cut-down representation $pW_\la^{N+n}$ to $W_\la^{n}$. We will do this by finding a weight of the representation $\rho(pW_\la^{N+n})$. Note that the elements $\rho(J_{N+k})\in \bar G_{N+n}[0]$ satisfy the same recurrence relations as the elements $J_k\in \bar G_{n}[0]$. Hence we have that $\rho(J_{N+k}) = J_k$. Let $\La \in \operatorname{Paths}(\lambda, N+n)$ such that $\Lambda(N) = \emptyset$. We then have that $({q_k^\Lambda}^{-1})_{k=1}^{N+k}$ is the weight of $ \psi_\La$ in $V_{\lambda}^{N+n}$. Hence
\[
    p J_{N+k} p \psi_\La = {q_{N+k}^\La}^{-1} p   \psi_\La 
\]
using that $J_{N+k}$ commutes with the $\bar{G}_N[0]$ subalgebra. As $p   \psi_\La \neq 0$ we get that $({q_{k}^{\La}}^{-1})_{k= N+1}^{N+n}$ is a weight of $pV_\la^{N+n}$.

As $\La(N) = \emptyset$, we can define 
\[\hat{\La} : \emptyset=\La(N) \to\La(N+1)\to \cdots \to \La(N+n)=\lambda\in \operatorname{Paths}(\lambda,n). \]
By construction we have $q_{N+k}^\La = q_{k}^{\hat{\La}}$. It follows that $({q_{k}^{\hat{\La}}}^{-1})_{k=1}^{n} = ({q_{k}^{\La}}^{-1})_{k=N+1}^{N+n}$ is a weight for $V_\lambda^{(n)}$ with weight vector $ \psi_{\hat{\La}}$. Thus $({q_{k}^{\La}}^{-1})_{k=N+1}^{N+n}$ is a weight for $W_\lambda^{(n)}$ by Lemma~\ref{Uladiff} (d). As the irreducible $\bar{G}_n[0]$ representations $\rho(p_{(1^N)}W_\la^{N+n})$ and $W_\la^{n}$ share a weight, they are isomorphic by Lemma~\ref{lem:weights}, Lemma~\ref{Uladiff} (d), and Proposition~\ref{prop:quotient}. 

In the case of $\nu = (\la, \pm)$ with $|\la|$ even, the above weight argument shows that $W_{\la,\pm}^{(n)} \cong \rho( p W_{\la,\varepsilon}^{N+n})$ for some sign $\varepsilon$. We inductively break the free symmetry between the modules $W_{\la,\pm}^{N+n}$ described in Remark~\ref{rem:signs} so that $W_{\la,\pm}^{(n)} \cong \rho( p W_{\la,\varepsilon}^{N+n})$ holds for all even $\lambda$ and all $n$.

We thus have that the two simple summands of $+^{N+n}$ and $+^n$ each labeled by $\nu$ are isomorphic in $\operatorname{Ab}(\overline{\mathcal{SE}_{N}})$. It follows that all simple objects labelled by $\nu$ in $\operatorname{Ab}(\overline{\mathcal{SE}_{N}})$ are isomorphic. Now consider two simples labelled by distinct $\nu_1 \neq \nu_2$. If $|\nu_1|  \not\equiv |\nu_2| \pmod N$, then there are no non-zero morphisms between $+^{|\nu_1|}$ and $+^{|\nu_2|}$. Hence the objects corresponding to the labels $\nu_1$ and $\nu_2$ are non-isomorphic. If $|\nu_1|  \equiv  |\nu_2| \pmod N$, then we have summands of both $\nu_1$ and $\nu_2$ in some suitably large $+^n$ (say $n = \max(|\nu_1|,|\nu_2|) 
  $ to be explicit). From earlier in the proof, we have that the $\nu_1$ summand corresponds to the representation $W^{(n)}_{\nu_1}$ and the $\lambda$ summand corresponds to the representation $W^{(n)}_{\nu_2}$. These representations are non-isomorphic by Proposition~\ref{prop:quotient}, and so the objects labeled by $\nu_1$ and $\nu_2$ are non-isomorphic. All together, we have that the simple objects of $\operatorname{Ab}(\overline{\mathcal{SE}_{N}})$ are labeled by the set $V(N)[0]$ as claimed.


The claim regarding the dimensions of the simple objects follows from the explicit formula for the normalized categorical trace $\operatorname{tr}$ on $\operatorname{End}_{\overline{\mathcal{SE}_N}}(+^n)$ obtained in Subsection~\ref{sub:trace}.


Observe that the multiplicity of the object $(\ga)\in +^n$ in the tensor product $(\nu)\otimes +$ is equal to the rank of the minimal idempotent $p_\nu$ in the simple component $\End(+^{n+1})_\ga$ by Corollary~\ref{cor:RestrictionFusion}. This proves the statement about the decomposition of tensor products, by Lemma~\ref{restrictlemma}. 

By Proposition~\ref{prop:Ntensor} (c), the label $(\emptyset,+)$ corresponds to the projection $p_{(1^N)}\in \bar G_N[0]$. The determinant map of Definition~\ref{def:ext} along with the relation (Pair) then shows that $p_{(1^N)}$ projects onto the tensor unit of $\overline{\mathcal{SE}_N}$. Hence $(\emptyset,+)$ is the isomorphism class of the unit. From Corollary~\ref{prop:Ntensor} and the earlier results of this proof, we compute
 $$([N-1])\otimes ([1])\ \cong\ (\emptyset,+)\oplus (\emptyset,-)\oplus ([N-1,1],+)\oplus ([N-1,1],-),$$
  This shows that $([N-1])$ is the dual object of $([1])$. Further, by Definition~\ref{def:Eq}, we have that $g$ is a summand of $+ \otimes (+)^* \cong (+)^* \otimes +\cong ([N-1])\otimes ([1])$. Supposing $N\geq 4$, a dimension count of the simples above shows that $g$ must correspond to $(\emptyset,-)$.
 \end{proof}

 \begin{cor}\label{cor:SEN} 
 Let $N\in \mathbb{N}_{\geq 4}$.
 
 \begin{enumerate}[(a)]
     \item  The tensor product of a simple object $(\nu)$ in $\operatorname{Ab}(\overline{\mathcal{SE}_{N}})$ with $-\cong [N-1]$ is isomorphic to the direct sum of objects $(\gamma)$ whose multiplicities are given by the number of edges from $\gamma$ to $\nu$ in the graph $\Galnk$. 

     \item We have $(\la)\otimes g\cong (\la)$, $(\la,+)\otimes g\cong (\la,-)$ and $(\la,-)\otimes g\cong (\la,+)$, where $g\cong ([0],-)$.

     \item The multiplicity of a simple object $\nu$ in $+^r-^s\cong [1]^r\otimes [N-1]^s$ is equal to the number of paths from $[0]$ to $\nu$ of length $r+s$ where we first move $r$ times in direction of the edges of $\Gnoo$ and then $s$ times in the opposite direction. In particular, the strict Young diagram corresponding to $\nu$ has at most $r$ rows of length $<N/4$, and at most $s$ rows of length $>3N/4$ if $r+s<N/4$, and no rows of other lengths.
 \end{enumerate}

 \end{cor}

 \begin{proof} Part (a) follows from Frobenius reciprocity. Let $\ga_n$ be the automorphism of $\Gno$ given by conjugation by $v_n$. It was shown in Lemma \ref{irred},(c) that $\ga_n(G_n[0]_{(\la,+)})=G_n[0]_{(\la,-)}$. 
 We then have
 $$p_{(\la,+)}\otimes p_{(N,-)}=\ga_{n+N}(p_{(\la,+)}\otimes p_{(N,+)})=\ga_{n+N}(p_{(\la,+)})\in G_{n+N}[0]_{(\la,-)},$$
 which implies $(\la,+)\otimes g=(\la,-)$. The other claims in (b) are proved similarly.
 Part (c) follows from (a) by induction on $r+s$.
 \end{proof}
\subsection{The category $\operatorname{Ab}(\E_q)$ for generic $q$} We can now also give an explicit description of the simple objects of $\operatorname{Ab}(\E_q)$ for generic $q$, i.e. for all but countably many values of $q$. Our first step for this is to find a new labeling of the objects of $\operatorname{Ab}(\overline{\mathcal{SE}_N})$ which is stable as $N\to \infty$. Let 
$$Y(N)\ =\ \{ \nu\in Y_N^<,\ \nu_i>3N/4\ {\rm or}\ \nu_i<N/4\}. $$
We define the map
\begin{equation}\label{PhiN:def}
\Phi_N: \nu\in Y(N)\ \mapsto\ (\la,\mu)\ =\ ([\nu_j,\nu_{j+1},\ ...\ \nu_{\ell(\nu)}], [N-\nu_{j-1},\ ..., N-\nu_1],
\end{equation}
where $j$ is the smallest index such that $\nu_j < N/4$. We define the set $V\E[0]$ by
\begin{equation}\label{VEN:def}
V\E[0]\ =\ \{ (\la,\mu,0),\ \ell(\la)+\ell(\mu)\ {\rm odd}\}\ \cup\ \{ (\la,\mu,\pm),\ \ell(\la)+\ell(\mu)\ {\rm even}\}.
\end{equation}
Observe that $V(N)[0]$ can be similarly described as the set $\{ (\nu,\ep)\}$, where $\nu$ ranges over all strict Young diagrams $\nu$ with $\nu_1<N$ and $\ep=0$ if $\ell(\nu)$ is odd, and $\ep\in \{\pm\}$ if $\ell(\nu)$ is even. We may often just write $(\nu)$ for $(\nu,0)$ for convenience.

\begin{lem}\label{genericprep}
Let $N\in \mathbb{N}_{\geq 3}$. 

\begin{enumerate}[(a)]
    \item  The map $\Phi_N$  induces an injective map  $(\nu,\ep)\in V(N)[0]\mapsto (\Phi_N(\nu),\ep)\in V\E[0]$, and any element in $V\E[0]$ is in the image of this map for large enough $N$.

    \item Using the map $\Phi_N$, the tensor product rules in Theorem \ref{fusionCat} can be expressed as follows:
$$\Phi_N^{-1}(\la,\mu,\ep)\otimes \Phi_N^{-1}([1],\emptyset)\cong \bigoplus_{(\tilde\la,\mu,\ep)}m_{(\tilde\la,\mu,\ep)}\Phi_N^{-1}(\tilde\la,\mu,\ep)\ \oplus\ \bigoplus_{(\la,\tilde\mu,\ep)}m_{(\la,\tilde\mu,\ep)}\Phi_N^{-1}(\la,\tilde\mu,\ep),$$
where $\tilde\la$ ranges over all strict diagrams obtained from $\la$ by adding a box, and $\tilde\mu$ ranges over all strict diagrams obtained from $\mu$ by removing a box. If $\ep$ can be $\pm$, both options appear in the direct sum. The multiplicities $m_{(\tilde\la,\mu,\ep)}$ and $m_{(\la,\tilde\mu,\ep)}$ are equal to 1 except when both $\ell(\la)+\ell(\mu)$ and  $\ell(\tilde\la)+\ell(\mu)$ or  $\ell(\la)+\ell(\tilde\mu)$ are even, when it is equal to 2.
The tensor product with $[N-1])=\Phi_N^{-1}(\emptyset,[1])$ can be expressed in a similar way, where now $\tilde\la$ is obtained by removing a box from $\la$ and $\tilde\mu$ is obtained by removing a box from $\mu$.

\item Let $\mu\in Y(N)$ such that $\mu_1<N/4$. Then $\Phi_N^{-1}(\emptyset,\mu,\ep)=(\mu,\ep)^*$, the dual representation of $(\mu,\ep)$ in 
 $\operatorname{Ab}(\overline{\mathcal{SE}_{N}})$.

 \item  Let $q$ be a primitive $4N$-th root of unity such that $q^N=i$, and let $(\la,\mu)=\Phi_N(\nu)$ for $\nu\in Y(N)$. Let $q_\nu$ be as defined in Theorem \ref{qdimensionsGn}. Then 
we have
$$q_\nu\ =\ q_{(\la,\mu)}\ :=\ q_\la q_\mu \prod_{(r,s)}\frac{q^{\la_r+\mu_s}+q^{-\la_r-\mu_s}}{q^{\la_r-\mu_s}+q^{\mu_s-\la_r}},\quad {\rm where}\ 1\leq r\leq \ell(\la),\ 1\leq s\leq \ell(\mu).$$
\end{enumerate}


\end{lem}

\begin{proof} Part (a) follows from the definitions. For part (b), using the notations of \ref{PhiN:def}, it suffices to observe that adding a box to $\nu$ in the first $j-1$ rows has the effect of removing a box from $\mu$, and adding a box in the other rows of $\nu$ results in adding a box to $\la$. One can similarly describe the decomposition of the tensor product of $\Phi_N^{-1}(\la,\mu,\ep)$ with $[N-1]$, using Corollary \ref{cor:SEN},(a). We show part (c) by induction on $k=\ell(\mu)$ and $m=\mu_k$. The statement has already been proved for $k=1=m$ in Theorem \ref{fusionCat}. The induction step will follow from the following

$Claim:$ Assume that Statement (c) is true for all $\mu$ with $\ell(\mu)\leq k$ and $\mu_k\leq m$. Then it is also true for any $\mu$ of the form $\mu=[\mu_1,\ ..., \mu_{k-1}, m+1]$ and $\mu=[\mu_1,\ ...,\ \mu_k,1]$.

We prove this first for $k$ even.  Using  Corollary \ref{cor:SEN},(a), we see that the tensor product  $((\mu,+)\otimes [1])^*\cong ([N-\mu_k,\ ...,\ N-\mu_1],+)\otimes [N-1]$ is isomorphic to
$$ [N-1,N-\mu_k,\ ...\ N-\mu_1]\ \oplus\ \bigoplus_{j=1}^k([N-\mu_k,\ ...\ N-\mu_j+1,\ ... N-\mu_1],\pm).$$
If we compare this with $((\mu,+)\otimes [1])^*$ obtained by calculating $(\mu,+)\otimes [1]$ first and then applying $^*$, we obtain (using the assumption for diagrams with $\ell(\mu)<k$ and $\mu_k\leq m$)
$$([N-\mu_k,\ ...\ N-\mu_j+1,\ ... N-\mu_1],\pm)\oplus [N-1,N-\mu_k,\ ...\ N-\mu_1]\ \cong\ ([\mu_1,\ ...,\ m+1],\pm)^* \oplus ([\mu_1,\ ...,\ \mu_k,1])^*.$$
As $(\ga)\otimes g\cong (\ga)$ if and only if $(\ga)^*\otimes g\cong (\ga)^*$, we deduce that
$([N-\mu_k,\ ...\ N-\mu_j+1,\ ... N-\mu_1],\pm)\cong ([\mu_1,\ ...,\ m+1],\pm)^*$ and $([N-1,N-\mu_k,\ ...\ N-\mu_1])^*\cong ([\mu_1,\ ...,\ \mu_k,1])^*$. The proof for $k=\ell(\mu)$ odd goes similarly.

Using $q_{N-m}=q_m$, $[2N-k]=[k]$ and  Lemma \ref{Nadcor},(b), we obtain
\begin{align}
    q_{\nu}\ =&\ \prod_r q_{\la_r}\prod_{i<j} \frac{[\la_i-\la_j]}{[\la_i+\la_j]} 
    \prod_s q_{\mu_s}\prod_{i<j}\frac{[\mu_i-\mu_j]}{[\mu_i+\mu_j]}\ \prod_{(r,s)}\frac{[N-\la_r-\mu_s]}{[N-\mu_s+\la_r]}\cr
    =&\ q_\la q_\mu\ \prod_{(r,s)}\frac{q^{\la_r+\mu_s}+q^{-\la_r-\mu_s}}{q^{\la_r-\mu_s}+q^{-\la_r+\mu_s}},\cr
\end{align}
where $1\leq r\leq \ell(\la)$ and $1\leq s\leq \ell(\mu)$.
\end{proof}


We now wish to determine the structure of $\operatorname{Ab}(\mathcal{E}_q)$ for values of $q$ where this category is semisimple. By \cite[Theorem 3.3]{Tuba} this is equivalent to having the endomorphism algebras of $\mathcal{E}_q$ being semisimple. The statements below follow from the fact that the discriminant of   $\End_{\mathcal{E}_q}(+^r-^s)$ is a rational function in $q$. 

\begin{prop}\label{prop:bexistence} Assume the algebra $\End_{\mathcal{E}_q}(+^r-^s)$ is semisimple for at least one value of $q$. Then we have
\begin{enumerate}[(a)]
    \item The algebra $\End_{\mathcal{E}_q}(+^r-^s)$ is semisimple for all but a finite number of complex values of $q$.

    \item  The minimal central idempotents of $\End_{\mathcal{E}_q}(+^r-^s)$ are linear combinations of basis elements whose coefficients are specialisations of elements of a finite algebraic extension of $\mathbb{C}(q)$. In particular, they are well-defined for all but finitely many values of $q$.
\end{enumerate}
\end{prop}

Note that the set of excluded $q$ depends on $r$ and $s$ and will grow as $r$ and $s$ increase. 


The following theorem will give a labeling of the simple objects of the category $\operatorname{Ab}(\E_q)$ when it is semisimple. As for the category $\operatorname{Ab}(\overline{\mathcal{SE}_{N}})$, this will be fixed up to the choices of $+$ or $-$ in the labels, see Remark \ref{rem:signs}. This will not affect the tensor product rules stated there, which are symmetric under sign changes.
\begin{thm}\label{Eqobjects}
The simple objects of the category $\operatorname{Ab}(\E_q)$ for generic $q$ are labeled by the set $V\E[0]$, see \ref{VEN:def}. The dimension of $(\la,\mu,\ep)$ is given by $q_{(\la,\mu)}/2^{\lfloor (\ell(\la)+\ell(\mu))/2\rfloor}$, with $q_{(\la,\mu)}$ as defined in Lemma \ref{genericprep}.

The decomposition of the tensor product of $(\la,\mu,\ep)$ with $([1],\emptyset)$ or $(\emptyset,[1])$ is given by Lemma \ref{genericprep},(b) by using the labels for $V\E[0]$, i.e we replace all labels $\Phi_N^{-1}(\ga)$ there by $\ga$. 
\end{thm}

\begin{proof} 
We first observe that conjugation by a suitable braid element in $\mathcal{E}_q$ gives isomorphisms between minimal projections in $\End_{\mathcal{E}_q}(w)$ where $w\in \{+,-\}^n$, and minimal projections in $\End_{\mathcal{E}_q}(+^r-^s)$ where $r$ and $s$ are the number of $+$'s and $-$'s in $w$ respectively. Hence we only need to consider the algebras $\End_{\mathcal{E}_q}(+^r-^s)$ to obtain all simple objects of $\operatorname{Ab}(\mathcal{E}_q)$ up to isomorphism.

 It was shown in \cite[Corollary 5.9]{ConformalA} that $\dim\End_{\E_q}(+^r-^s)=
\dim\End_{\overline{\mathcal{SE}_N}}(+^r-^s)$ whenever $r+s<N/4$, where the dimension for $\dim\End_{\E_q}(+^r-^s)$ does not depend on $q$ as long as it is not a root of unity. Thus $\End_{\E_q}(+^r-^s)$ is semisimple for $q=e^{2\pi i/4N}$ for all $N>4(r+s)$, and so
it is semisimple for all but finitely many values of $q$ by Proposition~\ref{prop:bexistence}. In particular, it has the same decomposition into a direct sum of full matrix rings for all but (not necessarily the same) finitely many values of $q$, and the map $\Phi_N$ defines a labeling of these simple components by elements in $V\E[0]$ via specialisation of $q$. Let us label these simple components by $M_{\lambda, \mu, \epsilon}^{N,r,s}$, and the corresponding simple object of $\mathcal{E}_q$ by $(\lambda, \mu, \epsilon)_{N,r,s}$. \textit{A-priori} different values of $N$ and ${r,s}$ could correspond to different simples of $\mathcal{E}_q$.

Consider the simple components $M_{\la,\mu,\ep}^{N,r_1,s_1}$ and $M_{\la,\mu,\ep}^{N,r_2,s_2}$ of $\End_{\mathcal{E}_q}( +^{r_1}-^{s_1}  )$ and $\End_{\mathcal{E}_q}( +^{r_2}-^{s_2}  )$ respectively. We will show that the corresponding simple objects of $\mathcal{E}_q$ are isomorphic. If $r_2-r_1=k>0$, we also have $s_2-s_1=k$, as $r_i-s_i=|\la|-|\mu|$ for $i=1,2$. Hence it suffices to prove this with $k=1$, $r_1=r$ and $s_1=s$. Let $z_{\la,\mu,\ep}^{(r,s)}$ and $z_{\la,\mu,\ep}^{(r+1,s+1)}$ be the central idempotent for the simple components $M_{\la,\mu,\ep}^{N,r,s}$ and $M_{\la,\mu,\ep}^{N,r+1,s+1}$ respectively. It follows from \cite[Corollary 5.9]{ConformalA} and Proposition~\ref{prop:bexistence} that $z_{\la,\mu,\ep}^{(r,s)}$ is a linear combination of basis elements of $\End_{\mathcal{E}_q}(+^r-^s)$ with coefficients specialisations of an element of an algebraic extension of $\C(q)$. In particular, it is well-defined for all but finitely many values in $q$. The same statement also applies for
$p$, the projection onto the trivial object in $\End_{\mathcal{E}_q}(+-)$, and hence also for $(p\otimes z_{\la,\mu,\ep}^{(r,s)})$. Obviously,  $(p\otimes z_{\la,\mu,\ep}^{(r,s)})\cong  z_{\la,\mu,\ep}^{(r,s)}$  as objects in $\operatorname{Ab}(\E_q)$. It follows from Theorem~\ref{fusionCat} when $q$ is specialised to a primitive $4N$-th root of unity that $(\lambda, \mu, \varepsilon)_{N,r,s}\cong(\lambda, \mu, \varepsilon)_{N,r+1,s+1}$ when $r+s+2<N/4$. Hence
$\beta(p\otimes z_{\la,\mu,\ep}^{(r,s)})\beta^{-1}z_{\la,\mu,\ep}^{(r+1,s+1)}\neq 0$ for $q$ a primitive $4N$-th root of unity, where  $\beta:\End(+-+^r-^s)\to \End(+^{r+1}-^{s+1})$ is the isomorphism constructed via the half-braiding of $\mathcal{E}_q$. This implies the same statement for all $q$ for which these elements are well-defined.
It follows that $(\la,\mu,\ep)_{N,r,s}\cong(\la,\mu,\ep)_{N,r+1,s+1}$ whenever $r+s+2<N/4$. Hence we may drop the $r,s$ notation in the labeling of the simple objects by picking our favorite representative.

To show our labeling of simple objects is independent of $N$, we will require certain tensor product rules for these simples. The decomposition of the tensor product of an object $(\nu,\ep)$ in  $\operatorname{Ab}(\overline{\mathcal{SE}_{N}})$ with $[1]$ or $[N-1]$ has been described in Theorem \ref{fusionCat} and Corollary \ref{cor:SEN}. In particular, if we tensor with $[N-1]$, we either add a row with $N-1$ boxes to $\nu$, or we remove a box from it. It follows by induction on $r$ and $s$ that any diagram $\la$ in the label of a simple component of $[1]^r\otimes[N-1]^s$ has only rows with either $\geq N-s$ boxes or with $\leq r$ boxes. As $r+s<N/4$, it follows that each such diagram is in the domain $Y(N)$ of $\Phi_N$. 
 It follows from  Lemma \ref{genericprep},(b) that $\Phi_N$ translates the tensor product rules for $+$ and $-$ in  $\operatorname{Ab}(\overline{\mathcal{SE}_{N}})$ to the tensor product rules claimed in this theorem, for large enough $N$. 

We now show that whenever $r+s<N/4$ and $r+s<M/4$ we have either $(\lambda, \mu, \varepsilon)_N \cong (\lambda, \mu, \varepsilon)_M$ or $(\lambda, \mu, \varepsilon)_N \cong (\lambda, \mu, -\varepsilon)_M$. The unresolved sign ambiguity does not affect the statement of the theorem. To avoid the ambiguity with using $+$ and $-$ in the labels we will use the notation $(\la,\mu,\pm)_N$ for the direct sum $(\la,\mu,+)_N\oplus (\la,\mu,-)_N$.
We say that the pair $(\la,\mu)$ is stable, if $(\la,\mu,\tilde\ep)_M\cong(\la,\mu,\tilde\ep)_N$, with $\tilde\ep=\pm$ or $\tilde\ep=0$, for all $M>N> 4r+4s$. As our representations of $G_n[0]$ are semisimple and depend continuously on $q$ except for finitely many values, it follows that $\{\la,\emptyset\}$ is stable for all strict Young diagrams $\la$.
As $(\emptyset,\mu,\tilde\ep)_N$ is the dual object of $(\mu,\emptyset,\tilde\ep)_N$, the pair $\{ \emptyset,\mu\}$ is also stable, see Lemma \ref{genericprep},(c). 
We prove the result for general $\{\la,\mu\}$ by induction on $k=\ell(\la)$ and $\la_k$, essentially by the same strategy which was used for Lemma \ref{genericprep},(c). Observe that we have already shown that $\{\la,\mu\}$ is stable if $k=1$ and $\la_1=0$. The induction step now follows from the following:

\textit{Claim:} Assume all pairs $\{\la,\mu\}$ are stable if $\ell(\la)\leq k$ and $\la_k\leq m$. Then also all pairs $\{[\la_1,\ ...\ \la_{k-1},m+1],\mu\}$ and  $\{[\la_1,\ ...\ \la_{k-1},m,1],\mu\}$ are stable, whenever the diagrams are strict.

To prove the claim when $\ell(\la)+\ell(\mu)$ is even, we compare the decomposition of the tensor products $(\la,\mu,+)_\ell\otimes ([1],\emptyset)_\ell$ with $\ell = N$ and with $\ell = M$. By induction assumption, we already know that all but three labels are stable. Hence we have 
$$([\la_1,\ ..., m+1],\mu,\pm)_N\oplus ([\la_1,\ ..., m,1],\mu)_N \ \cong\ 
([\la_1,\ ..., m+1],\mu,\pm)_M\oplus ([\la_1,\ ..., m,1],\mu)_M.$$
This gives six possibilities for isomorphisms between the simple objects on the left and right sides (of which only two possibilities are stable).

We take our representative of $(\emptyset, \emptyset, -)_N$ to be a simple summand of $\operatorname{End}_{\mathcal{E}_q}(+-)$. From our explicit basis for $\operatorname{End}_{\mathcal{E}_q}(+-)$, we can directly compute all minimal projections. By taking the categorical dimension of these minimal projections, we see that when $(\emptyset, \emptyset, -)$ must correspond to the generator $\att{splittingEq}{.1}$. This implies that $(\emptyset, \emptyset, -)_N \cong (\emptyset, \emptyset, -)_M$ for all $3< N< M$ and that this label corresponds to the object $g$. 

From Corollary~\ref{cor:SEN},(b) we know that $g\otimes (\lambda, \mu, +)_\ell \cong (\lambda, \mu, -)_\ell$ and that $g\otimes (\lambda, \mu)_\ell \cong (\lambda, \mu)_\ell$.
It follows that only the two stable possibilities above are compatible with this. Hence  $([\la_1,\ ..., m+1],\mu,\pm)_N\cong ([\la_1,\ ..., m+1],\mu,\pm)_M$ and $ ([\la_1,\ ..., m,1],\mu)_N\cong ([\la_1,\ ..., m,1],\mu)_M$. The proof for $\ell(\la)+\ell(\mu)$ odd goes similarly, where we now consider the tensor product $(\la,\mu)_\ell\otimes ([1],\emptyset)_\ell$. 
 
All together we have shown that the simple objects of $\operatorname{Ab}(\mathcal{E}_q)$ are parameterised by the set $V\mathcal{E}[0]$, and that the fusion rules are as claimed in the statement of the theorem.


As the dimension of a simple object is a rational function in $q$, it is determined if we know it for infinitely many values. The claim hence follows from Theorem \ref{fusionCat} and Lemma \ref{genericprep}(d).
\end{proof}

\begin{remark}
    The endomorphism algebras of $+^r-^s$ are a slight variation on  the quantum walled Brauer-Clifford algebras defined in \cite{quantum-walled}. Thus this description of simples in $\E_q$ shows that these new Brauer-Clifford algebras are semisimple for generic $q$ and describes their simple components. By contrast, the original Brauer-Clifford algebras defined in \cite{quantum-walled} are not semisimple.
\end{remark}

\bibliography{Cells}
\bibliographystyle{alpha}
\end{document}